\documentclass[11pt]{amsart}

\usepackage{amsfonts,amsmath,amssymb,amsthm}
\usepackage{latexsym}
\usepackage[ansinew]{inputenc}
\usepackage{amsfonts}
\usepackage[dvips]{graphicx}
\usepackage{color}
\usepackage[pagewise,mathlines]{lineno}
\newtheorem{theorem}{Theorem}[section]
\newtheorem{lemma}[theorem]{Lemma}

\newtheorem{proposition}[theorem]{Proposition}

\theoremstyle{definition}
\newtheorem{definition}[theorem]{Definition}

\theoremstyle{remark}
\newtheorem{remark}[theorem]{Remark}

\numberwithin{equation}{section}

\begin{document}

\title[Logarithmic Klein-Gordon Equation] {Orbital Stability of Periodic Standing Waves for the Logarithmic Klein-Gordon Equation}

\author[F. Natali]{F\'abio  Natali}

\address{%
Departamento de Matem\'atica - Universidade Estadual de Maring\'a\\
Avenida Colombo, 5790, CEP 87020-900, Maring\'a, PR, Brazil.}

\email{fmanatali@uem.br}

\author[E. Cardoso Jr.]{Eleomar Cardoso Jr.}

\address{%
Departamento de Matem\'atica - Universidade Federal de Santa Catarina\\
Rua Jo\~ao Pessoa, 2750, Bairro da Velha, CEP 89036-002, Blumenau, SC,
Brazil.}

\email{eleomar.junior@ufsc.br}



\begin{abstract}

The main goal of this paper is to present
orbital stability results of periodic standing waves for the one-dimensional
logarithmic Klein-Gordon equation. To do so, we first use compactness
arguments and a non-standard analysis to obtain the existence and
uniqueness of weak solutions for the associated Cauchy problem in the energy space. Second, we prove
the orbital stability of standing waves using a stablity analysis of
conservative systems.
\end{abstract} \maketitle

\section{Introduction}\
 Consider the Klein-Gordon equation with
 $p-$power nonlinearity,
\begin{equation}\label{KG1}
u_{tt}-u_{xx}+u-\mu\log(|u|^p)u=0.
\end{equation}
Here, $u:\mathbb{R}\times
\mathbb{R}\rightarrow \mathbb{C}$ is a complex valued function,
$\mu>0$ and $p$ is a positive integer.\\
\indent When $p=2$, the problem $(\ref{KG1})$ models a
relativistic version of logarithmic quantum mechanics introduced in
\cite{birula2} and \cite{birula1}. The parameter $\mu$ measures the
force of the nonlinear interactions. It has been shown
experimentally (see \cite{gahler}, \cite{shimony} and \cite{shull})
that the nonlinear effects in quantum mechanics are very small,
namely for $0<\mu< 3.3\times 10^{-15}$. Still, the model can be
found in many branches of physics, e.g. nuclear physics, optics,
geophysics (see \cite{gorka2}). In addition, Klein-Gordon equation
with logarithmic potential has been also introduced in the quantum
field theory as in \cite{rosen}. This kind of nonlinearity appears
naturally in inflation cosmology and in super-symmetric field
theories (see \cite{gorka2}).\\
\indent Along the last thirty years, the theory of stability of
traveling/standing wave solutions for nonlinear evolution equation
has increased into a large field that attracts the attention of both
mathematicians and physicists. By considering $\mu=1$
in equation $(\ref{KG1})$, our purpose is to give a contribution in
the stability theory by proving the first result of orbital
stability of periodic waves of the form $u(x,t)=e^{ict}\varphi(x)$,
$t>0$, where $c$ is called frequency and $\varphi$ is a real, even
and periodic function. So, if we substitute this kind of solution in
equation $(\ref{KG1})$, one has the following nonlinear ordinary
differential equation
\begin{equation}\label{travKG}
-\varphi_c''+(1-c^2)\varphi_c-\log(|\varphi_c|^p)\varphi_c=0,
\end{equation}
where $\varphi_c$ indicates the dependence of the function $\varphi$
with respect to the parameter $c$. In a general setting, we can use
a qualitative analysis of planar waves to determine the existence of special
solutions related to the equation, namely: \textit{solitary
waves} and \textit{periodic waves}. To do so, we see that equation
$(\ref{travKG})$ has three equilibrium points, a saddle point and
two center points (around the center points we have strictly
negative and positive periodic orbits). The saddle point is the origin and one sees an explicit
solitary wave (which is unique up to translation) given by
\begin{equation}\label{solitwave}
\varphi_c(x)=e^{\frac{1}{2}+\frac{1-c^2}{p}}e^{-\frac{px^2}{4}}.
\end{equation}
\indent Solution in $(\ref{solitwave})$ is very similar to the solitary wave
concerning the one dimensional version of the logarithmic nonlinear Schr\"odinger equation given by
\begin{equation}\label{logNLS}
iu_t+ \Delta u+\log(|u|^2)u=0,
\end{equation}
where $u:\mathbb{R}^n\times\mathbb{R}\rightarrow\mathbb{C}$ is a complex-valued function and $n\geq1$.\\
\indent Concerning the orbital stability of waves for the
equation $(\ref{logNLS})$, we have interesting results. In fact,
in \cite{Cazenave3} and \cite{cazenave2} the authors
have used variational techniques to get the orbital stability for
solitary waves for the model
$(\ref{logNLS})$ (considering equation $(\ref{KG1})$ in higher dimensions, similar results have been determined by the same authors). Furthermore, the stability of
solitary waves has also been treated in \cite{blanchard}, where the
authors used the general theory in \cite{grillakis1} in the space of
radial functions. In periodic context and $n=1$, we have the work in
\cite{natali1}, where the authors
approached the abstract theory in \cite{grillakis1} to deduce the orbital stability for the equation $(\ref{logNLS})$ in the even periodic Sobolev space $H_{per,e}^1([0,L])$ .\\
\indent In \cite{carles}, the authors have considered the Korteweg-de Vries equation with
logarithmic nonlinearity, namely,
\begin{equation}\label{log-KdV1}
u_t+u_{xxx}+(u\log(|u|))_x=0,
\end{equation}
where $u=u(x,t)$
is a real-valued function with $(x,t)\in\mathbb{R}\times\mathbb{R}$. The authors established results of linear stability
by using numerical approximations and they showed that the Gaussian initial data do not spread out and preserve their spatial Gaussian decay in the time evolution of the linearized logarithmic KdV equation. Concerning the stability of periodic traveling waves for the equation $(\ref{log-KdV1})$, we can cite \cite{CNP}. In both cases, the authors have determined their results of stability in a conditional sense since the uniqueness of solutions for the associated Cauchy problem is obtained by supposing that the evolution $u$ satisfies
$\partial_x(\log(|u|))\in L^{\infty}(0,T;L^{\infty})$.\\
\indent Next, we shall give an outline of our work. The logarithmic nonlinearity
in equation $(\ref{KG1})$ brings a rich set of difficulties
since function $x\in\mathbb{R}\mapsto x\log(|x|)$ is not
differentiable at the origin. The lack of smoothness of the
nonlinearity interferes in questions concerning the local
solvability
since it is not possible to apply a contraction
argument to deduce the existence, uniqueness and continuous dependence
with respect to the initial data. In all these works \cite{carles},
\cite{Cazenave3}, \cite{cazenave2} and \cite{cazenave}, the authors have determined a modified
problem related to the model to overcome the absence of
regularity. The modified solution (or approximate solution)
converges in some sense to the solution of the original problem
provided that convenient uniform estimates
for the approximate solution are established.
The construction of the
approximate solution, which converges in the weak sense
to the solution of the original problem, gives us the existence of weak solutions in a convenient Banach space. \\
\indent In \cite{gorka2} the author have used the Galerkin approximation to deduce existence of global weak solutions related to the problem in $(\ref{KG1})$ in bounded domains by assuming Dirichlet boundary conditions. In our work, the existence of periodic weak solutions follows the same spirit of \cite{gorka2} and we present new results concerning the uniqueness of weak periodic solutions without restrictions on the initial data. The authors in \cite{gorka1} have determined the uniqueness of smooth solutions for the one-dimensional version of $(\ref{KG1})$ posed on the real line by assuming that the initial data is bounded away from zero. Concerning the equation $(\ref{logNLS})$,
a uniqueness result has been treated in \cite{cazenave} by combining energy estimates
with a convenient Gronwall-type inequality and fact that the $L^2-$norm is a conserved quantity.
The $L^2-$norm is not a conserved quantity when equation $(\ref{KG1})$ is considered and it seems a hard task
some kind of adaptation of the arguments contained in \cite{cazenave}.\\
\indent We prove the existence of global weak solutions in time, uniqueness and existence of conserved quantities in the following theorem:
\begin{theorem}\label{t.bc.gl.1.1} Let $p$ be a positive integer and consider $L>0$. There exists a unique global (weak) solution
to the problem (\ref{gl.1}) in the sense that
$$u\in L^{\infty}(0,T;H^1_{per}([0,L])),\
\ u_t\in L^{\infty}(0,T;L^2_{per}([0,L])),\ u_{tt}\in
L^{\infty}(0,T;H^{-1}_{per}([0,L])),$$  and $u$ satisfies,
\begin{eqnarray*}&&
\langle
u_{tt}(\cdot,t),\zeta\rangle_{H^{-1}_{per},H^1_{per}}+\displaystyle\int_0^L\nabla
u(\cdot,t)\cdot\overline{\nabla\zeta}\
dx\nonumber\\
&&+\displaystyle\int_0^Lu(\cdot,t)\overline{\zeta}\
dx=\displaystyle\int_0^Lu(\cdot,t)\log(|u(\cdot,t)|^p)\overline{\zeta}\
dx\ \ \ \ \mbox{a.e.}\ t\in [0,T],\end{eqnarray*} for all $\zeta\in
H^1_{per}([0,L])$. Furthermore, $u$ must satisfy
$$u(\cdot,0)=u_0\ \ \textrm{and}\ \ u_t(\cdot,0)=u_1.$$
\indent In addition, the weak solution satisfies the following conserved quantities:
\begin{equation}\label{cons-quant-1}\mathcal{E}(u(\cdot,t),u'(\cdot,t))=\mathcal{E}(u_{0},u_{1})\
\ \mbox{and}\ \
\mathcal{F}(u(\cdot,t),u'(\cdot,t))=\mathcal{F}(u_{0},u_{1}),
\end{equation}
a.e. $t\in [0,T]$. Here, $\mathcal{E}$ and $\mathcal{F}$ are defined by
\begin{eqnarray}\label{gl.108.1}\ \ \ \ \mathcal{E}(u(\cdot,t),u_t(\cdot,t)):=\displaystyle\frac{1}{2}\displaystyle\left[\displaystyle\int_0^L\
|u_x(\cdot,t)|^2+|u_t(\cdot,t)|^2+\displaystyle\left(1+\displaystyle\frac{p}{2}-\log(|u(\cdot,t)|^p)\right)|u(\cdot,t)|^2
dx\right]\end{eqnarray} and
\begin{eqnarray}\label{gl.109.1}\mathcal{F}(u(\cdot,t),u_t(\cdot,t))&:=&\textrm{Im}\displaystyle\int_0^L\overline{u(\cdot,t)}\ u_t(\cdot,t)\
dx\nonumber\\
\\
&=&\displaystyle\int_0^L[\textrm{Re}\ u(\cdot,t)\ \textrm{Im}\
u_t(\cdot,t)-\textrm{Im}\ u(\cdot,t)\ \textrm{Re}\ u_t(\cdot,t)]\
dx.\nonumber
\end{eqnarray}

\end{theorem}

We now present the basic ideas to prove Theorem $\ref{t.bc.gl.1.1}$. In fact, we first present an approximate problem
for the equation $(\ref{KG1})$ defined on the subspace
$$V_m=[\omega_1,...,\omega_m],$$ where
$\{\omega_{\nu}\}_{\nu\in\mathbb{N}}$ is a complete orthonormal set
in $L^2_{per}([0,L])$ which is orthogonal in $H_{per}^1([0,L])$.
After that, we  use Caratheodory's Theorem to deduce the
existence of an approximate solution $u_m$
for the approximate problem (see $(\ref{gl.3})$) for all $m\in\mathbb{N}$.
Thus, uniform bounds are required in order to get the existence of global
weak solutions after passage to the limit. In addition, the approximate solution satisfies
$$\mathcal{E}(u_m(\cdot,t),u_m'(\cdot,t))=\mathcal{E}(u_{0,m},u_{1,m})\
\ \mbox{and}\ \
\mathcal{F}(u_m(\cdot,t),u_m'(\cdot,t))=\mathcal{F}(u_{0,m},u_{1,m}).$$
However, since $u$ and $u_t$ are obtained as weak limits of the
approximate solution $u_m$ and $u_m'$, respectively, we only
guarantee the ``conserved'' inequalities
\begin{equation}\label{eqcons}
\mathcal{E}(u(\cdot,t),u_t(\cdot,t))\leq\mathcal{E}(u_0,u_{1})\ \
\mbox{and}\ \
\mathcal{F}(u(\cdot,t),u_t(\cdot,t))\leq\mathcal{F}(u_{0},u_{1})\ \
\ \mbox{a.e.}\ t\in[0,T].\end{equation} In order to obtain that
equalities in $(\ref{eqcons})$ occur, we employ the arguments in \cite{lions} (see Chapter 1, page 22) provided that a
result of uniqueness for weak solutions can be established. The uniqueness of solutions for the model $(\ref{KG1})$ is one of the cornerstones of our paper. In fact, as we have already mentioned above, the nonlinearity $x\in\mathbb{R}\mapsto x\log(|x|)$ is not locally Lipschitz in
convenient Lebesgue measurable spaces. Thus, uniqueness of solutions can not be determined using a difference of weak solutions with zero initial data combined with contraction arguments as is usual in evolution problems. To give a positive answer for the uniqueness of solutions, let $L>0$ be fixed. We first establish the existence of $T_0\in
(0,\frac{L}{4})$ such that the weak solution $w$ for the
inhomogeneous Cauchy problem
\begin{eqnarray}\label{NCP}\displaystyle\left\{\begin{array}{l}
                                       w_{tt}-w_{xx}=f(x,t),\ \ (x,t)\in\mathbb{R}\times [0,T_0].\\
                                       w(x,0)=0,\ \ w_{t}(x,0)=0,\ \
                                       x\in\mathbb{R}.\\
                                       w(x+L,t)=w(x,t) \ \ \textrm{for\ all}\
                                       t\in [0,T_0],\ \ x\in
                                       \mathbb{R},
                                     \end{array}\right.\end{eqnarray}
satisfies
\begin{equation}\label{eqint}w(x,t)=\displaystyle\frac{1}{2}\displaystyle\int_0^{t}\int_{x-t+\tau}^{x+t-\tau}f(y,\tau)\ dy\
d\tau, \ \
(x,t)\in\mathbb{R}\times [0,T_0]. \end{equation}
We use the characterization in $(\ref{eqint})$ to determine our uniqueness result provided that $t\in[0,T_0]$. Moreover, since solution $w$ is global in time, we employ an interaction argument to get the uniqueness over $\mathbb{R}\times [0,T]$, $T>0$. In \cite{cazenave}, the same formula $(\ref{eqint})$ has been used to obtain uniqueness of weak solutions for a similar problem as in $(\ref{NCP})$ posed on $\mathbb{R}^3\times [0,T]$.\\

 \indent With the results determined by Theorem $\ref{t.bc.gl.1.1}$ in hands, we are enable to establish our orbital stability result given by:

\begin{theorem}\label{t.lkg.2.1.1} Consider $p=1,2,3$ and $c\in I$ satisfying $\displaystyle\frac{\sqrt{p}}{2}<|c|<1$. Let $\varphi_c$ be a periodic
solution for the equation $(\ref{travKG})$. The periodic
wave $\widetilde{u}(x,t)=e^{ict}\varphi_c(x)$ is orbitally stable by
the periodic flow of the equation (\ref{KG1}).
\end{theorem}

 We present some important facts concerning Theorem $\ref{t.lkg.2.1.1}$. First, it is important to mention that we are interested in obtaining smooth periodic waves. Thus, the presence of a logarithmic type nonlinearity in the ODE  $(\ref{travKG})$  forces us to consider \textit{positive and periodic waves}. A planar analysis concerning the equilibrium points of $(\ref{travKG})$ gives us that the period of our periodic solutions must satisfy $L>\frac{2\pi}{\sqrt{p}}$. Additionally, it is important to mention that it is possible deduce (at least formally) periodic waves having small periods (see \cite{gorka1}). However, we are not capable to decide about the orbital stability in this case because the energy $\mathcal{E}$ in $(\ref{gl.108.1})$ is not a smooth functional at these waves (in fact, the waves obtained with this property have two zeros over the interval $[0,L]$) and our stability analysis consists in proving that the periodic waves are minimizers of the energy $\mathcal{E}$ with constraint $\mathcal{F}$.\\
 \indent In a general setting, let us consider the Hill operator
\begin{equation}\label{hill11}
\mathcal{L}_Q=\displaystyle-\frac{d^2}{dx^2}+Q(x),
\end{equation}
where $Q$ is a smooth $L-$periodic function. In \cite{natali1}, the authors have presented a tool based on the classical Floquet theorem  to establish a
characterization of the first eigenvalues of $\mathcal{L}_Q$
by knowing one of its eigenfunctions. The key point of the work is
that it is not necessary to know an explicit smooth solution
$\varphi=\varphi(x)$ which solves the general nonlinear differential
equation
\begin{equation}\label{geneq}
-\varphi''+h(c,\varphi)=0,\end{equation} where $h$ is smooth in
a open subset contained in $\mathbb{R}^2$. Moreover, in
\cite{natali1} it is possible to decide that the eigenvalue
\textit{zero} is simple without knowing an explicit periodic
solution which solves equation $(\ref{geneq})$. In \cite{AN1},
\cite{AN2}, \cite{AN}, and references
therein, the authors have determined explicit solutions to obtain
the behavior of the non-positive spectrum for the
 Hill operator $(\ref{hill11})$ and the orbital stability of periodic waves. \\
\indent Following arguments in \cite{bona}, \cite{grillakis1} and
\cite{weinstein1}, the first requirement for the stability of
periodic waves related to the equation $(\ref{KG1})$ concerns in
proving the existence of an open interval $I\subset\mathbb{R}$ and a
smooth branch $c\in I\mapsto \varphi_{c}$ which solves
$(\ref{travKG})$, all of them with the same period $L>0$. In our
study, if $L\in\left(\frac{2\pi}{\sqrt{p}},+\infty\right)$ we are
capable to combine the approach in \cite{natali1} and the Implicit
Function Theorem to deduce a smooth curve of even periodic solutions
defined over $I\subset \mathbb{R}$. The second step to obtain the
stability of periodic waves is to analyze the behavior of the
non-positive spectrum of the linearized operator
\begin{equation*}\label{opera1}
\mathcal{L}_{\varphi_c}=\displaystyle\left(
                  \begin{array}{cccc}
                    \mathcal{L}_{Re,\varphi_c} & 0  \\
                    0 & \mathcal{L}_{Im,\varphi_c} \\
                  \end{array}
                \right),
\end{equation*}
where $\mathcal{L}_{Re,\varphi_c}$ and $\mathcal{L}_{Im,\varphi_c}$
are defined, respectively, as
\begin{eqnarray}\label{lkg.13.14}\mathcal{L}_{Re,\varphi_c}=\displaystyle\left(\begin{array}{cc}
                                                         -\partial^2_{x}+1-\log(|\varphi_c|^p)-p & -c \\
                                                         -c &
                                                         1
                                                       \end{array}\right)
\end{eqnarray} and \begin{eqnarray}\label{lkg.14.14}\mathcal{L}_{Im,\varphi_c}=\displaystyle\left(\begin{array}{cc}
                                                         -\partial^2_{x}+1-\log(|\varphi_c|^p) & c \\
                                                         c & 1
                                                       \end{array}\right).\end{eqnarray}
The approach in \cite{natali1} can be used to conclude that the
diagonal operator $\mathcal{L}_{\varphi_c}$ has only one negative
eigenvalue which is simple, zero is double eigenvalue with
$$\ker(\mathcal{L}_{\varphi_c})=\textrm{span}\{(\varphi_c',c\varphi_c',0,0),(0,0,\varphi_c,-c\varphi_c)\}.$$
Moreover, the remainder of the spectrum is discrete and bounded away
from zero.\\
\indent Finally, the stability of periodic
waves can be determined provided that the following condition holds
\begin{eqnarray}\label{lkg.56.1}\displaystyle\left\langle\mathcal{L}_{Re,\varphi_c}^{-1}\displaystyle\left(\begin{array}{c}
                                                                                             c\varphi_c \\
                                                                                             \varphi_c
                                                                                           \end{array}
\right),\displaystyle\left(\begin{array}{c}
                                                                                             c\varphi_c \\
                                                                                             \varphi_c
                                                                                           \end{array}
\right)\right\rangle_{2,2}&=&\displaystyle\int_0^L\varphi_c^2\
dx+c\displaystyle\frac{d}{dc}\displaystyle\left(\displaystyle\int_0^L\varphi_c^2\
dx\right)=-d''(c)<0.\nonumber\end{eqnarray} \indent In our case, we
establish that $d''(c)>0$ provided that
$c^2\in\left(\frac{p}{4},1\right)$ and $p=1,2,3$.\\

This paper is organized as follows: in Section 2 we present some
basic notations and results which will be useful in the whole paper. In
Section 3 we study existence, uniqueness and existence of
conservation laws related to the model $(\ref{KG1})$. Finally, the
orbital stability of periodic waves will be shown in Section 4.

\section{Preliminary Results}

In this section, some basic notation and results are presented in
order to give a complete explanation of the arguments discussed in our
paper. The arguments below
can be found in \cite{iorio1}. \\

The $L^2$-based Sobolev spaces of periodic functions are defined as
follows: if $\mathcal{P}=C_{per}^{\infty}$ denotes the collection of
all functions $f:\mathbb{R}\rightarrow\mathbb{C}$ which are
$C^{\infty}$ and periodic with period $L>0$, collection
$\mathcal{P}'$ of all continuous linear functionals from
$\mathcal{P}$ into $\mathbb{C}$. \\
\indent If $\Psi\in\mathcal{P}'$ we denote the evaluation of $\Psi$ at
$\varphi\in\mathcal{P}$ by $\Psi(\varphi)=[\Psi,\varphi]$. For $k\in\mathbb{Z}$, consider
$\displaystyle\Lambda_k(x)=e^{\frac{2\pi ik x}{L}}$, $x\in\mathbb{R}$ and $k\in\mathbb{Z}$.
The Fourier transform of $\Psi\in\mathcal{P}'$ is a function
$\widehat{\Psi}:\mathbb{Z}\rightarrow\mathbb{C}$ defined by $
\widehat{\Psi}(k)=\frac{1}{L}[\Psi,\Lambda_k]$, $ k\in\mathbb{Z}$.
$\widehat{\Psi}(k)$ are called the Fourier coefficients of $\Psi$.
As usual, a function $f\in L_{per}^p([0,L])$, $p\geq1$ is an element
of $\mathcal{P}'$ by defining
$$
[f,g]=\displaystyle\frac{1}{L}\int_{0}^{L}f(x)g(x)dx,
\ \ \ \  g\in\mathcal{P}.
$$
\indent We denote by $C_{per}:=C_{per}^0$ the space of the continuous and $L$-periodic
functions. For $s\in\mathbb{R}$, the Sobolev space
$H_{per}^{s}([0,L]):=H_{per}^s$ is the set of all $f\in\mathcal{P}'$
such that
$$\displaystyle||f||_{H_{per}^s}^2:=||f||_{s}^{2}\equiv L\sum_{k=-\infty}^{+\infty}(1+|k|^2)^s|\widehat{f}(k)|^2
<\infty,$$
where $\widehat{f}$ indicates the periodic Fourier transform defined as above.\\
\indent Collection $H_{per}^s$ is a Hilbert space with inner
product
$$( f,g)_{H_{per}^s}:=( f,g)_{s}=L\sum_{k=-\infty}^{+\infty}(1+|k|^2)^s
\widehat{f}(k)\overline{\widehat{g}(k)}.
$$
When $s=0$, $H_{per}^s([0,L])$ is a Hilbert space which is
isometrically isomorphic to a subspace of $L^2([0,L])$ and
$(f,g)_{H_{per}^0}=(f,g)_{L_{per}^2}=\displaystyle\int_{0}^{L}f(x)\overline{g(x)}dx$.
Space $H_{per}^0$ will be denoted by $L_{per}^2$ and its norm will
be $||.||_{L^2_{per}}$. Of course, $H_{per}^s\subseteq L_{per}^2$, $s\geq0$, and we have the Sobolev embedding
$H_{per}^s\hookrightarrow C_{per}$, $s>\frac{1}{2}$. For $s\geq0$, we denote
$H_{per,e}^s([0,L])$ as the closed subspace of $H_{per}^s([0,L])$
constituted by even periodic functions. If $H$ is a Hilbert space, we denote by $\langle \cdot,\cdot\rangle_{H',H}$ the duality pair. In particular, if $H=L_{per}^2([0,L])$ we are able to write $H'=L_{per}^2([0,L])$ with $\langle \cdot,\cdot\rangle_{L_{per}^2,L_{per}^2}=(\cdot,\cdot)_{L_{per}^2}$.\\

\indent Next propositions are technical results used in our
manuscript.
\begin{proposition}\label{p.bc.gl.02} Consider $\alpha_1,\ \alpha_2\in\mathbb{C}$
satisfying $|\alpha_1|\leq|\alpha_2|$. Then,
\begin{eqnarray*}|\alpha_1\log(|\alpha_1|)-\alpha_2\log(|\alpha_2|)|\leq [1+|\log(|\alpha_2|)|]\ |\alpha_1-\alpha_2|.\end{eqnarray*}
\end{proposition}
\begin{proof} See \cite[Chapter II, Lemma
2.4.3]{cazenave}.\end{proof}

\begin{proposition}\label{t.bc.gl.5} Consider $\alpha>0$ and
$v_0\in\displaystyle\left[0,\displaystyle\frac{1}{e}\right]$. Let
$v\in L^{\infty}(0,T)$ be a non-negative function. If
\begin{eqnarray*}v(t)\leq v_0-\alpha\displaystyle\int_0^tv(s)\log(v(s))\
ds\ \ \ \ a.e.\ \ t\in [0,T],\end{eqnarray*} then, $$v(t)\leq
\displaystyle\left(v_0\right)^{e^{-\alpha t}}$$  with $$0\leq
t\leq\inf\displaystyle\left\{\displaystyle\frac{\log(-\log(v_0))}{\alpha},T\right\}.$$
\end{proposition}
\begin{proof} See \cite[Chapter III, Corollary
2.1.2]{cazenave}.\end{proof}


\textbf{Important Remark:}  Since $u$ found in $(\ref{KG1})$ is a complex valued function, the notion of stability will be considered over the complex space
$X=H_{per}^1([0,L])\times L_{per}^2([0,L])$. However, in some particular cases, it is convenient to consider real spaces $H_{per}^s([0,L])$, $s\geq0$, instead of complex ones (see Section 4). In this paper, we will not distinguish whether this space is real or complex.

\section{Existence and Uniqueness of Weak
Solutions}\label{sec.exist.}

Next, we establish results of existence and uniqueness of weak
solutions. Let us consider $(u_0,u_1)\in H^1_{per}([0,L])\times
L^2_{per}([0,L])$. First, we use Galerkin's method to prove the
existence of weak solutions for the following Cauchy
problem
\begin{eqnarray}\label{gl.1}\displaystyle\left\{\begin{array}{l}
                                       u_{tt}-u_{xx}+u-\log(|u|^p)u=0,\ \ (x,t)\in\mathbb{R}\times [0,T].\\
                                       u(x,0)=u_0(x),\ \ u_{t}(x,0)=u_1(x),\ \
                                       x\in\mathbb{R}.\\
                                       u(x+L,t)=u(x,t)\ \ \textrm{for\ all}\
                                       t\in [0,T],\ \ x\in
                                       \mathbb{R},
                                     \end{array}\right.\end{eqnarray}
where $p\in \mathbb{N}$, $L>0$ and $T>0$. Regarding the uniqueness
of weak solutions, we need to use a non-standard analysis to rewrite
the weak solution into a convenient integral form. A consequence of
this last fact enables us to deduce the existence of two conserved
quantities $\mathcal{E}$ and $\mathcal{F}$ as in $(\ref{gl.108.1})$
and $(\ref{gl.109.1})$.

\begin{remark} \textit{ Once obtained the results contained in Theorem $\ref{t.bc.gl.1.1}$, we can apply the arguments in \cite{lions}
to establish the following smoothness result. Indeed, if $u$ is a
weak solution to the problem (\ref{gl.1}), one has
$$u\in C^0([0,T];L^2_{per}([0,L]))\cap C_s(0,T;H^1_{per}([0,L]))$$ and $$u_t\in C^0([0,T];H^{-1}_{per}([0,L]))\cap
C_s(0,T;L^2_{per}([0,L])).$$
Here, $C_s(0,T;H)$ indicates the space of weakly continuous functions in the Hilbert space $H$, that is, the set of $f\in L^{\infty}(0,T;H)$ such that the map $t\mapsto \langle f(t),v\rangle_{H',H}$ is continuous over $[0,T]$ for all $v\in H$. }
\end{remark}

Next, we present the proof of Theorem $\ref{t.bc.gl.1.1}$ by splitting it into two parts. The first one concerns the existence and uniqueness of global weak solutions.
\begin{proposition}\label{t.bc.gl.1} Let $p$ be a positive integer and consider $L>0$. There exists a unique global weak solution
to the problem (\ref{gl.1}) in the sense as mentioned in Theorem $\ref{t.bc.gl.1.1}$.
\end{proposition}
\textit{Proof.} We perform the proof of this result by showing two basics steps. The first one corresponds to
the existence of weak solutions (which contains the
construction of the approximate problem, a priori estimates and
passage to the limit). After that, we show the uniqueness of weak solutions. \\

Step 1. \textit{Existence of weak solutions.}\\

Here, the main point is to use Galerkin's method to determine the
existence of weak solutions related to the problem $(\ref{gl.1})$.
Firstly, we need to present the approximate solution $u_m$,
$m\in\mathbb{N}$. After that, we obtain a priori estimate to give an
uniform bound for the approximate solution in a convenient space.
This last fact enables us to consider this solution for all values
$T>0$ and for $m$ large, we can pass to the limit to obtain a weak solution $u$.\\
\indent Indeed, let
$\{\omega_j\}_{j\in\mathbb{N}}$ be a sequence of periodic eigenfunctions
satisfying
\begin{equation}\label{eigen-val-p}\displaystyle\left\{\begin{array}{l}
                                       -\partial_x^2\omega_j+\omega_j=\lambda_j\omega_j\ \ \textrm{in}\ \ [0,L]. \\
                                       \omega_j(0)=\omega_j(L).
                                     \end{array}
\right.\end{equation}
Using Fourier analysis, we can choose $\{\omega_j\}_{j\in\mathbb{N}}$ such
that \begin{itemize}
\item $\{\omega_{j}\}_{j\in\mathbb{N}}$ is a complete orthonormal set
in $L^2_{per}([0,L])$;
\item $\{\omega_{j}\}_{j\in\mathbb{N}}$ is a complete orthogonal set in $H^1_{per}([0,L])$;
\item
$\{\omega_{j}\}_{j\in\mathbb{N}}$ is smooth.
\end{itemize}

\indent By $(\ref{eigen-val-p})$ it is possible to give a characterization of the eigenfunctions $\omega_{j}$ and eigenvalues $\lambda_{j}$, $j\in\mathbb{N}$, as
$$\{\omega_j\}_{j\in\mathbb{N}}=\left\{\frac{e^{\frac{2i\pi j \cdot}{L}}}{L}\right\}_{j\in\mathbb{N}}\ \ \ \ \ \mbox{and}\ \ \ \ \ \ \{\lambda_j\}_{j\in\mathbb{N}}=\left\{\left(\frac{2\pi j}{L}\right)^2+1\right\}_{j\in\mathbb{N}}.$$
\indent Consider $V_m$ the subspace spanned by the $m$ first
eigenfunctions $\omega_1,\ \omega_2,\ \ldots, \omega_m$. Let
$u_m(t)\in V_m$ be the function defined as
\begin{eqnarray}\label{gl.2}
u_m(t)=\displaystyle\sum_{j=1}^mg_{jm}(t)\omega_j.\end{eqnarray}
\indent We see that solution $u_m$ in (\ref{gl.2}) solves the
following approximate problem
\begin{eqnarray}\label{gl.3}\displaystyle\left\{\begin{array}{l}
                                      \langle u_m''(t),{\omega}_j\rangle_{L^2_{per},L^2_{per}}+\langle \nabla
u_m(t),\nabla
{\omega}_j\rangle_{L^2_{per},L^2_{per}}=\\
-\langle
u_m(t),{\omega}_j\rangle_{L^2_{per},L^2_{per}}+\langle u_m(t)\log(|u_m(t)|^p),{\omega}_j\rangle_{L^2_{per},L^2_{per}},\ \ \ j\in\{1,2,\ldots,m\}.\\
                                                  \\
u_m(0)=u_{0,m}=\displaystyle\sum_{j=1}^m\alpha_{jm}\omega_j,\ \alpha_{jm}=\langle u_0,{\omega}_j\rangle_{L^2_{per},L^2_{per}}, \\
                                                   \\
u_m'(0)=u_{1,m}=\displaystyle\sum_{j=1}^m\beta_{jm}\omega_j,\ \
\beta_{jm}=\langle u_1,{\omega}_j\rangle_{L^2_{per},L^2_{per}}.
\end{array} \right.\end{eqnarray} \indent Thus, Caratheodory's
Theorem can be used to establish that (\ref{gl.3}) has an absolutely
continuous solution $u_m$ given by (\ref{gl.2}), defined on an
interval $[0,T_m)$, where $0<T_m\leq T$. Next, we need to show good
bounds for the solution $u_m$ to guarantee that solution $u_m$ can
be defined for all values of
$T>0$.\\

\indent Multiplying both sides of the first identity in
(\ref{gl.3}) by $\overline{g_{jm}'(t)}$ and adding the final result
in $j=1,2,\ldots,m$, one has
\begin{eqnarray*}&&\displaystyle\frac{d}{dt}\left(\|u_m'(t)\|^2_{L^2_{per}}\right)+\displaystyle\frac{d}{dt}\left(\|\nabla
u_m(t)\|^2_{L^2_{per}}\right)+\displaystyle\left(1+\displaystyle\frac{p}{2}\right)\displaystyle\frac{d}{dt}\left(\|u_m(t)\|^2_{L^2_{per}}\right)\nonumber\\
\nonumber\\
&&={p}\displaystyle\left[\displaystyle\frac{d}{dt}\left(\displaystyle\int_0^L|u_m(t)|^2\log(|u_m(t)|)\
dx\right)\right].\nonumber\end{eqnarray*} Hence for all $t\in
[0,T]$, we obtain
\begin{eqnarray}\label{gl.23}\begin{array}{l}\|u_m'(t)\|^2_{L^2_{per}}+\|\nabla
u_m(t)\|^2_{L^2_{per}}+\displaystyle\left(1+\displaystyle\frac{p}{2}\right)\|u_m(t)\|^2_{L^2_{per}}-p\displaystyle\int_0^L|u_m(t)|^2\log(|u_m(t)|)\\
\\
=\|u_{1,m}\|^2_{L^2_{per}}+\|\nabla
u_{0,m}\|^2_{L^2_{per}}+\displaystyle\left(1+\displaystyle\frac{p}{2}\right)\|u_{0,m}\|^2_{L^2_{per}}-p\displaystyle\int_0^L|u_{0,m}|^2\log(|u_{0,m}|)\
dx.\end{array}\end{eqnarray} Since
$u_{0,m}\stackrel{m\rightarrow\infty}{\longrightarrow} u_0$ in
$H^1_{per}([0,L])$ and
$u_{1,m}\stackrel{m\rightarrow\infty}{\longrightarrow} u_1$ in
$L^2_{per}([0,L])$, there exists a constant $C_1>0$ such that
\begin{eqnarray*}\|u_{1,m}\|^2_{L^2_{per}}+\|\nabla
u_{0,m}\|^2_{L^2_{per}}+\displaystyle\left(1+\displaystyle\frac{p}{2}\right)\|u_{0,m}\|^2_{L^2_{per}}\leq
C_1\ \textrm{for\ all}\ m\in\mathbb{N}.\end{eqnarray*}

Furthermore, the Sobolev embedding $H^1_{per}([0,L])\hookrightarrow
L^{\infty}_{per}([0,L])$ yields
\begin{eqnarray*}\displaystyle\left|p\displaystyle\int_0^L|u_{0,m}|^2\log(|u_{0,m}|)\ dx\right|\leq C_2\ \textrm{for\ all}\
m\in\mathbb{N}.\end{eqnarray*} Define the constant $C_3:=C_1+C_2$.
We can write
\begin{eqnarray}\begin{array}{l}\label{gl.26}\|u_m'(t)\|^2_{L^2_{per}}+\|\nabla
u_m(t)\|^2_{L^2_{per}}+\displaystyle\left(1+\displaystyle\frac{p}{2}\right)\|u_m(t)\|^2_{L^2_{per}}\\
\\
\leq
C_3+\displaystyle\frac{p}{2}\displaystyle\int_0^L|u_m(t)|^2\log(|u_m(t)|^2)\
dx\ \ \textrm{for\ all}\ m\in\mathbb{N}.\end{array}\end{eqnarray}

\indent In order to estimate the right-hand-side of the inequality
(\ref{gl.26}), we see that
\begin{eqnarray}\label{novo1}\displaystyle\int_0^L|u_m(t)|^2\log(|u_m(t)|)\ dx\leq \displaystyle\int_0^L|u_m(t)|^3\ dx\leq \|u_m(t)\|_{L^{\infty}_{per}}\|u_m(t)\|_{L^{2}_{per}}^2.\end{eqnarray}

\indent Using the Sobolev embedding
$H^1_{per}([0,L])\hookrightarrow L^{\infty}_{per}([0,L])$ and an
application of Young's inequality we obtain the existence of
a constant $C_4>0$ such that
\begin{eqnarray}\label{novo2}\displaystyle\int_0^L|u_m(t)|^2\log(|u_m(t)|)\ dx&\leq& C_4\|u_m(t)\|_{H^{1}_{per}}\|u_m(t)\|_{L^{2}_{per}}^2\nonumber\\
\\
&\leq& C_4\displaystyle\left[\displaystyle\frac{\varrho
\|u_m(t)\|^2_{H^1_{per}}}{2}+\displaystyle\frac{\|u_m(t)\|^4_{L^2_{per}}}{2\varrho}\right],\nonumber\end{eqnarray}
where $\varrho>0$ is a fixed parameter.

\indent Let $\varrho>0$ be sufficiently small. By combining
(\ref{gl.26}) and (\ref{novo2}), there exists a constant $C_5>0$
such that
\begin{eqnarray}\begin{array}{l}\label{gl.26.5}\ \ \|u_m'(t)\|^2_{L^2_{per}}+\|\nabla
u_m(t)\|^2_{L^2_{per}}+\|u_m(t)\|^2_{L^2_{per}}\leq
C_5+C_5\displaystyle\left[\|u_m(t)\|^2_{L^2_{per}}\right]^2\
\textrm{for\ all}\ m\in\mathbb{N}.\end{array}\end{eqnarray}

Next, let us denote
$$u_m(t)=u_{m}(0)+\displaystyle\int_0^t\displaystyle\frac{\partial
u_m}{\partial \varsigma}(\varsigma)\ d\varsigma.$$ Hence,
\begin{eqnarray*}\|u_m(t)\|^2_{L^2_{per}} &\leq&
2\|u_{0,m}\|^2_{L^2_{per}}+2\displaystyle\left\|\displaystyle\int_0^t\displaystyle\frac{\partial
u_m}{\partial \varsigma}(\varsigma)\ d\varsigma\right\|^2_{L^2_{per}}\nonumber\\
\nonumber\\
&\leq&C_6+2\displaystyle\int_0^t\|u_m'(\varsigma)\|^2_{L^2_{per}}\
d\varsigma,\end{eqnarray*} where $C_6>0$ is a constant. Using
inequality (\ref{gl.26.5}),
\begin{eqnarray}\label{gl.38}\|u_m(t)\|^2_{L^2_{per}}\leq C_7+C_7\displaystyle\int_0^t\displaystyle\left[\|u_m(\varsigma)\|^2_{L^2_{per}}\right]^2\ d\varsigma\
\textrm{\ for\ all\ } m\in\mathbb{N},
\end{eqnarray} where $C_7>0$ is a constant.

From Gronwall-Bellman-Bihari inequality, there
exists $C_8(T)>0$ such that
\begin{eqnarray*}\|u_m(t)\|^2_{L^2_{per}}
\leq C_8(T)\ \textrm{for\ all} \ m\in\mathbb{N}.
\end{eqnarray*} In addition, by (\ref{gl.26.5}) there exists a constant $C_9(T)>0$
satisfying
\begin{eqnarray}\label{gl.39}\ \ \ \ \ \ \|u_m'(t)\|^2_{L^2_{per}}+\|\nabla
u_m(t)\|^2_{L^2_{per}}+\|u_m(t)\|^2_{L^2_{per}}\leq C_9(T)\
\end{eqnarray} for all $m\in\mathbb{N}$. By (\ref{gl.39}), we can deduce
that the interval of existence of approximate solutions can be
chosen as $[0,T]$ for all $T>0$. Furthermore,
\begin{eqnarray}\label{gl.41}\|u_m\|^2_{L^{\infty}(0,T;H^1_{per})}+\|u_m'\|^2_{L^{\infty}(0,T;L^2_{per})}\leq
C_9(T)\ \textrm{for\ all}\ m\in\mathbb{N}.\end{eqnarray}

The next step is to estimate $\|u_m''(t)\|_{H^{-1}_{per}}$. Consider
$v\in H^1_{per}([0,L])$ such that $\|v\|_{H^1_{per}}\leq 1$. 
Decompose $v:=v_1+v_2$, where $v_1\in \
\textrm{span}\{\omega_{\nu}\}_{\nu=1}^{m}$ and
\begin{eqnarray}\label{gl.42}0=\langle
v_2,\omega_{j}\rangle_{L^2_{per},L^2_{per}},\ \ \ \ \
j\in\{1,2,\ldots,m\}.\end{eqnarray} Identities in (\ref{gl.3}) and
(\ref{gl.42}) give us
\begin{eqnarray*}\langle u_m''(t),v\rangle_{L^2_{per},L^2_{per}}&=&-\langle
\nabla u_m(t),\nabla v_1\rangle_{L^2_{per},L^2_{per}}-\langle
u_m(t),v_1\rangle_{L^2_{per},L^2_{per}}\nonumber\\
\nonumber\\
&+&\langle
u_m(t)\log(|u_m(t)|^p),v_1\rangle_{L^2_{per},L^2_{per}}\end{eqnarray*}
and
\begin{eqnarray}\label{gl.45}|\langle u_m''(t),v\rangle_{L^2_{per},L^2_{per}}|&\leq&2\|u_m(t)\|_{H^1_{per}}\nonumber\\
\\
&+&\displaystyle\left(\displaystyle\int_0^L|u_m(t)\log\
(|u_m(t)|^{p})|\ dx\right)\|v_1\|_{L^{\infty}_{per}}.\nonumber
\end{eqnarray}

On the other hand, we can use inequality (\ref{gl.41}) to show that
$\|u_m(t)\|_{H^1_{per}}\leq \sqrt{C_{9}(T)}\ \textrm{for\ all\ }
m\in\mathbb{N}.$ Next, the Sobolev embedding
$H^1_{per}([0,L])\hookrightarrow L^{\infty}_{per}([0,L])$ and the
orthogonality of $\{\omega_{\nu}\}_{\nu\in\mathbb{N}}$ in
$H^1_{per}([0,L])$ enable us to conclude
$\|v_1\|_{L^{\infty}_{per}}\leq C_{10}\|v_1\|_{H^{1}_{per}}\leq
C_{10}\|v\|_{H^{1}_{per}}\leq C_{10},$ $C_{10}>0$ being a
constant. In addition, there exists a constant $C_{11}(T)>0$ such
that
\begin{eqnarray*}\displaystyle\int_0^L|u_m(t)\log\ (|u_m(t)|^p)|\ dx\leq
C_{11}(T)\ \textrm{for\ all}\ m\in\mathbb{N}.\end{eqnarray*}

Taking $C_{12}(T):=2\sqrt{C_{9}(T)}+C_{10}C_{11}(T)$, we have from
(\ref{gl.45}) and the final inequality
\begin{eqnarray*}|\langle
u_m''(t),v\rangle_{H^{-1}_{per},H^1_{per}}|\leq C_{12}(T)\
\textrm{for\ all}\ m\in\mathbb{N}.\end{eqnarray*} Since $v$ is arbitrary one has
\begin{eqnarray}\label{gl.49}\|u_m''\|_{L^{\infty}(0,T;H^{-1}_{per})}\leq
C_{12}(T)\ \textrm{for\  all}\  m\in\mathbb{N}.\end{eqnarray}\

\indent Inequalities (\ref{gl.41}) and (\ref{gl.49}) enable us to
consider a subsequence (still denoted by $\{u_m\}_{m\in\mathbb{N}}$)
and a function $u$ such that
\begin{eqnarray}\label{gl.50}\displaystyle\begin{array}{l}
                                             u_m\rightharpoonup u\ \ \
\textrm{weak\ in\ } \ \ L^2(0,T;H^1_{per}([0,L])), \\
                                             u_m'\rightharpoonup u'\ \ \ \textrm{weak\ in\ } \ \
L^2(0,T;L^2_{per}([0,L])), \\
                                             u_m''\rightharpoonup u''\ \ \ \textrm{weak\ in\ } \ \
L^2(0,T;H^{-1}_{per}([0,L])), \\
                                             u_m\stackrel{\star}{\rightharpoonup}u\ \ \ \textrm{weak-$\star$\ in\ } \ \
L^{\infty}(0,T;H^1_{per}([0,L])), \\
                                             u_m'\stackrel{\star}{\rightharpoonup}u'\ \ \ \textrm{weak-$\star$\ in\ } \ \ L^{\infty}(0,T;L^2_{per}([0,L])), \\
                                             u_m''\stackrel{\star}{\rightharpoonup}u''\ \ \
\textrm{weak-$\star$\ in\ } \ \ L^{\infty}(0,T;H^{-1}_{per}([0,L])).
                                           \end{array}
\end{eqnarray} \indent Moreover, using Aubin-Lions-Simon
Theorem (see \cite[Theorem II.5.16]{lions1} for more details) and the fact that $u_m$
is bounded in $L^2(0,T;H^1_{per}([0,L]))$, we can choose
$\{u_m\}_{m\in\mathbb{N}}$ such that
\begin{eqnarray}\label{gl.55}u_m\rightarrow u\ \ \ \textrm{strongly\ in\ } \ \
L^2(0,T;L^2_{per}([0,L])).\end{eqnarray} As a consequence of this
last convergence we obtain $u_m\longrightarrow u\ \ \ \textrm{a.e.\
\ in}\ \ \ [0,L]\times [0,T].$ Now, since the map $x\mapsto
x\log(|x|^p)$ is continuous, one has
\begin{eqnarray*}|u_m\log(|u_m|^p)-u\log(|u|^p)|\longrightarrow 0 \ \ \ \textrm{a.e.\ \ in}\ \ \
[0,L]\times [0,T].\end{eqnarray*}
\indent On the other hand, from the Sobolev
embedding $H_{per}^1([0,L])\hookrightarrow L_{per}^{\infty}([0,L])$ we obtain that $|u_m\log(|u_m|^p)-u\log(|u|^p)|$ is
bounded in $L^{\infty}([0,L]\times [0,T])$. Next, taking into
account the Lebesgue Dominated Convergence Theorem, we have
$$u_m\log(|u_m|^{p})\rightarrow u\log(|u|^{p})\ \ \ \textrm{strongly\ in\ }\ \
L^2(0,T;L^2_{per}([0,L])).$$ \indent Finally, one can pass to the limit at the equation (\ref{gl.3}) to obtain the existence of weak solutions
related to (\ref{gl.1}) in a standard form. We can
easily check that the initial conditions are also satisfied. \\

Step 2. \textit{Uniqueness of Solution.}\\

Let $u$ and $v$ be two weak solutions for the Cauchy
problem (\ref{gl.1}). Define $w:=u-v$. We see that
\begin{eqnarray}\label{gl.80}\begin{array}{l}\langle
w_{tt}(\cdot,t),\zeta\rangle_{H^{-1}_{per},H^1_{per}}+\displaystyle\int_0^L\nabla
w(\cdot,t)\cdot\overline{\nabla\zeta}\
dx+\displaystyle\int_0^Lw(\cdot,t)\overline{\zeta}\
dx\\
=\displaystyle\int_0^L[u(\cdot,t)\log(|u(\cdot,t)|^p)-v(\cdot,t)\log(|v(\cdot,t)|^p)]\overline{\zeta}\
dx\ \textrm{a.e.}\ t\in [0,T],\end{array}\end{eqnarray} for all
$\zeta\in H^1_{per}([0,L]).$ Furthermore, $w(\cdot,0)=0$ and
$w_t(\cdot,0)=0$. Our intention is to prove that $w\equiv 0$, that
is, the Cauchy problem in $(\ref{gl.1})$ has an unique weak
solution. To do so, we need to stop with the proof of the proposition for a while to present the following auxiliary lemma:
\begin{lemma}\label{p.bc.gl.01} Consider $0<T_0<\displaystyle\frac{L}{4}$ and $w$ an $L$-periodic function
at the spatial variable. Suppose that $w(\cdot,0)=0$ 
and $f\in L^2(0,T_0;L^2_{per}([0,L]))$. Furthermore, let us assume that
\begin{eqnarray}\label{gl.81}\displaystyle\iint_{[0,L]\times[0,T_0]}w(\zeta_{tt}-\zeta_{xx})\ dx\ dt=\displaystyle\iint_{[0,L]\times [0,T_0]}f\zeta\ dx\
dt,\end{eqnarray} for all smooth real function $\zeta$ defined over
$\mathbb{R}\times [0,T_0]$ which is periodic at the spatial variable for all $t\in [0,T_0]$ and it satifies $\zeta(x,T_0)=\zeta_t(x,T_0)=0$ for all $x\in \mathbb{R}$.
Thus for all $(x,t)\in\mathbb{R}\times
[0,T_0]$ one has
\begin{eqnarray*}w(x,t)=\displaystyle\frac{1}{2}\displaystyle\int_0^{t}\int_{x-t+\tau}^{x+t-\tau}f(y,\tau)\ dy\
d\tau.\end{eqnarray*}
\end{lemma}
\begin{proof}
First of all, it is important to mention that standard methods of
partial differential equations enable us to deduce that the concept
of weak solutions can be rewritten in order to use $(\ref{gl.81})$.
Thus, let $(x,t)\in \mathbb{R}\times (0,T_0]$ be fixed. Define an
$L$-periodic smooth real function
$\chi^{\ast}:\mathbb{R}\rightarrow\mathbb{R}$ as
\begin{eqnarray*}\label{gl.83}\chi^{\ast}(y):=\displaystyle\left\{\begin{array}{l}
                                                        0,\ y\in \left[0,\frac{2t}{5}\right]. \\
                                                        0\leq \chi^{\ast}(y)\leq 1,\ y\in [\frac{2t}{5},\frac{4t}{5}].  \\
                                                        1,\ y\in [\frac{4t}{5},t]. \\
                                                        \chi^{\ast}(2t-y),\ y\in [t,2t]. \\
                                                        0,\ \ y\in [2t,L].
                                                      \end{array}
\right.\end{eqnarray*} \indent In what follows, let us assume that
$\chi^{\ast}$ is a nondecreasing function in the interval
$\left[\frac{2t}{5},\frac{4t}{5}\right]$. Now, let $\varepsilon>0$
be sufficiently small and define the auxiliary $L$-periodic function
\begin{eqnarray*}\label{gl.85}\chi_{\varepsilon}(y):=\displaystyle\left\{\begin{array}{l}
                                                        0,\  y\in [0,\frac{2t\varepsilon}{5}]. \\
                                                        \chi^{\ast}(\frac{y}{\varepsilon}),\ \ y\in [\frac{2t\varepsilon}{5},\frac{4t\varepsilon}{5}].  \\
                                                        1,\ y\in [\frac{4t\varepsilon}{5},t]. \\
                                                        \chi_{\varepsilon}(2t-y),\ y\in [t,2t]. \\
                                                        0,\ \ y\in [2t,L].
                                                      \end{array}
\right.\end{eqnarray*} We have that
\begin{eqnarray}\label{gl.86}\chi_{\varepsilon}(y)\stackrel{\varepsilon\rightarrow
0^{+}}{\longrightarrow}\displaystyle\left\{\begin{array}{c}
                                                                                                                    1,\ y\in (0,2t) \\
                                                                                                                    0,\
                                                                                                                    y\in
                                                                                                                    (2t,L)
                                                                                                                  \end{array}
\right.\end{eqnarray} and
\begin{eqnarray*}\label{gl.87}\chi'_{\varepsilon}(y)=\displaystyle\left\{\begin{array}{l}
                                                        0,\ y\in [0,\frac{2t\varepsilon}{5}]. \\
                                                        \frac{1}{\varepsilon}(\chi^{\ast})'(\frac{y}{\varepsilon}),\ y\in [\frac{2t\varepsilon}{5},\frac{4t\varepsilon}{5}].  \\
                                                        0,\ y\in [\frac{4t\varepsilon}{5},2t-\frac{4t\varepsilon}{5}]. \\
                                                        -\chi'_{\varepsilon}(2t-y)=-\frac{1}{\varepsilon}(\chi^{\ast})'(\frac{2t-y}{\varepsilon}),\ y\in [2t-\frac{4t\varepsilon}{5},2t-\frac{2t\varepsilon}{5}].\\
                                                        0,\ \ y\in
                                                        [2t-\frac{2t\varepsilon}{5},2t].\\
                                                        0,\ \ y\in [2t,L].
                                                      \end{array}\right.\end{eqnarray*}

Letting
\begin{eqnarray*}\zeta_{\varepsilon}(y,\tau):=\chi_{\varepsilon}(y-\tau-x+t)\ \chi_{\varepsilon}(x+t-y-\tau),\
\ (y,\tau)\in \mathbb{R}\times [0,T_0].\end{eqnarray*}

We see that $\zeta_{\varepsilon}$ is a smooth and $L$-periodic
function. Therefore, we deduce from (\ref{gl.81}),
\begin{eqnarray}\label{gl.89}\ \ \ \displaystyle\iint_{[0,L]\times[0,T_0]}w[(\zeta_{\varepsilon})_{tt}-(\zeta_{\varepsilon})_{xx}]\
dy\ d\tau=\displaystyle\iint_{[0,L]\times
[0,T_0]}f\zeta_{\varepsilon}\ dy\ d\tau.\end{eqnarray}

Next, if $0<x-t<x+t<L$ and $t\in (0,T_0]$ we get
$$\{[\textrm{supp}\ (\zeta_{\varepsilon})]\cap [0,L]\times
[0,T_0]\}\subset\{(y,\tau);\ x-t+\tau<y<x+t-\tau,\ 0<\tau<t\},$$
since
$$t<\frac{L}{2}-t \Leftrightarrow
t<\frac{L}{4}\ \ \ \textrm{and}\ \ \ t\leq T_0<\frac{L}{4}.$$



Using (\ref{gl.86}), one can see from Lebesgue Dominated
Convergence Theorem that
\begin{eqnarray}\label{gl.91}&&
\displaystyle\lim_{\varepsilon\rightarrow
0^{+}}\displaystyle\iint_{[0,L]\times [0,T_0]}f\zeta_{\varepsilon}\
dy\
d\tau=\displaystyle\int_0^{t}\displaystyle\int_{x-t+\tau}^{x+t-\tau}f(y,\tau)\
dy\ d\tau.\end{eqnarray} By considering other values of $(x,t)$, it
is possible to deduce a similar situation as in (\ref{gl.91}) since
function $\zeta_{\varepsilon}$ is periodic at the spatial variable.

On the other hand,
\begin{eqnarray*}\displaystyle\iint_{[0,L]\times[0,T_0]}w[(\zeta_{\varepsilon})_{tt}-(\zeta_{\varepsilon})_{xx}]\
dy\
d\tau=(I)_{\varepsilon}+(II)_{\varepsilon}+(III)_{\varepsilon},\end{eqnarray*}
where
\begin{eqnarray*}(I)_{\varepsilon}&:=&4\displaystyle\int_{t-\frac{3t\varepsilon}{5}}^{t-\frac{2t\varepsilon}{5}}\displaystyle\int_{x-t+\tau+\frac{2t\varepsilon}{5}}^{x+t-\tau-\frac{2t\varepsilon}{5}}\chi'_{\varepsilon}(y-\tau-x+t)\ \chi'_{\varepsilon}(x+t-y-\tau)\
w(y,\tau)\ dy\ d\tau\nonumber\\
\nonumber\\
&+&4\displaystyle\int_{t-\frac{4t\varepsilon}{5}}^{t-\frac{3t\varepsilon}{5}}\displaystyle\int_{x+t-\tau-\frac{4t\varepsilon}{5}}^{x-t+\tau+\frac{4t\varepsilon}{5}}\chi'_{\varepsilon}(y-\tau-x+t)\
\chi'_{\varepsilon}(x+t-y-\tau)\
w(y,\tau)\ dy\ d\tau \nonumber\\
\nonumber\\
&=&2w\displaystyle\left(x,t-\displaystyle\frac{4t\varepsilon}{5}\right)+{O}(\varepsilon),\nonumber
\end{eqnarray*}
\begin{eqnarray*}(II)_{\varepsilon}&:=&4\displaystyle\int_{0}^{\frac{t\varepsilon}{5}}\displaystyle\int_{x-t+\tau+\frac{2t\varepsilon}{5}}^{x-t-\tau+\frac{4t\varepsilon}{5}}\chi'_{\varepsilon}(y-\tau-x+t)\ \chi'_{\varepsilon}(x+t-y-\tau)\ w(y,\tau)\ dy\
d\tau\nonumber\\
\nonumber\\
&=&{O}(\varepsilon)-2\displaystyle\int_{\frac{2t}{5}}^{\frac{4t}{5}}(\chi^{\ast})'(\kappa)\
\chi^{\ast}\displaystyle\left(\kappa\right)\
w\displaystyle\left(x-t+{\varepsilon \kappa},0\right)\
d\kappa\nonumber
\end{eqnarray*} and
\begin{eqnarray*}(III)_{\varepsilon}&=&4\displaystyle\int_{0}^{\frac{t\varepsilon}{5}}\displaystyle\int_{x+t+\tau-\frac{4t\varepsilon}{5}}^{x+t-\tau-\frac{2t\varepsilon}{5}}\chi'_{\varepsilon}(y-\tau-x+t)\ \chi'_{\varepsilon}(x+t-y-\tau)\ w(y,\tau)\ dy\
d\tau\nonumber\\
\nonumber\\
&=&{O}(\varepsilon)-2\displaystyle\int_{\frac{2t}{5}}^{\frac{4t}{5}}(\chi^{\ast})'(\kappa)\
\chi^{\ast}\displaystyle\left(\kappa\right)\
w\displaystyle\left(x+t-{\varepsilon \kappa},0\right)\
d\kappa.\nonumber
\end{eqnarray*}

Since $w(\cdot,0)=0$ we deduce for $\varepsilon\rightarrow 0^{+}$
that
\begin{eqnarray}\label{gl.96}\displaystyle\lim_{\varepsilon\rightarrow
0^{+}}\displaystyle\iint_{[0,L]\times[0,T_0]}w[(\zeta_{\varepsilon})_{tt}-(\zeta_{\varepsilon})_{xx}]\
dy\ d\tau=2w(x,t).\end{eqnarray} The proof of the proposition is completed by combining
(\ref{gl.89}), (\ref{gl.91}) and (\ref{gl.96}). 

\end{proof}

Next, we continue with the proof of the Proposition $\ref{t.bc.gl.1}$ and we are going to follow the arguments in \cite{cazenave}. Indeed, let us
consider $$f:=-(u-v)+[u\log(|u|^p)-v\log(|v|^p)]\in
L^2(0,T;L^2_{per}([0,L])).$$ Assume that $T_0\leq T$, where
$0<T_0<\displaystyle\frac{L}{4}$. Since $w$ satisfies (\ref{gl.80})
with $w(\cdot,0)=0$ and $w_t(\cdot,0)=0$, we can apply the
characterization of the functionals in the dual space of
$H_{per}^1([0,L])$ in order to deduce that $w$ satisfies
(\ref{gl.81}). This last fact enables us to apply Lemma
\ref{p.bc.gl.01} to obtain
\begin{eqnarray}\label{gl.97} w(x,t)&=&-\displaystyle\frac{1}{2}\displaystyle\int_0^t\displaystyle\int_{x-t+\tau}^{x+t-\tau}[u(y,\tau)-v(y,\tau)]\ dy\ d\tau\nonumber\\
\\
&+&\displaystyle\frac{p}{2}\displaystyle\int_0^t\displaystyle\int_{x-t+\tau}^{x+t-\tau}[u(y,\tau)\log(|u(y,\tau)|)-v(y,\tau)\log(|v(y,\tau)|)]\
dy\ d\tau,\nonumber
\end{eqnarray} for all $(x,t)\in\mathbb{R}\times [0,T_0]$.\\
\indent If $(y,\tau)\in\mathbb{R}\times [0,T_0]$ and $|u(y,\tau)|\leq
|v(y,\tau)|$, we can use Proposition \ref{p.bc.gl.02} to get
\begin{eqnarray}\label{gl.99}\ \ \ \  \ \ |u(y,\tau)\log(|u(y,\tau)|)-v(y,\tau)\log(|v(y,\tau)|)|\leq [1+|\log(|v(y,\tau)|)|]\ |u(y,\tau)-v(y,\tau)|.
\end{eqnarray}

First, let us suppose $|v(y,\tau)|\leq 1$. So,
$|v(y,\tau)-u(y,\tau)|\leq |v(y,\tau)|+|u(y,\tau)|\leq 2
|v(y,\tau)|$ and
\begin{eqnarray}\label{gl.100}|\log(|v(y,\tau)|)|\leq |\log(|w(y,\tau)|)|+\log
2.\end{eqnarray} Now, if $|v(y,\tau)|\geq 1$ then
\begin{eqnarray}\label{gl.101}|\log(|v(y,\tau)|)|\leq
1+|v(y,\tau)|.\end{eqnarray}
\indent Using (\ref{gl.99}), (\ref{gl.100})
and (\ref{gl.101}), we have
\begin{eqnarray}\label{gl.102}\begin{array}{l}|u(y,\tau)\log(|u(y,\tau)|)-v(y,\tau)\log(|v(y,\tau)|)|\\
\\
\leq [2+\log 2+|v(y,\tau)|+|\log(|w(y,\tau)|)|]\
|w(y,\tau)|.\end{array}
\end{eqnarray}
\indent On the other hand, if we assume
$|v(y,\tau)|\leq |u(y,\tau)|,$ then
\begin{eqnarray}\label{gl.103}\begin{array}{l}|u(y,\tau)\log(|u(y,\tau)|)-v(y,\tau)\log(|v(y,\tau)|)|\\
\\
\leq [2+\log 2+|u(y,\tau)|+|\log(|w(y,\tau)|)|]\
|w(y,\tau)|.\end{array}
\end{eqnarray} Inequalities
(\ref{gl.102}) and (\ref{gl.103}) can be applied to establish
\begin{eqnarray}\label{gl.104}\begin{array}{l}|u(y,\tau)\log(|u(y,\tau)|)-v(y,\tau)\log(|v(y,\tau)|)|\\
\\
\leq [4+\log
4+|u(y,\tau)|+|v(y,\tau)|+2|\log(|w(y,\tau)|)|]\ |w(y,\tau)|\\
\\
\leq C_{13}\ [1+|\log(|w(y,\tau)|)|]\ |w(y,\tau)|,\end{array}
\end{eqnarray} where $C_{13}>0$ is a constant. The existence of $C_{13}>0$ is obtained since we have $$|u(y,\tau)|\leq \|u\|_{L^{\infty}([0,L]\times [0,T])}\ \ \textrm{and}
\ \ |v(y,\tau)|\leq \|v\|_{L^{\infty}([0,L]\times [0,T])}.$$

From (\ref{gl.97}) and (\ref{gl.104}) we have
$(x,t)\in \mathbb{R}\times [0,T_0]$,
\begin{eqnarray}\label{gl.104.1}|w(x,t)|&\leq&
\displaystyle\frac{1}{2}\displaystyle\int_0^t\displaystyle\int_{x-t+\tau}^{x+t-\tau}|u(y,\tau)-v(y,\tau)|\
dy\ d\tau\nonumber\\
\nonumber\\
&+&\frac{pC_{13}}{2}\displaystyle\int_0^t\displaystyle\int_{x-t+\tau}^{x+t-\tau}[1+|\log(|w(y,\tau)|)|]\
|w(y,\tau)|\ dy\ d\tau\\
\nonumber\\
&\leq&(1+pC_{13})T_0\displaystyle\int_0^t\|w(\cdot,\tau)\|_{L_{per}^{\infty}}\
d\tau+pC_{13}T_0\displaystyle\int_0^t\|w(\cdot,\tau)\log(|w(\cdot,\tau)|)\|_{L_{per}^{\infty}}\
d\tau.\nonumber
\end{eqnarray}

Let us define the constant $C_{14}:=\frac{(1+pC_{13})L}{4}$.
Inequality (\ref{gl.104.1}) gives us
\begin{eqnarray}\label{gl.105}\ \ \ \ \|w(\cdot,t)\|_{L_{per}^{\infty}}\leq C_{14}\displaystyle\int_0^t\|w(\cdot,\tau)\|_{L_{per}^{\infty}}\ d\tau+\displaystyle
C_{14}\displaystyle\int_0^t\|w(\cdot,\tau)\log(|w(\cdot,\tau)|)\|_{L_{per}^{\infty}}\
d\tau,
\end{eqnarray} for all $t\in [0,T_0]$. Hence, using
(\ref{gl.105}) and the fact that $w$ is bounded on $[0,L]\times
[0,T]$, there exists $T_1\in (0,T_0)$ such that
$\|w(\cdot,t)\|_{L_{per}^{\infty}}<\displaystyle\frac{1}{e}\
\textrm{for\ all\ } t\in [0,T_1].$

Next, define the function $\tilde{F}(\beta)=|\beta|\log\ (|\beta|)$\
for all $\beta\in\mathbb{C}.$ Since $|\beta|<-\tilde{F}(\beta)\
\textrm{for\ all\ } |\beta|\in \left(0,\frac{1}{e}\right),$ one has
\begin{eqnarray*}\displaystyle\int_0^t\|w(\cdot,\tau)\|_{L^{\infty}(\mathbb{R})}\
d\tau \leq
-\displaystyle\int_0^t\|w(\cdot,\tau)\|_{L_{per}^{\infty}}\log(\|w(\cdot,\tau)\|_{L_{per}^{\infty}})\
d\tau.\end{eqnarray*} Thus, from (\ref{gl.105}), we see that
\begin{eqnarray*}\|w(\cdot,t)\|_{L_{per}^{\infty}}
\leq-2C_{14}\displaystyle\int_0^t\|w(\cdot,\tau)\|_{L_{per}^{\infty}}\log(\|w(\cdot,\tau)\|_{L_{per}^{\infty}})\
d\tau\ \textrm{ for\ all\ } t\in [0,T_1].
\end{eqnarray*} Using Proposition \ref{t.bc.gl.5}, we
finally conclude $\|w(\cdot,t)\|_{L_{per}^{\infty}}=0\ \
\textrm{a.e.}\ \ t\in [0,T_1].$

Let us prove the result over $[0,T]$. Consider the auxiliary
functions
$$\widetilde{w}(\cdot,\cdot):={w}(\cdot,\cdot+T_1),\ \
\widetilde{u}(\cdot,\cdot):={u}(\cdot,\cdot+T_1)\ \ \textrm{and}\ \
\widetilde{v}(\cdot,\cdot):={v}(\cdot,\cdot+T_1).$$

The function $\widetilde{w}$ is a weak solution to the Cauchy
problem
\begin{eqnarray*}\displaystyle\left\{\begin{array}{l}
                                       \widetilde{w}_{tt}-\widetilde{w}_{xx}=-(\widetilde{u}-\widetilde{v})+p[\log(|\widetilde{u}|)\widetilde{u}-\log(|\widetilde{v}|)\widetilde{v}],\ (x,t)\in\mathbb{R}\times (0,T_1).\\
                                       \widetilde{w}(x,0)=0,\ x\in\mathbb{R}. \\
                                       \widetilde{w}_{t}(x,0)=0,\ x\in\mathbb{R}.
                                     \end{array}
\right.\end{eqnarray*} So, \begin{eqnarray*} \widetilde{w}(x,t)&=&-\displaystyle\frac{1}{2}\displaystyle\int_0^t\displaystyle\int_{x-t+\tau}^{x+t-\tau}[\widetilde{u}(y,\tau)-\widetilde{v}(y,\tau)]\ dy\ d\tau\nonumber\\
\nonumber\\
&+&\displaystyle\frac{p}{2}\displaystyle\int_0^t\displaystyle\int_{x-t+\tau}^{x+t-\tau}[\widetilde{u}(y,\tau)\log(|\widetilde{u}(y,\tau)|)-\widetilde{v}(y,\tau)\log(|\widetilde{v}(y,\tau)|)]\
dy\ d\tau,\nonumber
\end{eqnarray*} for all $(x,t)\in\mathbb{R}\times [0,T_1]$. Similar argument above gives us
$\|w(\cdot,t)\|_{L_{per}^{\infty}}=0\ \ \textrm{a.e.}\ \ t\in
[T_1,2T_1].$

Therefore, we deduce the uniqueness for all values of $T>0$ using an
interaction argument. We have completed the proof of Proposition
$\ref{t.bc.gl.1}$.
\begin{flushright}
$\square$
\end{flushright}

\indent Now we prove that our weak solution obtained in Proposition $\ref{t.bc.gl.1}$ satisfies equalities in $(\ref{cons-quant-1})$. Thus, Theorem $\ref{t.bc.gl.1.1}$ can be proven by combining both results.

\begin{proposition}\label{prop-cons-quant}
The weak solution obtained in Proposition $\ref{t.bc.gl.1}$  satisfies the following conserved quantities:
$$\mathcal{E}(u(\cdot,t),u'(\cdot,t))=\mathcal{E}(u_{0},u_{1})\
\ \mbox{and}\ \
\mathcal{F}(u(\cdot,t),u'(\cdot,t))=\mathcal{F}(u_{0},u_{1})\ \ \ \ \ \mbox{a.e.}\ t\in [0,T],$$
where $\mathcal{E}$ and $\mathcal{F}$ are defined as in $(\ref{gl.108.1})$ and $(\ref{gl.109.1})$.
\end{proposition}

\begin{proof}
First of all, we prove that the conserved quantity
(\ref{gl.108.1}) holds. Indeed, the construction of $u_m$ and
(\ref{gl.23}) gives us
\begin{eqnarray}\label{gl.110}\mathcal{E}(u_m(t),u_m'(t))=\mathcal{E}(u_{0,m},u_{1,m}).\end{eqnarray}

Consider $\vartheta\in C^0([0,T])$, $\vartheta\geq 0$. Multiplying
the identity (\ref{gl.110}) by the function $\vartheta$ and
integrating the result over interval $[0,T]$,
\begin{eqnarray}\label{gl.111}\displaystyle\int_0^T\mathcal{E}(u_m(t),u_m'(t))\vartheta(t)\
dt=\displaystyle\int_0^T\mathcal{E}(u_{0,m},u_{1,m})\vartheta(t)\
dt.\end{eqnarray} Next, since
$u_{0,m}\stackrel{m\rightarrow\infty}{\longrightarrow} u_0$ in
$H^1_{per}([0,L])$ and
$u_{1,m}\stackrel{m\rightarrow\infty}{\longrightarrow} u_1$ in
$L^2_{per}([0,L])$, we get
\begin{eqnarray}\label{gl.114}\displaystyle\int_0^T\mathcal{E}(u_{0,m},u_{1,m})\vartheta(t)\ dt\stackrel{m\rightarrow\infty}{\longrightarrow}\displaystyle\int_0^T\mathcal{E}(u_0,u_1)\vartheta(t)\ dt.\end{eqnarray}
\indent In addition, using (\ref{gl.55}),
\begin{eqnarray}\label{gl.118}\displaystyle\int_0^T\vartheta(t)\displaystyle\int_0^L|u_m(x,t)|^2 dx\
dt\stackrel{m\rightarrow\infty}{\longrightarrow}\displaystyle\int_0^T\vartheta(t)\displaystyle\int_0^L
|u(x,t)|^2\ dx\ dt
\end{eqnarray} and
\begin{eqnarray}\label{gl.119} \ \ \ \ \ \ \
 \displaystyle\int_0^T\vartheta(t)\displaystyle\int_0^L|u_m(x,t)|^2\log(|u_m(x,t)|^p)
dx\
dt\stackrel{m\rightarrow\infty}{\longrightarrow}\displaystyle\int_0^T\vartheta(t)\displaystyle\int_0^L
|u(x,t)|^2\log(|u(x,t)|^p) dx\ dt.
\end{eqnarray}
\indent Next, the first limit in (\ref{gl.50}) enables us to deduce
$$\sqrt{\vartheta}u_m\rightharpoonup \sqrt{\vartheta}u\ \
\textrm{weakly\ in\ } L^2(0,T;H^1_{per}([0,L])).$$ So, Fatou's Lemma
applied to last weak convergence gives us
\begin{eqnarray}\label{gl.121}\displaystyle\int_0^T\vartheta(t)\|u(t)\|^2_{H^1_{per}}\
dt\leq \displaystyle\liminf_{m\rightarrow
\infty}\displaystyle\int_0^T\vartheta(t)\|u_m(t)\|^2_{H^1_{per}}\
dt.\end{eqnarray} \indent Similarly, using the second limit in
(\ref{gl.50}), we get
\begin{eqnarray}\label{gl.122}\displaystyle\int_0^T\vartheta(t)\|u'(t)\|^2_{L^2_{per}}\
dt\leq
\displaystyle\liminf_{m\rightarrow\infty}\displaystyle\int_0^T\vartheta(t)\|u_m'(t)\|^2_{L^2_{per}}\
dt.\end{eqnarray} From (\ref{gl.111}), (\ref{gl.114}),
(\ref{gl.118}), (\ref{gl.119}), (\ref{gl.121}) and (\ref{gl.122}),
we conclude the following inequality
\begin{eqnarray*}\displaystyle\int_0^T\mathcal{E}(u(\cdot,t),u_t(\cdot,t))\vartheta(t)\
dt\leq \displaystyle\int_0^T\mathcal{E}(u_0,u_1)\vartheta(t)\
dt.\end{eqnarray*} \indent Therefore, using the fact that
function $\vartheta$ is arbitrary, we obtain
\begin{eqnarray}\label{gl.124}\mathcal{E}(u(\cdot,t),u_t(\cdot,t))\leq
\mathcal{E}(u_0,u_1)\ \textrm{a.e.}\ t\in [0,T].\end{eqnarray}

On the other hand, consider $s$, $t$ satisfying $0<s<t<T$. Since $u$ is a weak solution, one has
$$u_{tt}(\cdot,\xi)-u_{xx}(\cdot,\xi)+u(\cdot,\xi)-u(\cdot,\xi)\log(|u(\cdot,\xi)|^p)=0\
\ \textrm{in}\ \ H^{-1}_{per}([0,L]),$$ a.e. $\xi\in [0,T]$.
Furthermore,
\begin{eqnarray}\label{gl.125}u_{tt}-u_{xx}+u-u\log(|u|^p)=0\ \ \textrm{in}\ \
L^2(0,T;H^{-1}_{per}([0,L])).\end{eqnarray}

Let us consider $n\in\mathbb{N}$, such that
$n>\max\left\{\frac{1}{s},\frac{1}{T-t}\right\}$. Define function
$\vartheta_n$ as
\begin{eqnarray*}\label{gl.126}\vartheta_n(\xi)=\displaystyle\left\{\begin{array}{l}
                                                    0,\ 0\leq\xi\leq s-\displaystyle\frac{1}{n}. \\
                                                    1+n(\xi-s),\ s-\displaystyle\frac{1}{n}\leq\xi\leq s. \\
                                                    1,\ s\leq\xi\leq t. \\
                                                    1-n(\xi-t),\ t\leq\xi\leq t+\displaystyle\frac{1}{n}. \\
                                                    0,\ t+\displaystyle\frac{1}{n}\leq\xi\leq T.
                                                  \end{array}
\right.\end{eqnarray*}

Next, let $\{\rho_k\}_{k\in\mathbb{N}}\subset
C^{\infty}_0(\mathbb{R})$ be the standard mollifier, where $\rho_k$
is an even function satisfying $\textrm{supp}\ (\rho_k)\
\subset\left[-\frac{1}{k},\frac{1}{k}\right]$ for all
$k\in\mathbb{N}$. Consider
$k>\max\left\{\frac{2n}{ns-1},\frac{2n}{Tn-tn-1}\right\}$ and define
function
$\Theta_{n,k}:=\vartheta_n[(\vartheta_nu')\ast\rho_k\ast\rho_k].$
Here, the symbol $\ast$ denotes the usual convolution in
time-variable. We have the following property
\begin{eqnarray*}\textrm{supp}\ (\Theta_{n,k})\subset\displaystyle\left[s-\displaystyle\frac{1}{n},t+\displaystyle\frac{1}{n}\right]+\displaystyle\left[-\displaystyle\frac{2}{k},\displaystyle\frac{2}{k}\right]\subset(0,T).
\end{eqnarray*} One has
$$(u\vartheta_n)\ast\rho_k\ast\rho_k'=(u\vartheta_n)'\ast\rho_k\ast\rho_k=(u'\vartheta_n)\ast\rho_k\ast\rho_k+(u\vartheta_n')\ast\rho_k\ast\rho_k$$
and
$$\Theta_{n,k}=\vartheta_n[(u\vartheta_n)\ast\rho_k\ast\rho_k'-(u\vartheta_n')\ast\rho_k\ast\rho_k]\in
L^2(0,T;H^1_{per}([0,L])).$$

Hence, from the identity given in (\ref{gl.125}), we get
\begin{eqnarray}\label{gl.131}\begin{array}{l}\displaystyle\int_0^T\langle u''(\cdot,\xi),\Theta_{n,k}(\xi)\rangle_{H^{-1}_{per},H^1_{per}}\ d\xi-\displaystyle\int_0^T\langle
u_{xx}(\cdot,\xi),\Theta_{n,k}(\xi)\rangle_{H^{-1}_{per},H^1_{per}}\
d\xi\\
\\
+\displaystyle\int_0^T\langle
u(\cdot,\xi),\Theta_{n,k}(\xi)\rangle_{L^{2}_{per},L^2_{per}}\
d\xi=\displaystyle\int_0^T\langle
u(\cdot,\xi)\log(|u(\cdot,\xi)|^p),\Theta_{n,k}(\xi)\rangle_{L^{2}_{per},L^2_{per}}\
d\xi.\end{array}
\end{eqnarray}

In what follows, we will omit temporal variables. A similar procedure as in \cite[Theorem 1.6]{lions} establishes a convenient expression to the first term in the left-hand side of $(\ref{gl.131})$ as
\begin{eqnarray}\label{gl.131.01}\displaystyle\int_0^T\langle u'',\Theta_{n,k}\rangle_{H^{-1}_{per},H^1_{per}}\
d\xi&=& \displaystyle\int_0^T\langle (\vartheta_n
u')'\ast{\rho_k},(\vartheta_nu')\ast\rho_k\rangle_{H^{-1}_{per},H^1_{per}}\
d\xi\nonumber\\
\\
&-&\displaystyle\int_0^T\langle (\vartheta_n'
u')\ast{\rho_k},(\vartheta_nu')\ast\rho_k\rangle_{H^{-1}_{per},H^1_{per}}\
d\xi.\nonumber
\end{eqnarray}

In addition, since
\begin{eqnarray}\label{gl.132}(\vartheta_n'u')\ast\rho_k\stackrel{k\rightarrow\infty}{\longrightarrow}\vartheta_n'u'\
\ \textrm{ in}\ \ L^2(0,T;L^2_{per}([0,L]))\end{eqnarray} and
\begin{eqnarray}\label{gl.133}(\vartheta_n u')\ast\rho_k\stackrel{k\rightarrow\infty}{\longrightarrow}\vartheta_n u'\
\ \textrm{ in}\ \ L^2(0,T;L^2_{per}([0,L])),\end{eqnarray} we
deduce, from (\ref{gl.132}) and (\ref{gl.133}),
\begin{eqnarray}\label{gl.135}&&\displaystyle\lim_{k\rightarrow\infty}\displaystyle\int_0^T\langle (\vartheta_n'
u')\ast{\rho_k},(\vartheta_nu')\ast\rho_k\rangle_{H^{-1}_{per},H^1_{per}}\
d\xi\nonumber\\
\nonumber\\
&&=\displaystyle\lim_{k\rightarrow\infty}\displaystyle\int_0^T\langle
(\vartheta_n'
u')\ast{\rho_k},(\vartheta_nu')\ast\rho_k\rangle_{L^{2}_{per},L^2_{per}}\
d\xi\\
\nonumber\\
&&=\displaystyle\int_0^T\langle \vartheta_n'
u',\vartheta_nu'\rangle_{L^{2}_{per},L^2_{per}}\
d\xi=\displaystyle\int_0^T\vartheta_n\vartheta_n'\|u'(\cdot,\xi)\|^2_{L^2_{per}}\
d\xi.\nonumber
\end{eqnarray}

Furthermore, since $\textrm{supp}\
[(\vartheta_nu')\ast\rho_k]\subset
\left[s-\frac{1}{n},t+\frac{1}{n}\right]+\left[-\frac{1}{k},\frac{1}{k}\right]\subset
(0,T),$ we obtain
 \begin{eqnarray}\label{gl.136}\ \ \ \ \ \
\textrm{Re}\displaystyle\left(\displaystyle\int_0^T\displaystyle\langle(\vartheta_nu')'\ast\rho_k,(\vartheta_nu')\ast\rho_k\rangle_{H^{-1}_{per},H^1_{per}}
d\xi\right)=\displaystyle\frac{1}{2}\displaystyle\int_0^T\displaystyle\frac{\partial}{\partial
\xi}\displaystyle\left[\|(\vartheta_nu')\ast\rho_k\|_{L^2_{per}}^2\right]
d\xi=0.\end{eqnarray}

From (\ref{gl.131.01}), (\ref{gl.135}) and (\ref{gl.136}), we see that
\begin{eqnarray}\label{gl.138}&&\displaystyle\lim_{k\rightarrow\infty}\textrm{Re}\displaystyle\left(\displaystyle\int_0^T\langle
u'',\Theta_{n,k}\rangle_{H^{-1}_{per},H^1_{per}}
d\xi\right)=-\displaystyle\int_0^T\vartheta_n\vartheta_n'\|u'(\cdot,\xi)\|^2_{L^2_{per}}
d\xi.
\end{eqnarray} Moreover,
\begin{eqnarray}\label{gl.139}\begin{array}{l}\ \ \ \ -\displaystyle\int_0^T\langle
u_{xx},\Theta_{n,k}\rangle_{H^{-1}_{per},H^1_{per}}\
d\xi+\displaystyle\int_0^T\langle
u,\Theta_{n,k}\rangle_{L^{2}_{per},L^2_{per}}\
d\xi\\
\\
\ \ \ \ =\displaystyle\int_0^T\langle
u,\vartheta_n[(\vartheta_nu')\ast\rho_k\ast\rho_k]\rangle_{H^{1}_{per},H^1_{per}}\
d\xi\\
\\
\ \ \ \ =\displaystyle\int_0^T\langle
(\vartheta_nu)\ast{\rho_k},(\vartheta_nu)\ast\rho_k'\rangle_{H^{1}_{per},H^1_{per}}\
d\xi-\displaystyle\int_0^T\langle
(\vartheta_nu)\ast{\rho_k},(\vartheta_n'u)\ast\rho_k\rangle_{H^{1}_{per},H^1_{per}}\
d\xi.\end{array}
\end{eqnarray} \indent On the left-hand-side of $(\ref{gl.139})$, it is possible to use a similar argument as in (\ref{gl.138}) to obtain
\begin{eqnarray}\label{gl.145}\begin{array}{l}\displaystyle\lim_{k\rightarrow\infty}\textrm{Re}\displaystyle\left(-\displaystyle\int_0^T\langle
u_{xx},\Theta_{n,k}\rangle_{H^{-1}_{per},H^1_{per}}\
d\xi+\displaystyle\int_0^T\langle
u,\Theta_{n,k}\rangle_{L^{2}_{per},L^2_{per}}\
d\xi\right)\\
\\
=-\displaystyle\int_0^T\vartheta_n\vartheta_n'\|u(\cdot,\xi)\|^2_{H^1_{per}}\
d\xi.\end{array}
\end{eqnarray}
Next,
\begin{eqnarray}\label{gl.147}\displaystyle\lim_{k\rightarrow\infty}\displaystyle\int_0^T\langle u\log(|u|^p),\Theta_{n,k}\rangle_{L^2_{per},L^2_{per}}\
d\xi=\displaystyle\int_0^T\vartheta_n^2\langle
u\log(|u|^p),u'\rangle_{L^2_{per},L^2_{per}}\ d\xi.
\end{eqnarray}

Hence, collecting the results in (\ref{gl.131}), (\ref{gl.138}), (\ref{gl.145}) and
(\ref{gl.147}), we have for $k$ large enough,
\begin{eqnarray}\label{gl.148}\begin{array}{l}\ \ \ \ \ \ -\displaystyle\int_0^T\vartheta_n\vartheta_n'\displaystyle\left(\|u'(\cdot,\xi)\|_{L^2_{per}}^2+\|u(\cdot,\xi)\|_{H^1_{per}}^2\right) d\xi - \displaystyle\frac{p}{2}\displaystyle\int_0^T\vartheta_n(\xi)\vartheta_n'(\xi)\displaystyle\left[\displaystyle\int_0^L|u(x,\xi)|^2\
dx\right]d\xi\\
\\
\ \ \ \ \ \
=-\displaystyle\frac{p}{2}\displaystyle\int_0^T\vartheta_n(\xi)\vartheta'_n(\xi)\displaystyle\left[\displaystyle\int_0^L|u(x,\xi)|^2\log(|u(x,\xi)|^2)\
dx \right]d\xi.\end{array}
\end{eqnarray}
Now for $\Lambda\in L^1([0,T])$ one has
\begin{eqnarray}\label{gl.148.1}-\displaystyle\int_0^T\vartheta_n\vartheta_n'\Lambda\
d\xi\stackrel{n\rightarrow\infty}{\longrightarrow}\displaystyle\frac{1}{2}\Lambda(t)-\displaystyle\frac{1}{2}\Lambda(s)\end{eqnarray}
(see \cite[pg. 25]{lions}). Convergence in (\ref{gl.148.1}) combined
with (\ref{gl.148}) enables us to conclude
\begin{eqnarray}\label{gl.152}\mathcal{E}(u(\cdot,t),u'(\cdot,t))=\mathcal{E}(u(\cdot,s),u'(\cdot,s)),\
\ 0<s<t<T.\end{eqnarray}


Since $u\in C^0([0,T];L^2_{per}([0,L]))$, we obtain
\begin{eqnarray}\label{gl.153}\displaystyle\int_0^L|u(x,s)|^2\ dx\stackrel{s\rightarrow 0^{+}}{\longrightarrow}\displaystyle\int_0^L|u(x,0)|^2\
dx=\displaystyle\int_0^L|u_0(x)|^2\ dx.\end{eqnarray} In addition,
\begin{eqnarray}\label{gl.155} \ \ \displaystyle\int_0^L|u(x,s)|^2\log(|u(x,s)|^p) dx\stackrel{s\rightarrow 0^{+}}{\longrightarrow}
\displaystyle\int_0^L|u_0(x)|^2\log(|u_0(x)|^p) dx.\end{eqnarray}

On the other hand, the fact that $u\in C_s(0,T;H^1_{per}([0,L]))$
and Fatou's Lemma imply that
\begin{eqnarray}\label{gl.156}\|u_0\|^2_{H^1_{per}}=\|u(\cdot,0)\|^2_{H^1_{per}}\leq\displaystyle\liminf_{s\rightarrow
0^{+}}\|u(\cdot,s)\|^2_{H^1_{per}}.
\end{eqnarray} Similarly, since $u'\in C_s(0,T;L^2_{per}([0,L]))$, we get
\begin{eqnarray}\label{gl.157}\|u_1\|^2_{L^2_{per}}=\|u'(\cdot,0)\|^2_{L^2_{per}}\leq\displaystyle\liminf_{s\rightarrow
0^{+}}\|u'(\cdot,s)\|^2_{L^2_{per}}.
\end{eqnarray}

From relations (\ref{gl.152}), (\ref{gl.153}), (\ref{gl.155}),
(\ref{gl.156}) and (\ref{gl.157}), we obtain that
\begin{eqnarray}\label{gl.158}\ \ \ \ \mathcal{E}(u_0,u_1)\leq \displaystyle\liminf_{s\rightarrow
0^{+}}\mathcal{E}(u(\cdot,s),u'(\cdot,s))=\mathcal{E}(u(\cdot,t),u'(\cdot,t))\
\ \textrm{a.e.}\ \ t\in [0,T].
\end{eqnarray} Therefore, from
(\ref{gl.124}) and (\ref{gl.158})
$\mathcal{E}(u(\cdot,t),u'(\cdot,t))=\mathcal{E}(u_0,u_1)\
 \textrm{a.e.}\ t\in [0,T].$ So, $\mathcal{E}$ is a conserved quantity.\\

\indent The next step is to prove the conserved quantity in
(\ref{gl.109.1}). Consider a function $\tilde{\vartheta}\in
\mathcal{D}([0,T])$. We recall that function $u_m$ satisfies the
identity
\begin{eqnarray*}u_m''(t)-\Delta u_{m}(t)+u_{m}(t)-u_m(t)\log(|u_m(t)|^p)=0\ \ \textrm{in}\ \
V_m\ \textrm{a.e.}\ t\in [0,T].\end{eqnarray*} Since
$\{\omega_\nu\}_{\nu\in\mathbb{N}}\subset H^3_{per}([0,L])$, we have
\begin{eqnarray}\label{gl.170}(u_m)_{tt}-(u_m)_{xx}+u_m-u_m
\log(|u_m|^p)=0\end{eqnarray} a.e. $(x,t)\in [0,L]\times [0,T]$. We
claim that
\begin{eqnarray}\label{gl.171}\textrm{Im}\
\displaystyle\int_0^L\overline{u_m(x,t)}u_m'(x,t)\ dx=\textrm{Im}\
\displaystyle\int_0^L\overline{u_{0,m}(x)}u_{1,m}(x)\ dx\
\textrm{a.e.}\ t\in [0,T].
\end{eqnarray} Indeed, from (\ref{gl.170}) we have that
(\ref{gl.171}) occurs since
\begin{eqnarray}\label{gl.172}
\ \displaystyle\int_0^L [\textrm{Re}\ {u_m(x,t)} \textrm{Im}\
u_m''(x,t)-\textrm{Im}\ {u_m(x,t)}\textrm{Re}\ u_m''(x,t)] dx=0\
\textrm{a.e.}\ t\in [0,T],\end{eqnarray} and identities
(\ref{gl.171}) and (\ref{gl.172}) are equivalent.

Multiplying identity (\ref{gl.171}) by $\tilde{\vartheta}$ and
integrating the result over the interval $[0,T]$, we deduce for $m$
large,
\begin{eqnarray}\label{gl.178}
\displaystyle\int_0^T\tilde{\vartheta}(t)\displaystyle\left[\textrm{Im}\
\displaystyle\int_0^L\overline{u(x,t)}u'(x,t)\ dx\right]
dt=\displaystyle\int_0^T\tilde{\vartheta}(t)\displaystyle\left[\textrm{Im}\
\displaystyle\int_0^L\overline{u_0(x)}u_1(x)\ dx\right] dt.\
\end{eqnarray}

Finally, using (\ref{gl.178}) and since function
$\tilde{\vartheta}$ is arbitrary, we have that
\begin{eqnarray*}\textrm{Im}\
\displaystyle\int_0^L\overline{u(x,t)}u'(x,t)\ dx=\textrm{Im}\
\displaystyle\int_0^L\overline{u_0(x)} u_1(x)\ dx\
 \textrm{a.e.}\ t\in [0,T].
\end{eqnarray*} Therefore, $\mathcal{F}$ is also a conserved quantity.\\
\end{proof}

\begin{remark}\label{o.t.bc.gl.5}\textit{ A similar
argument as determined in the proof of Proposition $\ref{prop-cons-quant}$ to prove that $\mathcal{E}$ is a conserved quantity enables us to deduce
\begin{eqnarray*}\|u(\cdot,t_0)\|^2_{H^1_{per}}\leq\displaystyle\liminf_{s\rightarrow
t_0}\|u(\cdot,s)\|^2_{H^1_{per}},\ \ \
\|u_t(\cdot,t_0)\|^2_{L^2_{per}}\leq\displaystyle\liminf_{s\rightarrow
t_0}\|u_t(\cdot,s)\|^2_{L^2_{per}}\end{eqnarray*} and
\begin{eqnarray*}\displaystyle\lim_{s\rightarrow
t_0}\|u(\cdot,s)\|^2_{L^2_{per}}=\|u(\cdot,t_0)\|^2_{L^2_{per}}.
\end{eqnarray*} Moreover, \begin{eqnarray*}\displaystyle\lim_{s\rightarrow
t_0}\displaystyle\int_0^L\log(|u(x,s)|^p)|u(x,s)|^2\
dx=\int_0^L\log(|u(x,t_0)|^p)|u(x,t_0)|^2\ dx,
\end{eqnarray*} for all $t_0\in (0,T)$. Since $\mathcal{E}$ is a
conserved quantity, the arguments contained in Cazenave \cite[Chapter II,
Lemma 2.4.4]{cazenave} give us that
\begin{eqnarray*}u\in C^0([0,T];H^1_{per}([0,L]))\ \ \ \textrm{and}\ \ \ u_t\in C^0([0,T];L^2_{per}([0,L])).\end{eqnarray*}
}\end{remark}\

\textbf{Existence and Uniqueness in $H_{per,e}^1([0,L])$.}\\

Similar arguments can be used to show existence and uniqueness of
weak solutions $u$ related to the Cauchy problem $(\ref{gl.1})$ in a convenient subspace constituted by even periodic functions. Suppose
$p$ and $L$ as above. Consider $T>0$, $u_0\in H^1_{per,e}([0,L])$
and $u_1\in L^2_{per,e}([0,L])$. We have the following result:
\begin{theorem}\label{t.bc.gl.6} There exists an unique weak solution $u:\mathbb{R}\times [0,T]\rightarrow
\mathbb{C}$ for the Cauchy problem (\ref{gl.1}). In addition, solution $u$ must satisfy
$u\in L^{\infty}(0,T;H^1_{per,e}([0,L]))$, $u_t\in
L^{\infty}(0,T;L^2_{per,e}([0,L]))$ and $u_{tt}\in
L^{\infty}(0,T;(H^{1}_{per,e}([0,L]))').$
\end{theorem}
{\begin{flushright} $\square$ \end{flushright}}

\section{Orbital Stability of Standing Waves for the Logarithmic Klein-Gordon Equation}
 Now, we establish the  orbital stability of periodic waves related to the
 equation $(\ref{KG1})$. First, we
 present the existence of periodic standing waves $u(x,t)=e^{ict}\varphi(x)$, $t>0$, where $c\in\mathbb{R}$ is called frequency of the wave
 and $\varphi$ is a smooth periodic function. After that, we need some basic tools concerning the spectral analysis related
 to the linearized operators in $(\ref{lkg.13.14})$ and
 $(\ref{lkg.14.14})$. To do so, we recall the arguments in
 \cite{natali1} which give us the spectral information for the associated Hill
 operator. Finally, we present the orbital stability of periodic
 standing waves using the arguments in
 \cite{bona}, \cite{grillakis1} and \cite{weinstein1}.

\subsection{Existence of Periodic Solutions}\label{s.exist.1}

Let us consider $p\in\mathbb{Z}^{+}$ and $c\in\mathbb{R}$.
We seek for periodic waves of the form $u(x,t)=e^{ict}\varphi(x)$
where $\varphi$ is a smooth $L$-periodic function. By substituting
this kind of solution in (\ref{KG1}) with $\mu=1$, we obtain the following
Euler-Lagrange equation,
\begin{eqnarray}\label{lkg.15}-\varphi''+h(c,\varphi)=0,\end{eqnarray}
where $h$ is given by
$h(c,\varphi)=(1-c^2)\varphi-\log(|\varphi|^p)\varphi.$\\
 \indent We
suppose that the function $\varphi$ is strictly positive. So, $h$ is
smooth in the open subset
$\mathcal{O}\subset\mathcal{R}=\mathbb{R}\times (0,+\infty)$.
Equation (\ref{lkg.15}) is conservative and periodic solutions are
contained in the level curves of the energy
\begin{eqnarray*}\mathcal{H}(\varphi,\xi)=-\displaystyle\frac{\xi^2}{2}+\displaystyle\left(\displaystyle\frac{1-c^2}{2}
+\displaystyle\frac{p}{4}\right)\varphi^2-\displaystyle\frac{1}{2}\log(\varphi^p)\varphi^2,\end{eqnarray*}
where $\xi=\varphi'$. The function $H$,
\begin{eqnarray*}H(c,\varphi)=\displaystyle\left(\displaystyle\frac{1-c^2}{2}+\displaystyle\frac{p}{4}\right)\varphi^2-\displaystyle\frac{1}{2}\log(\varphi^p)\varphi^2,\end{eqnarray*}
satisfies $\displaystyle\frac{\partial H}{\partial\varphi}=h$ and
$\displaystyle\lim_{\varphi\rightarrow 0^{+}}H(c,\varphi)=0$.

For a fixed $c\in\mathbb{R}$, we see that function $h(c,\cdot)$ has only
three zeros, namely, $-e^{\frac{1-c^2}{p}}$, $0$ e
$e^{\frac{1-c^2}{p}}$, since $h(c,\varphi)\rightarrow 0$ if
$\varphi\rightarrow 0^{+}$. In our analysis, we consider the
consecutive roots $r_1(c)=0$ and $r_2(c)=e^{\frac{1-c^2}{p}}$. On
$\mathcal{O}$ the derivative of $h$ with respect to parameter
$\varphi$ is given by
\begin{eqnarray*}\displaystyle\frac{\partial h}{\partial \varphi}(c,\varphi)=1-c^2-p-\log(\varphi^p).\end{eqnarray*}
So, $\displaystyle\frac{\partial h}{\partial
\varphi}\left(c,e^{\frac{1-c^2}{p}}\right)=-p<0$ and
$\displaystyle\lim_{\varphi\rightarrow
0^{+}}\displaystyle\frac{\partial h}{\partial
\varphi}(c,\varphi)=\infty$. From standard ODE theory, the pair
$(\varphi,\varphi')=\displaystyle\left(e^{\frac{1-c^2}{p}},0\right)$
is a center point and $(\varphi,\varphi')=(0,0)$ is a saddle point.
Around the center point, we obtain periodic solutions for the
equation (\ref{lkg.15}). Furthermore, the function
\begin{eqnarray*}\varrho_c(x)=e^{\frac{1}{2}+\frac{1-c^2}{p}}e^{-\frac{px^2}{4}}\end{eqnarray*}
is a solitary wave for the equation (\ref{lkg.15}). The
pair $(\varrho_c,\varrho_c')$ determines a closed curve $C^{\infty}$
which contains $(r_2(c),0)$ in its interior. In addition,
$(\varrho_c,\varrho_c')$ also satisfies the identity
\begin{eqnarray*}\mathcal{H}(\varrho_c,\varrho_c')=-\displaystyle\frac{(\varrho_c')^2}{2}+\displaystyle\left(\displaystyle\frac{1-c^2}{2}+\displaystyle\frac{p}{4}\right)\varrho_c^2-\displaystyle\frac{1}{2}\log(\varrho_c^p)\varrho_c^2=0=\mathcal{H}(0,0).\end{eqnarray*}

In applications, the closed curve $C^{\infty}$ is formed by either
graphs of a homoclinic orbits or graphs of pairs heteroclinic orbits
of (\ref{lkg.15}). All the orbits (\ref{lkg.15}) that live inside
curve $C^{\infty}$ are positive periodic orbits that turn around of
$(r_2(c),0)$ and they are contained in the level curves
$\mathcal{H}(\varphi,\xi)=B$. Here, $B$ is a real constant
satisfying
\begin{eqnarray*}0=\mathcal{H}(0,0)<B<\mathcal{H}\left(e^{\frac{1-c^2}{p}},0\right)=\displaystyle\frac{p}{4}e^{\frac{2(1-c^2)}{p}}.\end{eqnarray*}

We have the following result:
\begin{proposition}\label{t.lkg.1} For all
$c\in\mathbb{R}$, the equation
\begin{eqnarray*}-\varphi''+(1-c^2)\varphi-\log(|\varphi|^p)\varphi=0\end{eqnarray*}
has a positive $L_c$-periodic solution, where
$L_c\in\left(\frac{2\pi}{\sqrt{p}},\infty\right)$. Moreover, the
solution $\varphi=\varphi_c$ and the period $L_c$ are continuously
differentiable with respect to the parameter $c$.
\end{proposition}
\begin{proof}
See \cite[Theorem 2.1]{natali1}.
\end{proof}

Standard ODE theory gives us that if the pair
$(\alpha_0,\alpha_1)\in\mathbb{R}^2$, $(\alpha_0,\alpha_1)\neq
(r_2(c),0)$ belongs to the interior of the closed curve
$C^{\infty}$, then there exists a unique periodic positive
solution $\varphi$ of (\ref{lkg.15}) with $\varphi(0)=\alpha_0$ and
$\varphi'(0)=\alpha_1$. Moreover, if
$$e^{\frac{1-c^2}{p}}=r_2(c)<\alpha_0<\varrho_c(0)=e^{\frac{1}{2}+\frac{1-c^2}{p}}\ \ \textrm{and} \ \
\alpha_1=0,$$ we can use the symmetry of the problem in order to see
that $\varphi$ is an even function. In particular,
$\displaystyle\max_{x\in\mathbb{R}}\varphi(x)=\varphi(0)$ and the period satisfies
$L_c>\displaystyle\frac{2\pi}{\sqrt{p}}$.


\subsection{Spectral Analysis}\label{s.spect.a}
Let $\varphi_c$ be the $L_c-$periodic solution obtained in
Proposition $\ref{t.lkg.1}$. Define the linearized operator
$\mathcal{L}_{\varphi_c}$ around
$(\varphi_c,ic\varphi_c)=(\varphi_c,c\varphi_c,0,0)$,
\begin{eqnarray*}\label{lkg.12}\mathcal{L}_{\varphi_c}=\displaystyle\left(
                  \begin{array}{cccc}
                    -\partial^2_{x}+1-\log(|\varphi_c|^p)-p & -c & 0 & 0 \\
                    -c & 1 & 0 & 0 \\
                    0 & 0 & -\partial^2_{x}+1-\log(|\varphi_c|^p) & c \\
                    0 & 0 & c & 1 \\
                  \end{array}
                \right),
\end{eqnarray*} defined in $[L^2_{per}([0,L_c])]^4$ with the domain $[H^2_{per}([0,L_c])\times L^2_{per}([0,L_c])]^2$. Our objective is
to study the behavior of the non-positive spectrum of
$\mathcal{L}_{\varphi_c}$. First of all, we need to present some
preliminary results which will be useful later. Consider the
auxiliary operators
\begin{eqnarray}\label{lkg.13}\mathcal{L}_{Re,\varphi_c}=\displaystyle\left(\begin{array}{cc}
                                                         -\partial^2_{x}+1-\log(|\varphi_c|^p)-p & -c \\
                                                         -c &
                                                         1
                                                       \end{array}\right)
\end{eqnarray} and \begin{eqnarray}\label{lkg.14}\mathcal{L}_{Im,\varphi_c}=\displaystyle\left(\begin{array}{cc}
                                                         -\partial^2_{x}+1-\log(|\varphi_c|^p) & c \\
                                                         c & 1
                                                       \end{array}\right).\end{eqnarray}

\begin{proposition}\label{p.lkg.1} Let us consider the self-adjoint operator
\begin{eqnarray}\label{opera12}\mathcal{L}_{1,\varphi_c}:=-\partial_x^2+(1-c^2)-p-\log(|\varphi_c|^p),\end{eqnarray}
defined in $L_{per}^2([0,L_c])$ with the domain $H_{per}^2([0,L_c])$. Let
$\mathcal{L}_{Re,\varphi_c}$ be the operator in (\ref{lkg.13}),
defined in $[L_{per}^2([0,L_c])]^2$ with the domain
$H_{per}^2([0,L_c])\times L_{per}^2([0,L_c])$. The real number
$\lambda\leq 0$ is a non-positive eigenvalue of
$\mathcal{L}_{Re,\varphi_c}$, if and only if,
$\gamma:=\lambda\left(1-\frac{c^2}{\lambda-1}\right)\leq 0$ is a
non-positive eigenvalue of $\mathcal{L}_{1,\varphi_c}$.
\end{proposition}
\begin{proof} Indeed, let $\lambda\leq 0$ be a non-positive
eigenvalue for the operator $\mathcal{L}_{Re,\varphi_c}$ whose
eigenfunction is $\displaystyle(\zeta_1,\zeta_2)\in
H^2_{per}([0,L_c])\times L^2_{per}([0,L_c])$. Thus,
\begin{eqnarray*}&&\mathcal{L}_{Re,\varphi_c}\displaystyle\left(\begin{array}{c}
                                      \zeta_1 \\
                                      \zeta_2
                                    \end{array}
\right)=\lambda\displaystyle\left(\begin{array}{c}
                                      \zeta_1 \\
                                      \zeta_2
                                    \end{array}
\right).\end{eqnarray*} So,
\begin{eqnarray}\label{kg2.11}-\zeta_1''+\zeta_1-\log(|\varphi_c|^p)\zeta_1-p\zeta_1-c\zeta_2=\lambda
\zeta_1\ \ \ \textrm{and} \ \ \
\zeta_2=\displaystyle-\frac{c}{\lambda-1}\zeta_1.\end{eqnarray} We
have that $\gamma\leq 0$ since $\lambda\leq 0$. Moreover, from
(\ref{kg2.11}), we have that
\begin{eqnarray*}\mathcal{L}_{1,\varphi_c}(\zeta_1)&=&-\zeta_1''+(1-c^2)\zeta_1-p\zeta_1-\log(|\varphi_c|^p)\zeta_1\nonumber\\
\nonumber\\
&=&\lambda
\zeta_1-c^2\zeta_1-\displaystyle\frac{c^2}{\lambda-1}\zeta_1=\lambda\displaystyle\left(1-\frac{c^2}{\lambda-1}
\right)\zeta_1=\gamma\zeta_1.
\end{eqnarray*}
The converse of the result follows from similar arguments.
\end{proof}

\begin{remark}\textit{\label{p.lkg.2} A similar result can be determined by comparing
operator
\begin{eqnarray}\label{operator123}
\mathcal{L}_{2,\varphi_c}:=-\partial_x^2+(1-c^2)-\log(|\varphi_c|^p),\end{eqnarray}
with the operator $\mathcal{L}_{Im,\varphi_c}$ given by
(\ref{lkg.14}).}
\end{remark}

Previous proposition helps us to determine the non-positive spectrum
of the linear operator $\mathcal{L}_{\varphi_c}$ by knowing the
behavior of the non-positive spectra for the operators
$\mathcal{L}_{1,\varphi_c}$ and $\mathcal{L}_{2,\varphi_c}$. In
fact, operator $\mathcal{L}_{\varphi_c}$ is diagonal and thus, it is
sufficient to analyze the non-positive spectra of the operators
$\mathcal{L}_{Re,\varphi_c}$ and $\mathcal{L}_{Im,\varphi_c}$. In
addition, operators $\mathcal{L}_{1,\varphi_c}$ and
$\mathcal{L}_{2,\varphi_c}$ are Hill operators since their
potentials are periodic. In \cite{natali1}, the authors have
determined a method to establish the position of the
zero eigenvalue related to the general Hill operator
\begin{eqnarray}\label{lkg.17.01}\mathcal{L}_{s}(y)=-y''+g(s,\varphi(x))y,\end{eqnarray} where $s$ belongs to a convenient open interval
contained in $\mathbb{R}$ and $g$ is a smooth function which depends
smoothly on the parameters $s$ and $\varphi$. We will describe the method in some few lines. Indeed, accordingly to \cite{magnus}, the
spectrum of $\mathcal{L}_{s}$ is formed by an unbounded sequence of
real numbers $$\gamma_0<\gamma_1\leq \gamma_2<\gamma_3\leq
\gamma_4<\ldots<\gamma_{2n-1}\leq\gamma_{2n}<\ldots$$ where equality
means that $\gamma_{2n-1}=\gamma_{2n}$ is a double eigenvalue. The
spectrum of $\mathcal{L}_s$ is characterized by the number of zeros
of the eigenfunctions if $\tilde{q}$ is an
eigenfunction associated to the eigenvalue $\gamma_{2n-1}$ or
$\gamma_{2n}$, then $\tilde{q}$ has exactly $2n$ zeros in the
half-open interval $[0,L_s)$. Moreover, since equation
\begin{eqnarray}\label{lkg.17.1}-y''+g(s,\varphi(x))=0\end{eqnarray} is a Hill type equation, we conclude from
classical Floquet theory in \cite{magnus} the existence of an
$2$-dimensional basis $\{\bar{y},q\}$ formed by smooth solutions of
the equation (\ref{lkg.17.1}), where $q$ is $L_s$-periodic. In
addition, there exists a constant $\theta\in\mathbb{R}$ such that
\begin{eqnarray}\label{lkg.17.2}\bar{y}(x+L_s)=\bar{y}(x)+\theta q(x)\ \textrm{for\ all}\ x\in\mathbb{R}.\end{eqnarray} Constant $\theta$ in (\ref{lkg.17.2}) measures how function $\bar{y}$ is
periodic. In fact, $\bar{y}$ is periodic if, and only if,
$\theta=0$.

\begin{definition}\label{d.lkg.1} \textit{The inertial index $\textrm{in}(\mathcal{L}_s)$ of $\mathcal{L}_s$ is a pair of
integers $(n,z)$, where $n$ is the dimension of the negative
subspace of $\mathcal{L}_s$ and $z$ is the dimension of the null
subspace of $\mathcal{L}_s$.}
\end{definition}

We also need the concept of isoinertial family of self-adjoint
operators.

\begin{definition}\label{d.lkg.2} \textit{A family of self-adjoint operators $\mathcal{L}_s$, which depends
on the real parameter $s$, is called isoinertial if
the inertial index $\textrm{in}(\mathcal{L}_s)$ of $\mathcal{L}_s$
does not depend on $s$.}
\end{definition}

Next result has been determined in \cite{natali1} and \cite{neves1}
and it determines the behavior of the non-positive spectrum of the
linear operator $\mathcal{L}_s$ in (\ref{lkg.17.01}) just by knowing
it for a fixed value $s_0$ in an open interval of $\mathbb{R}$.

\begin{proposition}\label{t.lkg.2} Let $\mathcal{L}_s$ be the Hill operator as in (\ref{lkg.17.01}) defined in $L^2_{per}([0,L_s])$ with the domain $D(\mathcal{L}_s)=H^2_{per}([0,L_s])$. If $\lambda=0$ is an eigenvalue of
$\mathcal{L}_s$ for every $s$ in an open interval of
$\mathbb{R}$ and the potential $g(s,\varphi(x))$ is
continuously differentiable in all variables, then the family of
operators $\mathcal{L}_s$ is isoinertial.
\end{proposition}

First, we shall calculate the inertial of
$\mathcal{L}_{1,\varphi_{c_0}}$ for a fixed value of the parameter
$s_0:=c_0\in I=\mathbb{R}$. Let $\varphi_{c_0}$ be a periodic function with period $L_{0}:=L_{c_0}$ for the equation $(\ref{travKG})$.
The mentioned solution satisfies
$(\varphi_{c_0}(0),\varphi_{c_0}'(0))=(\alpha_0,0)$, where
$\alpha_0\in\displaystyle\left(e^{\frac{1-c_0^2}{p}},e^{\frac{1}{2}+\frac{1-c_0^2}{p}}\right)$.
Hence, $\varphi_{c_0}$ is even, positive and $L_0$-periodic
function. Moreover, $\displaystyle\max_{x\in
[0,L_0]}\varphi_{c_0}(x)=\varphi_{c_0}(0)$ and
$L_0>\displaystyle\frac{2\pi}{\sqrt{p}}$.

Next, let us consider $q=\varphi_{c_0}'$ and $\bar{y}$ as the
unique solution of the Cauchy problem
\begin{eqnarray*}\displaystyle\left\{\begin{array}{l}
                                        -\bar{y}''+[(1-c_0^2)-p-\log(|\varphi_{c_0}|^p)]\bar{y}=0\\
                                        \bar{y}(0)=-\displaystyle\frac{1}{\varphi_{c_0}''(0)}\\
                                        \bar{y}'(0)=0
                                     \end{array}
\right.\end{eqnarray*} and also, from (\ref{lkg.17.2}), we see that
\begin{eqnarray}\label{lkg.17.3}\theta=\displaystyle\frac{\bar{y}(L_0)}{\varphi_{c_0}''(0)}.\end{eqnarray}
It is important to mention that if $q=\varphi_{c_0}'$,
the solution $\bar{y}$ is obtained by a simple application of Lemma
$2.1$ in \cite{neves}.

Next $\varphi'_{c_0}\in \ker(\mathcal{L}_{1,\varphi_{c_0}})$, that
is, this smooth function is an eigenfunction associated to the zero
eigenvalue. Moreover, $\varphi'_{c_0}$ has exactly two zeros in the
half-open interval $[0,L_0)$. So, from Floquet's Theory, we have
three possibilities:
\begin{itemize}
\item[(i)] $\gamma_1=\gamma_2=0\Rightarrow
\textrm{in}(\mathcal{L}_{1,\varphi_{c_0}})=(1,2)$.\
\item[(ii)] $\gamma_1=0<\gamma_2\Rightarrow
\textrm{in}(\mathcal{L}_{1,\varphi_{c_0}})=(1,1)$.\
\item[(iii)] $\gamma_1<\gamma_2=0\Rightarrow
\textrm{in}(\mathcal{L}_{1,\varphi_{c_0}})=(2,1)$.\
\end{itemize}
The method that we use to decide and calculate the inertial index is
based on Lemma 2.1, Theorem 2.2 and Theorem 3.1 of \cite{neves}.
This result can be stated as follows.

\begin{proposition}\label{t.lkg.3} Let $\theta$ be the constant given by
(\ref{lkg.17.3}). Then the eigenvalue $\lambda=0$ of
$\mathcal{L}_{1,\varphi_{c_0}}$ is simple, if and only if,
$\theta\neq 0$. Moreover, if $\theta\neq 0$ then $\gamma_1=0$ if
$\theta<0$ and $\gamma_2=0$ if $\theta>0$.
\end{proposition}

The next step is to count exactly the number of negative eigenvalues
for the operator $\mathcal{L}_{1,\varphi_{c_0}}$ and proving that zero is simple. For this purpose, we need to consider
$p$ and $c_0$ fixed. In order to illustrate the method, we shall
consider some specific values for $p$. In all cases, we
establish $c_0=0.5$.
\begin{center}
\begin{tabular}{|c|c|c|c|c|c|c|c|c|}
\hline
\multicolumn{9}{|c|}{\textbf{Some values of $\theta$ related to $p\in \mathbb{Z}^{+}$}}\\
\hline $p$ & $\varphi_{c_0}(0)$ & $\varphi_{c_0}'(0)$ & $\varphi_{c_0}''(0)$ & $\bar{y}(0)$ & $L_0$ & $\bar{y}(L_0)$ & $\bar{y}'(L_0)$ & $\theta$ \\
\hline\hline
$1$ & 2.5 & 0 & -0.4157 & 2.4054 & 6.3129 & 2.4054 & 0.2316 & -0.5571  \\
\hline
$2$ & 1.5 & 0 & -0.0914 & 10.941 & 4.4425 & 10.941 & 0.0563 & -0.6158  \\
\hline
$3$ & 1.5 & 0 & -0.6996 & 1.4294 & 3.6462 & 1.4294 & 0.1823 & -0.2606  \\
\hline
$4$ & 1.5 & 0 & -1.3078 & 0.7646 & 3.1775 & 0.7646 & 0.2710 & -0.2073  \\
\hline
$5$ & 1.5 & 0 & -1.9160 & 0.5219 & 2.8580 & 0.5219 & 0.3287 & -0.1716  \\
\hline
$6$ & 1.5 & 0 & -2.5242 & 0.3962 & 2.6214 & 0.3962 & 0.3677 & -0.1457  \\
\hline
$8$ & 1.5 & 0 & -3.7406 & 0.2673 & 2.2873 & 0.2673 & 0.4138 & -0.1106  \\
\hline
$10$ & 1.5 & 0 & -4.9570 & 0.2017 & 2.0570 & 0.2017 & 0.4366 & -0.0881  \\
\hline
$20$ & 1.5 & 0 & -11.034 & 0.0906 & 1.4754 & 0.0906 & 0.4433 & -0.0402  \\
\hline
\end{tabular}
\end{center}\

From previous table it is possible to see from Proposition
\ref{t.lkg.3} that
$\textrm{in}(\mathcal{L}_{1,\varphi_{c_0}})=(1,1).$ So, zero is a
simple eigenvalue and $\mathcal{L}_{1,\varphi_{c_0}}$ has only one
negative eigenvalue which is simple (the first eigenvalue is always
simple using Floquet's Theorem). Since
$\{\mathcal{L}_{1,\varphi_{c}}\}_{c\in\mathbb{R}}$ is isoinertial,
we deduce from Proposition \ref{t.lkg.2} that
$\textrm{in}(\mathcal{L}_{1,\varphi_{c}})=(1,1)$ for all
$c\in\mathbb{R}$.

From Proposition \ref{p.lkg.1}, we have that the operator
$\mathcal{L}_{Re,\varphi_c}$ has only one negative eigenvalue which
is simple and zero is a simple eigenvalue whose eigenfunction is
$(\varphi_c',c\varphi_c')$. Moreover, the remainder of the spectrum
is constituted by a discrete set of eigenvalues (so, it is bounded
away from zero).

Concerning the operator
$\mathcal{L}_{2,\varphi_c}$ as defined in $(\ref{operator123})$,
the procedure is quite similar. However, since $\varphi_{c_0}(x)>0$ we deduce directly from Floquet's Theory that zero is the first eigenvalue of $\mathcal{L}_{2,\varphi_{c_0}}$ which is simple. The family of operators
$\{\mathcal{L}_{2,\varphi_c}\}_{c\in\mathbb{R}}$ is isoinertial and,
therefore, $\textrm{in}(\mathcal{L}_{2,\varphi_{c}})=(0,1)$ for all
$c\in\mathbb{R}$. Hence, from Remark \ref{p.lkg.2}, operator
$\mathcal{L}_{Im,\varphi_c}$ has no negative eigenvalues and zero is
a simple eigenvalue whose eigenfunction is
$(\varphi_c,-c\varphi_c)$. Furthermore, the remainder of the
spectrum is discrete and bounded away from zero.

Arguments above allow us to conclude that diagonal operator
$\mathcal{L}_{\varphi_c}$ has only one negative eigenvalue which is
simple and zero is double eigenvalue with
$$\ker(\mathcal{L}_{\varphi_c})=\textrm{span}\{(\varphi_c',c\varphi_c',0,0),(0,0,\varphi_c,-c\varphi_c)\}.$$

Next, we analyze the non-positive spectrum of the linearized
operator $\mathcal{L}_{\varphi_c}$ defined in
$X_e:=[L^2_{per,e}([0,L_c])]^4$ with the domain
$Z_e:=[H^2_{per,e}([0,L_c])\times L^2_{per,e}([0,L_c])]^2.$ Indeed,
first we see that eigenfunction $(\varphi_c',c\varphi_c',0,0)$ does
not belong to the kernel of
$\displaystyle\left.\mathcal{L}_{\varphi_c}\right|_{X_e}$ because
$\varphi_c'$ is odd. Therefore, we deduce that
$\ker(\displaystyle\left.\mathcal{L}_{\varphi_c}\right|_{X_e})=\textrm{span}\{(0,0,\varphi_c,-c\varphi_c)\}.$
Moreover, since the eigenfunction for the first
eigenvalue of $\mathcal{L}_{1,\varphi_c}$ is even (see Theorem 1.1
in \cite{magnus}), we obtain that the number of negative eigenvalues
of the linearized operator $\mathcal{L}_{\varphi_c}$, defined in
$X_e$ with the domain $Z_e$, remains equal to one.

\subsection{The convexity of the function d}
Let $\varphi_{c_0}$ be a smooth, even, positive and periodic
solution with period $L_0>\frac{2\pi}{\sqrt{p}}$ for the equation
$(\ref{travKG})$. Operator $\mathcal{L}_{1,\varphi_{c_0}}$ in
$(\ref{opera12})$ has zero as a simple eigenvalue. So, from
Proposition \ref{t.lkg.3}, $\theta\neq 0$ where $\theta$ is given in
(\ref{lkg.17.3}). From Theorem 3.3 in \cite{natali1}, we conclude
the existence of a neighborhood $I$ of $c_0$ and a family of
functions $\{\varphi_c\}_{c\in I}$ such that $\varphi_c$ is a
solution to the equation (\ref{travKG}) for all $c\in I$. Moreover,
$\varphi_c$ is a smooth, even, positive and $L_c$-periodic function
and the map $c\in I\mapsto\varphi_c\in H^2_{per,e}([0,L_0])$ is
smooth. Furthermore, $\textrm{in}(\mathcal{L}_{Re,\varphi_c})=(1,1)$
and $\textrm{in}(\mathcal{L}_{Im,\varphi_c})=(0,1)$ for all $c\in
I$.

Now, in order to simplify the notation, we denote $L=L_0$. Consider $\mathcal{E}$ and
$\mathcal{F}$ the two conserved quantities in (\ref{gl.108.1}) and (\ref{gl.109.1}). Define
the function
\begin{eqnarray}\label{lkg.14.1}\begin{array}{rccl}
                        d: & \mathbb{R} & \rightarrow & \mathbb{R} \\
                           & c & \mapsto & \mathcal{E}(\varphi_c,c\varphi_c,0,0)-c\mathcal{F}(\varphi_c,c\varphi_c,0,0).
                       \end{array}\end{eqnarray}
Since $(\varphi_c,c\varphi_c,0,0)$ is a critical point of the
functional $\mathcal{G}=\mathcal{G}_c=\mathcal{E}-c\mathcal{F}$, we
deduce from (\ref{lkg.14.1})
\begin{eqnarray*}d'(c)=-\mathcal{F}(\varphi_c,c\varphi_c,0,0)=-\displaystyle\int_0^L c\varphi_c^2(x)\
dx.\end{eqnarray*} So for all $c\in I$,
\begin{eqnarray}\label{lkg.14.2}d''(c)=-\displaystyle\int_0^L\varphi_c^2\
dx-c\displaystyle\frac{d}{dc}\displaystyle\left(\displaystyle\int_0^L\varphi_c^2\
dx\right).
\end{eqnarray}

Our intention is to give a convenient expression for
(\ref{lkg.14.2}). In fact, recall that $\varphi_c>0$ and
\begin{eqnarray}\label{lkg.30}\varphi_c''-(1-c^2)\varphi_c+\log(\varphi_c^p)\varphi_c=0\ \ \textrm{for\ all}\
c\in I.\end{eqnarray} Since $c\in I\mapsto \varphi_c\in
H^2_{per,e}([0,L])$ is a smooth function, we can define
$\eta_c:=\displaystyle\frac{d}{dc}\displaystyle\left(\varphi_c\right)\in
H^2_{per}([0,L])$. Deriving equation (\ref{lkg.30}) with respect to
the parameter $c$ to obtain
\begin{eqnarray}\label{lkg.31}\eta_c''+2c\varphi_c-(1-c^2)\eta_c+\log(\varphi_c^p)\eta_c+p\eta_c=0.\end{eqnarray}

Multiplying equation (\ref{lkg.31}) by $\varphi_c$ and integrating
the result over $[0,L]$, we have
\begin{eqnarray*}\displaystyle\int_0^L \displaystyle\left[\eta_c''\varphi_c+2c\varphi_c^2-(1-c^2)\eta_c\varphi_c+\log(\varphi_c^p)\eta_c\varphi_c+p\varphi_c\eta_c\right]  \
dx=0.\end{eqnarray*}

Equation above is equivalent to
\begin{eqnarray}\label{lkg.32}\ \ \ \ \ \displaystyle\int_0^L
\displaystyle\left[\varphi_c''\eta_c+2c\varphi_c^2-(1-c^2)\eta_c\varphi_c+\log(\varphi_c^p)\eta_c\varphi_c+p\varphi_c\eta_c\right]
dx=0.\end{eqnarray} By combining expressions (\ref{lkg.30}) and
(\ref{lkg.32}), we see that
\begin{eqnarray*}\displaystyle\int_0^L\displaystyle\left[2c\varphi_c^2+p\varphi_c\eta_c\right]\
dx=0.\end{eqnarray*} Hence,
\begin{eqnarray}\label{lkg.33}2c\displaystyle\int_0^L\varphi_c^2\ dx+p\displaystyle\frac{d}{dc}\displaystyle\left(\displaystyle\int_0^L\displaystyle\frac{\varphi_c^2}{2}\ dx\right)=0.\end{eqnarray}

From identity (\ref{lkg.33}),
\begin{eqnarray*}\displaystyle\frac{d}{dc}\displaystyle\left(\displaystyle\int_0^L\varphi_c^2\ dx\right)=-\displaystyle\frac{4c}{p}\displaystyle\int_0^L\varphi_c^2\
dx.\end{eqnarray*} Therefore for all $c\in I$,
\begin{eqnarray*}d''(c)=-\displaystyle\int_0^L\varphi_c^2\ dx+\displaystyle\frac{4c^2}{p}\displaystyle\int_0^L\varphi_c^2\ dx
=\displaystyle\left(\displaystyle\frac{4c^2}{p}-1\right)\|\varphi_c\|^2_{L^2_{per}([0,L])}.\end{eqnarray*}

The sign of $d''(c)$ depends on the sign of the quantity $\displaystyle\frac{4c^2}{p}-1$, where $c\in I$.
Finally, we have that
\begin{eqnarray}\label{lkg.33.1}d''(c)>0\Leftrightarrow
|c|>\displaystyle\frac{\displaystyle\sqrt{p}}{2}\ \ \ \textrm{and}\
\ \ d''(c)<0\Leftrightarrow
|c|<\displaystyle\frac{\displaystyle\sqrt{p}}{2}.\end{eqnarray}

\subsection{Orbital Stability of Standing Waves}

 In this subsection, we prove results of orbital
stability of standing waves related to the Logarithmic Klein-Gordon equation. To do so, we use classical methods based on ideas established in
\cite{bona}, \cite{grillakis1} and \cite{weinstein1} to get the stability over the complex space
$X:=H^1_{per}([0,L])\times L^2_{per}([0,L])$.


In what follows, let us consider $p=1,2,3$ (the reason to
consider these values of $p$ will be explained later).
\begin{definition}\label{d.lkg.2.2}\textit{ We say that $\varphi_c$
is orbitally stable by the periodic flow of the equation
(\ref{KG1}), where $\varphi_c$ satisfies (\ref{travKG}) if for all $\epsilon>0$ there exists
$\delta>0$ such that if
\begin{eqnarray*}(u_0,u_1)\in X=H^1_{per}([0,L])\times
L^2_{per}([0,L])\ \textrm{satisfies}\
\|(u_0,u_1)-(\varphi,ic\varphi)\|_{X}<\delta\end{eqnarray*} then
$\vec{u}=(u,u_t)$ is a weak solution to equation (\ref{KG1}) with
$\vec{u}(\cdot,0)=(u_0,u_1)$ and
\begin{eqnarray*}
\displaystyle\sup_{t\geq
0}\displaystyle\inf_{\theta\in\mathbb{R},y\in\mathbb{R}}\|\vec{u}(\cdot,t)-e^{i\theta}(\varphi(\cdot+y),ic\varphi(\cdot+y))\|_{X}<\epsilon.\end{eqnarray*}
Otherwise, we say that $\varphi_c$ is orbitally unstable.}
\end{definition}


Firstly, we will need some additional information about spectral
proprieties of the operators $\mathcal{L}_{Re,\varphi_c}$ and
$\mathcal{L}_{Im,\varphi_c}$, in (\ref{lkg.13}) and (\ref{lkg.14})
respectively. Indeed, as we have already
determined in the last subsection one sees that
$\textrm{in}(\displaystyle\left.\mathcal{L}_{Re,\varphi_c}\right)=(1,1)$
and
$\textrm{in}(\displaystyle\left.\mathcal{L}_{Im,\varphi_c}\right)=(1,0)
$ for all $c\in I$. Since
$\textrm{in}(\displaystyle\left.\mathcal{L}_{Im,\varphi_c}\right)=(1,0)$,
we have that
\begin{eqnarray}\label{lkg.33.1.1}\ \ \ \ \displaystyle\left\langle \mathcal{L}_{Im,\varphi_c}\displaystyle\left(\begin{array}{c}
                                                                                                       \psi_1 \\
                                                                                                       \psi_2
                                                                                                     \end{array}\right),\displaystyle\left(\begin{array}{c}
                                                                                                       \psi_1 \\
                                                                                                       \psi_2
                                                                                                     \end{array}
\right)\right\rangle_{2,2}\geq 0\ \textrm{for\ all\ }
\displaystyle\left(\begin{array}{c}
                                                                                                       \psi_1 \\
                                                                                                       \psi_2
                                                                                                     \end{array}
\right)\in X,\end{eqnarray} where $\langle\cdot,\cdot\rangle_{2,2}$
denotes the inner product in $L^2_{per}([0,L])\times
L^2_{per}([0,L])$ and $\|\cdot\|_{2,2}$ corresponds to the respective
induced norm. We have the following result:

\begin{proposition}\label{p.lkg.4} If \begin{eqnarray*}\beta&:=&\inf\displaystyle\left\{\displaystyle\left\langle \mathcal{L}_{Im,\varphi_c}\displaystyle\left(\begin{array}{c}
                                                                                                       \psi_1 \\
                                                                                                       \psi_2
                                                                                                     \end{array}
\right),\displaystyle\left(\begin{array}{c}
                                                                                                       \psi_1 \\
                                                                                                       \psi_2
                                                                                                     \end{array}
\right)\right\rangle_{2,2};\ \displaystyle\left(\begin{array}{c}
                                                                                                       \psi_1 \\
                                                                                                       \psi_2
                                                                                                     \end{array}
\right)\in X,\right.\nonumber\\
\\
&&\ \ \ \ \
\displaystyle\left.\displaystyle\left\|\displaystyle\left(\begin{array}{c}
                                                                                                       \psi_1 \\
                                                                                                       \psi_2
                                                                                                     \end{array}
\right)\right\|_{2,2}=1,\ \displaystyle\left\langle
\displaystyle\left(\begin{array}{c}
                                                                                                       \psi_1 \\
                                                                                                       \psi_2
                                                                                                     \end{array}
\right),\displaystyle\left(\begin{array}{c}
                                                                                                       \varphi_c\log(\varphi_c^p) \\
                                                                                                       -c\varphi_c
                                                                                                     \end{array}
\right)\right\rangle_{2,2}=0\right\}
\end{eqnarray*} then $\beta>0$.
\end{proposition}
\begin{proof} First, from (\ref{lkg.33.1.1}) we obtain $\beta\geq 0$. Let us suppose that $\beta=0$. There exists a sequence
\begin{eqnarray}\label{lkg.33.2}\displaystyle\left\{\displaystyle\left(\begin{array}{c}
                                           \psi_{1,j} \\
                                           \psi_{2,j}
                                         \end{array}
\right)\right\}_{j\in\mathbb{N}}\subset X\end{eqnarray} such that
\begin{eqnarray}\label{lkg.33.3}\displaystyle\left\langle \mathcal{L}_{Im,\varphi_c}\displaystyle\left(\begin{array}{c}
                                                                                                       \psi_{1,j} \\
                                                                                                       \psi_{2,j}
                                                                                                     \end{array}
\right),\displaystyle\left(\begin{array}{c}
                                                                                                       \psi_{1,j} \\
                                                                                                       \psi_{2,j}
                                                                                                     \end{array}
\right)\right\rangle_{2,2}\stackrel{j\rightarrow\infty}{\longrightarrow}0\end{eqnarray}
and for all $j\in\mathbb{N}$, one has
\begin{equation}\label{cond12345}\displaystyle\left\| \displaystyle\left(\begin{array}{c}
                                           \psi_{1,j} \\
                                           \psi_{2,j}
                                         \end{array}
\right) \right\|_{2,2}=1\ \textrm{and}\ \displaystyle\left\langle
\displaystyle\left(\begin{array}{c}
                                                                                                       \psi_{1,j} \\
                                                                                                       \psi_{2,j}
                                                                                                     \end{array}
\right),\displaystyle\left(\begin{array}{c}
                                                                                                       \varphi_c\log(\varphi_c^p) \\
                                                                                                       -c\varphi_c
                                                                                                     \end{array}
\right)\right\rangle_{2,2}=0.
\end{equation}

\indent Now, using the convergence in $(\ref{lkg.33.3})$, we see that $\displaystyle\left\{\displaystyle\left(\begin{array}{c}
                                                                                                       \psi_{1,j} \\
                                                                                                       \psi_{2,j}
                                                                                                     \end{array}
\right)\right\}_{j\in\mathbb{N}}$ is uniformly bounded in $X$. There exists a subsequence of
(\ref{lkg.33.2}), still denoted by (\ref{lkg.33.2}) and
$\displaystyle\left(\begin{array}{c}
                                         \psi_1^{*} \\
                                         \psi_2^{*}
                                       \end{array}
\right)\in X$ such that
\begin{equation}\label{wconverg}\displaystyle\left(\begin{array}{c}
                                           \psi_{1,j} \\
                                           \psi_{2,j}
                                         \end{array}
\right)\rightharpoonup\displaystyle\left(\begin{array}{c}
                                         \psi_1^{*} \\
                                         \psi_2^{*}
                                       \end{array}
\right)\ \ \textrm{weakly\ in}\ \  X.\end{equation}

\indent On the other hand, since the embedding $H^1_{per}([0,L])\hookrightarrow L^2_{per}([0,L])$ is compact and $\{\psi_{1,j}\}$ is uniformly bounded in $H_{per}^1([0,L])$,
we deduce, up to a subsequence that
\begin{equation}\label{conver1}
\psi_{1,j}\rightarrow \psi_1^{*}\ \ \mbox{in}\ L_{per}^2([0,L]).
\end{equation}
The weak convergence in $(\ref{wconverg})$ and the strong convergence in $(\ref{conver1})$ give us that second condition in $(\ref{cond12345})$ is satisfied for
$\left(\begin{array}{c}
\psi_1^{*} \\
\psi_2^{*}
\end{array}
\right)$.

From $(\ref{lkg.33.3})$, we see that for all $\varepsilon>0$, there exists $j_0\in\mathbb{N}$ such that if $j\geq j_0$, we have
\begin{equation}\label{conver2}
\left|\int_{0}^L(\psi_{1,j}')^2+\psi_{1,j}^2-\log(|\varphi_c|^p)\psi_{1,j}^2+2c\psi_{1,j}\psi_{2,j}+\psi_{2,j}^2dx\right|<\varepsilon.
\end{equation}
that is, if $j\geq j_0$, we obtain from the fact $\displaystyle\left\| \displaystyle\left(\begin{array}{c}
                                           \psi_{1,j} \\
                                           \psi_{2,j}
                                         \end{array}
\right) \right\|_{2,2}=1$ that

\begin{equation}\label{conver3}
1\leq \int_{0}^L(\psi_{1,j}')^2+\psi_{1,j}^2+\psi_{2,j}^2dx<\int_{0}^L\log(|\varphi_c|^p)\psi_{1,j}^2-2c\psi_{1,j}\psi_{2,j}dx+\varepsilon.
\end{equation}
Since in particular $1<\displaystyle\int_{0}^L\log(|\varphi_c|^p)\psi_{1,j}^2-2c\psi_{1,j}\psi_{2,j}dx+\varepsilon$, we obtain from the fact $\psi_{2,j}\rightharpoonup \psi_2^{*}$ weakly in $L_{per}^2([0,L])$ and $(\ref{conver1})$ that

\begin{equation}\label{conver4}
1\leq\int_{0}^L\log(|\varphi_c|^p){\psi_{1}^{*}}^2-2c\psi_{1}^{*}\psi_{2}^{*}dx+\varepsilon.
\end{equation}
From $(\ref{conver4})$, we conclude that $\displaystyle\left(\begin{array}{c}
                                         \psi_1^{*} \\
                                         \psi_2^{*}
                                       \end{array}
\right)\neq\vec{0}$.
Consider $\displaystyle\left(\begin{array}{c}
                                           \psi_{1} \\
                                           \psi_{2}
                                         \end{array}
\right) =\frac{1}{\left\|\displaystyle\left(\begin{array}{c}
                                           \psi_{1}^{*} \\
                                           \psi_{2}^{*}
                                         \end{array}
\right)\right\|_{2,2} }\displaystyle\left(\begin{array}{c}
                                           \psi_{1}^{*}\\
                                           \psi_{2}^{*}
                                         \end{array}
\right) $. We obtain
\begin{eqnarray}\label{lkg.50.5}\displaystyle\left\| \displaystyle\left(\begin{array}{c}
                                           \psi_{1} \\
                                           \psi_{2}
                                         \end{array}
\right) \right\|_{2,2}=1\ \textrm{and}\ \displaystyle\left\langle
\displaystyle\left(\begin{array}{c}
                                                                                                       \psi_{1} \\
                                                                                                       \psi_{2}
                                                                                                     \end{array}
\right),\displaystyle\left(\begin{array}{c}
                                                                                                       \varphi_c\log(\varphi_c^p) \\
                                                                                                       -c\varphi_c
                                                                                                     \end{array}
\right)\right\rangle_{2,2}=0.\end{eqnarray}

In addition, Fatou's Lemma gives us
\begin{eqnarray*}0\leq\displaystyle\left\langle
\mathcal{L}_{Im,\varphi_c}\displaystyle\left(\begin{array}{c}
                                                                                                       \psi_{1} \\
                                                                                                       \psi_{2}
                                                                                                     \end{array}
\right),\displaystyle\left(\begin{array}{c}
                                                                                                       \psi_{1} \\
                                                                                                       \psi_{2}
                                                                                                     \end{array}
\right)\right\rangle_{2,2}\leq\frac{1}{\left\|\displaystyle\left(\begin{array}{c}
			\psi_{1}^{*} \\
			\psi_{2}^{*}
		\end{array}
		\right)\right\|_{2,2}^2}\displaystyle\liminf_{j\rightarrow
\infty}\displaystyle\left\langle
\mathcal{L}_{Im,\varphi_c}\displaystyle\left(\begin{array}{c}
                                                                                                       \psi_{1,j} \\
                                                                                                       \psi_{2,j}
                                                                                                     \end{array}
\right),\displaystyle\left(\begin{array}{c}
                                                                                                       \psi_{1,j} \\
                                                                                                       \psi_{2,j}
                                                                                                     \end{array}
\right)\right\rangle_{2,2}=0.
\end{eqnarray*} \indent Therefore, we obtain $\displaystyle\left\langle
\mathcal{L}_{Im,\varphi_c}\displaystyle\left(\begin{array}{c}
                                                                                                       \psi_{1} \\
                                                                                                       \psi_{2}
                                                                                                     \end{array}
\right),\displaystyle\left(\begin{array}{c}
                                                                                                       \psi_{1} \\
                                                                                                       \psi_{2}
                                                                                                     \end{array}
\right)\right\rangle_{2,2}=0$ and as consequence, the minimum $\beta$ is attained in
$\displaystyle\left(\begin{array}{c}
                                         \psi_1 \\
                                         \psi_2
                                       \end{array}
\right)\neq \vec{0}$.\\ 
\indent We are in position to use Lagrange's Theorem to guarantee the existence of $(a,b)\in\mathbb{R}^2$ such
that
\begin{eqnarray}\label{lkg.51}\mathcal{L}_{Im,\varphi_c}\displaystyle\left(\begin{array}{c}
                                                                                                       \psi_{1} \\
                                                                                                       \psi_{2}
                                                                                                     \end{array}
\right)=a\displaystyle\left(\begin{array}{c}
                                                                                                       \psi_{1} \\
                                                                                                       \psi_{2}
                                                                                                     \end{array}
\right)+b\displaystyle\left(\begin{array}{c}
                                                                                                       \varphi_c\log(\varphi_c^p) \\
                                                                                                       -c\varphi_c
                                                                                                     \end{array}
\right).\end{eqnarray} \indent Taking the inner product of (\ref{lkg.51})
with the function $\left(
                 \begin{array}{c}
                   \psi_1 \\
                   \psi_2 \\
                 \end{array}
               \right),
$ we get from (\ref{lkg.50.5}) that $a=0$. \\
\indent On the other hand, since
$\ker(\mathcal{L}_{Im,\varphi_c})=\textrm{span}\displaystyle\left\{\displaystyle\left(\begin{array}{c}
                                                                                                \varphi_c \\
                                                                                                -c\varphi_c
                                                                                              \end{array}
\right)\right\}$ and $\mathcal{L}_{Im,\varphi_c}$ is a self-adjoint
operator, we have the identity
\begin{eqnarray*}0=\displaystyle\left\langle\displaystyle\left(\begin{array}{c}
                                                                 \psi_1 \\
                                                                 \psi_2
                                                               \end{array}
\right),\mathcal{L}_{Im,\varphi_c}\displaystyle\left(\begin{array}{c}
                                                                                                \varphi_c \\
                                                                                                -c\varphi_c
                                                                                              \end{array}
\right)\right\rangle_{2,2}=b\displaystyle\left\langle\displaystyle\left(\begin{array}{c}
                                                                 \varphi_c\log(\varphi_c^p) \\
                                                                 -c\varphi_c
                                                               \end{array}
\right),\displaystyle\left(\begin{array}{c}
                                                                                                \varphi_c \\
                                                                                                -c\varphi_c
                                                                                              \end{array}
\right)\right\rangle_{2,2}.
\end{eqnarray*}

Identity (\ref{travKG}) allows us to deduce
\begin{eqnarray*}\displaystyle\left\langle\displaystyle\left(\begin{array}{c}
                                                                 \varphi_c\log(\varphi_c^p) \\
                                                                 -c\varphi_c
                                                               \end{array}
\right),\displaystyle\left(\begin{array}{c}
                                                                                                \varphi_c \\
                                                                                                -c\varphi_c
                                                                                              \end{array}
\right)\right\rangle_{2,2}=\displaystyle\int_0^L\varphi_c\displaystyle\left(\varphi_c\log(\varphi_c^p)+c^2\varphi_c\right)\
dx>0
\end{eqnarray*} and so, we conclude that $a=b=0$. There exists a
constant $c_1\neq 0$ such that
\begin{eqnarray*}\displaystyle\left(\begin{array}{c}
                                                                 \psi_1 \\
                                                                 \psi_2
                                                               \end{array}
\right)=c_1\displaystyle\left(\begin{array}{c}
                                                                 \varphi_c \\
                                                                 -c\varphi_c
                                                               \end{array}
\right).\end{eqnarray*} Hence,
\begin{eqnarray*}\displaystyle\left\langle\displaystyle\left(\begin{array}{c}
                                                                                                \psi_1 \\
                                                                                                \psi_2
                                                                                              \end{array}
\right),\displaystyle\left(\begin{array}{c}
                                                                 \varphi_c\log(\varphi_c^p) \\
                                                                 -c\varphi_c
                                                               \end{array}
\right)\right\rangle_{2,2}\neq 0,
\end{eqnarray*} which is a contradiction with (\ref{lkg.50.5}).
Therefore, $\beta>0$.
\end{proof}

In the next proposition we need to use Lemma 3.1 in
\cite{weinstein1}. In whole this subsection, we shall consider
$|c|>\displaystyle\frac{\sqrt{p}}{2}$.

\begin{proposition}\label{p.lkg.5}
\begin{itemize}
\item[(i)] If
\begin{eqnarray*}\gamma&:=&
\inf\displaystyle\left\{\displaystyle\left\langle
\mathcal{L}_{Re,\varphi_c}\displaystyle\left(\begin{array}{c}
                                                                                                       \psi_1 \\
                                                                                                       \psi_2
                                                                                                     \end{array}
\right),\displaystyle\left(\begin{array}{c}
                                                                                                       \psi_1 \\
                                                                                                       \psi_2
                                                                                                     \end{array}
\right)\right\rangle_{2,2};\ \displaystyle\left(\begin{array}{c}
                                                                                                       \psi_1 \\
                                                                                                       \psi_2
                                                                                                     \end{array}
\right)\in X,\right.\nonumber\\
&&\ \ \ \ \
\displaystyle\left.\displaystyle\left\|\displaystyle\left(\begin{array}{c}
                                                                                                       \psi_1 \\
                                                                                                       \psi_2
                                                                                                     \end{array}
\right)\right\|_{2,2}=1,\ \displaystyle\left\langle
\displaystyle\left(\begin{array}{c}
                                                                                                       \psi_1 \\
                                                                                                       \psi_2
                                                                                                     \end{array}
\right),\displaystyle\left(\begin{array}{c}
                                                                                                       c\varphi_c \\
                                                                                                       \varphi_c
                                                                                                     \end{array}
\right)\right\rangle_{2,2}=0\right\},
\end{eqnarray*} then $\gamma=0$.
\item[(ii)] If
\begin{eqnarray*}\kappa&:=&\inf\displaystyle\left\{\displaystyle\left\langle \mathcal{L}_{Re,\varphi_c}\displaystyle\left(\begin{array}{c}
                                                                                                       \psi_1 \\
                                                                                                       \psi_2
                                                                                                     \end{array}
\right),\displaystyle\left(\begin{array}{c}
                                                                                                       \psi_1 \\
                                                                                                       \psi_2
                                                                                                     \end{array}
\right)\right\rangle_{2,2};\ \displaystyle\left(\begin{array}{c}
                                                                                                       \psi_1 \\
                                                                                                       \psi_2
                                                                                                     \end{array}
\right)\in X,\
\displaystyle\left\|\displaystyle\left(\begin{array}{c}
                                                                                                       \psi_1 \\
                                                                                                       \psi_2
                                                                                                     \end{array}
\right)\right\|_{2,2}=1,\right.\nonumber\\
\\
&&\ \ \ \ \ \displaystyle\left. \displaystyle\left\langle
\displaystyle\left(\begin{array}{c}
                                                                                                       \psi_1 \\
                                                                                                       \psi_2
                                                                                                     \end{array}
\right),\displaystyle\left(\begin{array}{c}
                                                                                                       c\varphi_c \\
                                                                                                       \varphi_c
                                                                                                     \end{array}
\right)\right\rangle_{2,2}=0,\ \displaystyle\left\langle
\displaystyle\left(\begin{array}{c}
                                                                                                       \psi_1 \\
                                                                                                       \psi_2
                                                                                                     \end{array}
\right),\displaystyle\left(\begin{array}{c}
                                                                                                       \log(\varphi_c^p)\varphi_c'+p\varphi_c' \\
                                                                                                       c\varphi_c'
                                                                                                     \end{array}
\right)\right\rangle_{2,2}=0\right\},\\
\end{eqnarray*} then $\kappa>0$.
\end{itemize}
\end{proposition}
\begin{proof}
($i$). The fact that function $\varphi_c$ is bounded gives us that
$\gamma$ is finite. Now, since
\begin{eqnarray*}\displaystyle\left\langle\displaystyle\left(\begin{array}{c}
                                                               \varphi_c' \\
                                                               c\varphi_c'
                                                             \end{array}
\right),\left(\begin{array}{c}
                                                               c\varphi_c \\
                                                               \varphi_c
                                                             \end{array}
\right)\right\rangle_{2,2}=0\ \ \ \textrm{and}\ \ \
\mathcal{L}_{Re,\varphi_c}\displaystyle\left(\begin{array}{c}
                                              \varphi_c' \\
                                              c\varphi_c'
                                            \end{array}
\right)=\vec{0},\end{eqnarray*} it follows that $\gamma\leq 0$.

There exists a sequence $\displaystyle\left\{\displaystyle\left(\begin{array}{c}
                                           \psi_{1,j} \\
                                           \psi_{2,j}
                                         \end{array}
\right)\right\}_{j\in\mathbb{N}}\subset X$ such that
\begin{eqnarray*}\displaystyle\left\|\displaystyle\left(\begin{array}{c}
                                                                                                       \psi_{1,j} \\
                                                                                                       \psi_{2,j}
                                                                                                     \end{array}
\right)\right\|_{2,2}=1\ \ \textrm{and}\ \ \  \displaystyle\left\langle
\displaystyle\left(\begin{array}{c}
                                                                                                       \psi_{1,j} \\
                                                                                                       \psi_{2,j}
                                                                                                     \end{array}
\right),\displaystyle\left(\begin{array}{c}
                                                                                                       c\varphi_c \\
                                                                                                       \varphi_c
                                                                                                     \end{array}
\right)\right\rangle_{2,2}=0, \ \  \textrm{for\ all\ } j\in\mathbb{N}.\end{eqnarray*} Moreover, \begin{eqnarray}\label{Conver}\displaystyle\left\langle
\mathcal{L}_{Re,\varphi_c}\displaystyle\left(\begin{array}{c}
                                                                                                       \psi_{1,j} \\
                                                                                                       \psi_{2,j}
                                                                                                     \end{array}
\right),\displaystyle\left(\begin{array}{c}
                                                                                                       \psi_{1,j} \\
                                                                                                       \psi_{2,j}
                                                                                                     \end{array}
\right)\right\rangle_{2,2}\stackrel{j\rightarrow\infty}{\longrightarrow}\gamma.\end{eqnarray}

\indent Using a similar analysis as determined in Proposition
\ref{p.lkg.4}, we guarantee the existence of 
$\displaystyle\left(\begin{array}{c}
                                         \psi_1 \\
                                         \psi_2
                                       \end{array}
\right)$ such that \begin{eqnarray}\label{lkg.50.5.1}\displaystyle\left\| \displaystyle\left(\begin{array}{c}
                                           \psi_{1} \\
                                           \psi_{2}
                                         \end{array}
\right) \right\|_{2,2}=1,\ \ \displaystyle\left\langle
\displaystyle\left(\begin{array}{c}
                                                                                                       \psi_{1} \\
                                                                                                       \psi_{2}
                                                                                                     \end{array}
\right),\displaystyle\left(\begin{array}{c}
                                                                                                       c\varphi_c \\
                                                                                                       \varphi_c
                                                                                                     \end{array}
\right)\right\rangle_{2,2}=0\end{eqnarray} and \begin{eqnarray}\label{lkg.50.5.2}\displaystyle\left\langle
\mathcal{L}_{Re,\varphi_c}\displaystyle\left(\begin{array}{c}
                                                                                                       \psi_{1} \\
                                                                                                       \psi_{2}
                                                                                                     \end{array}
\right),\displaystyle\left(\begin{array}{c}
                                                                                                       \psi_{1} \\
                                                                                                       \psi_{2}
                                                                                                     \end{array}
\right)\right\rangle_{2,2}=\gamma.
\end{eqnarray} 

 Next, we apply the
method of Lagrange multipliers in order to guarantee the existence
of $(a_{\ast},b)\in\mathbb{R}^2$ such that
\begin{eqnarray}\label{lkg.54}\mathcal{L}_{Re,\varphi_c}\displaystyle\left(\begin{array}{c}
                                                                                                       \psi_{1} \\
                                                                                                       \psi_{2}
                                                                                                     \end{array}
\right)=a_{\ast}\displaystyle\left(\begin{array}{c}
                                                                                                       \psi_{1} \\
                                                                                                       \psi_{2}
                                                                                                     \end{array}
\right)+b\displaystyle\left(\begin{array}{c}
                                                                                                       c\varphi_c \\
                                                                                                       \varphi_c
                                                                                                     \end{array}
\right).\end{eqnarray} \indent Before analyzing identity (\ref{lkg.54}), we
will deduce the existence of $\displaystyle\left(\begin{array}{c}
                      M \\
                      N
                    \end{array}
\right)\in X$ such that
\begin{eqnarray}\label{lkg.54.1}\mathcal{L}_{Re,\varphi_c}\displaystyle\left(\begin{array}{c}
                                               M \\
                                               N
                                             \end{array}
\right)=\displaystyle\left(\begin{array}{c}
                                               c\varphi_c \\
                                               \varphi_c
                                             \end{array}
\right).\end{eqnarray} Indeed, the identity (\ref{lkg.54.1}) is
equivalent to
\begin{eqnarray*}\displaystyle\left(\begin{array}{c}
                                      -M''+(1-p)M-\log(\varphi_c^p)M-cN \\
                                      -cM+N
                                    \end{array}
\right)=\displaystyle\left(\begin{array}{c}
                                      c\varphi_c \\
                                      \varphi_c
                                    \end{array}
\right).
\end{eqnarray*}

Thus, $-cN=-c^2M-c\varphi_c$ and
$$\mathcal{L}_{1,\varphi_c}(M)=-M''+(1-c^2)M-pM-\log(\varphi_c^p)M=2c\varphi_c.$$
Since $\ker(\mathcal{L}_{1,\varphi_c})=\textrm{span}\{\varphi_c'\}$
and $\varphi_c\perp\varphi_c'$, we have that
$M=2c\mathcal{L}_{1,\varphi_c}^{-1}(\varphi_c)$. On the other hand,
by deriving equation (\ref{travKG}) with respect to $c$, one has
\begin{eqnarray*}-\displaystyle\left[\displaystyle\frac{d}{dc}(\varphi_c)\right]''+(1-c^2)\displaystyle\frac{d}{dc}(\varphi_c)-2c\varphi_c-\log(\varphi_c^p)\displaystyle\frac{d}{dc}(\varphi_c)-p\displaystyle\frac{d}{dc}(\varphi_c)=0,\end{eqnarray*}
that is,
\begin{eqnarray*}\mathcal{L}_{1,\varphi_c}\displaystyle\left(\displaystyle\frac{d}{dc}(\varphi_c)\right)=2c\varphi_c\
\ \ \textrm{and}\ \ \
M=2c\mathcal{L}_{1,\varphi_c}^{-1}(\varphi_c)=\displaystyle\frac{d}{dc}(\varphi_c).\end{eqnarray*}
\indent Therefore, we conclude
\begin{eqnarray*}\displaystyle\left(\begin{array}{c}
                                            M \\
                                            N
                                          \end{array}
\right)=\displaystyle\left(\begin{array}{c}
\displaystyle\frac{d}{dc}(\varphi_c)\\
\varphi_c+c\displaystyle\frac{d}{dc}(\varphi_c)
                                          \end{array}
\right).\end{eqnarray*} \indent Moreover, since
$|c|>\displaystyle\frac{\sqrt{p}}{2}$, from (\ref{lkg.33.1}), we
have
\begin{eqnarray}\label{lkg.56}\displaystyle\left\langle\mathcal{L}_{Re,\varphi_c}^{-1}\displaystyle\left(\begin{array}{c}
                                                                                             c\varphi_c \\
                                                                                             \varphi_c
                                                                                           \end{array}
\right),\displaystyle\left(\begin{array}{c}
                                                                                             c\varphi_c \\
                                                                                             \varphi_c
                                                                                           \end{array}
\right)\right\rangle_{2,2}&=&\displaystyle\left\langle\displaystyle\left(\begin{array}{c}
                                                                                             M \\
                                                                                             N
                                                                                           \end{array}
\right),\displaystyle\left(\begin{array}{c}
                                                                                             c\varphi_c \\
                                                                                             \varphi_c
                                                                                           \end{array}
\right)\right\rangle_{2,2}\nonumber\\
\\
&=&\displaystyle\int_0^L\varphi_c^2\
dx+c\displaystyle\frac{d}{dc}\displaystyle\left(\displaystyle\int_0^L\varphi_c^2\
dx\right)=-d''(c)<0.\nonumber\end{eqnarray}

We see that $-p$ is the unique negative eigenvalue of the operator
$\mathcal{L}_{1,\varphi_c}$ which is for the
eigenfunction $\varphi_c$. In fact, from (\ref{travKG}),
\begin{eqnarray*}\mathcal{L}_{1,\varphi_c}(\varphi_c)=-\varphi_c''+(1-c^2)\varphi_c-\log(\varphi_c^p)\varphi_c-p\varphi_c=-p\varphi_c.\end{eqnarray*}
From Proposition \ref{p.lkg.1}, we see that
$$\lambda_0:=\lambda_0(c)=\displaystyle\frac{1+c^2-p-\displaystyle\sqrt{1+2c^2+2p+c^4-2c^2p+p^2}}{2}<0$$
is the unique negative eigenvalue to operator
$\mathcal{L}_{Re,\varphi_c}$ which is for the
eigenfunction
$\displaystyle\left(\begin{array}{c}
                      \varphi_c \\
                      m\varphi_c
                    \end{array}
\right),\ \textrm{ where}\ \
m=-\displaystyle\frac{c}{\lambda_0-1}.$

Thus,
\begin{eqnarray}\label{lkg.57.2}\ \ \displaystyle\left\langle\left(\begin{array}{c}
                      c\varphi_c \\
                      \varphi_c
                    \end{array}
\right),\left(\begin{array}{c}
                      \varphi_c \\
                      m\varphi_c
                    \end{array}
\right)\right\rangle_{2,2}=\displaystyle\frac{(\lambda_0-2)c}{(\lambda_0-1)}\displaystyle\int_0^L\varphi_c^2\
dx\neq0.
\end{eqnarray}

From (\ref{lkg.50.5.1})-(\ref{lkg.54}) we get
\begin{eqnarray*}a_{\ast}=\displaystyle\left\langle\mathcal{L}_{Re,\varphi_c}\displaystyle\left(\begin{array}{c}
                                                                                             \psi_1 \\
                                                                                             \psi_2
                                                                                           \end{array}
\right),\displaystyle\left(\begin{array}{c}
                                                                                             \psi_1 \\
                                                                                             \psi_2 \end{array}
\right)\right\rangle_{2,2}=\gamma.\end{eqnarray*}

The spectral properties related to the operator
$\mathcal{L}_{Re,\varphi_c}$ allow us to conclude
$\lambda_0\leq\gamma$. Next, suppose that $\lambda_0=\gamma$. By
taking the inner product of (\ref{lkg.54}) with the function $\left(
                      \begin{array}{c}
                        \varphi_c \\
                        m\varphi_c \\
                      \end{array}
                    \right)
$, we have
\begin{eqnarray*}\lambda_0\displaystyle\left\langle\displaystyle\left(
\begin{array}{c}
  \psi_1 \\
  \psi_2
\end{array}\right)
,\displaystyle\left(
\begin{array}{c}\varphi_c \\
  m\varphi_c\end{array}\right)
\right\rangle_{2,2}&=&\displaystyle\left\langle\mathcal{L}_{Re,\varphi_c}\displaystyle\left(
\begin{array}{c}
  \psi_1 \\
  \psi_2
\end{array}\right)
,\displaystyle\left(
\begin{array}{c}\varphi_c \\
  m\varphi_c\end{array}\right)
\right\rangle_{2,2}\nonumber\\
\\
&=&a_{\ast}\displaystyle\left\langle\displaystyle\left(
\begin{array}{c}
  \psi_1 \\
  \psi_2
\end{array}\right)
,\displaystyle\left(
\begin{array}{c}\varphi_c \\
  m\varphi_c\end{array}\right)
\right\rangle_{2,2}+b\displaystyle\left\langle\displaystyle\left(
\begin{array}{c}
  c\varphi_c \\
  \varphi_c
\end{array}\right)
,\displaystyle\left(
\begin{array}{c}\varphi_c \\
  m\varphi_c\end{array}\right)
\right\rangle_{2,2},
\end{eqnarray*} that is, \begin{eqnarray*}0=b\displaystyle\left\langle\displaystyle\left(
\begin{array}{c}
  c\varphi_c \\
  \varphi_c
\end{array}\right)
,\displaystyle\left(
\begin{array}{c}\varphi_c \\
  m\varphi_c\end{array}\right)\right\rangle_{2,2}.\end{eqnarray*}
\indent Thus, from (\ref{lkg.57.2}), we get $b=0$. So,
\begin{eqnarray*}\mathcal{L}_{Re,\varphi_c}\displaystyle\left(\begin{array}{c}
                                                               \psi_1 \\
                                                               \psi_2
                                                             \end{array}
\right)=\lambda_0\displaystyle\left(\begin{array}{c}
                                                               \psi_1 \\
                                                               \psi_2
                                                             \end{array}
\right)\ \ \textrm{and}\ \ \displaystyle\left(\begin{array}{c}
                                                               \psi_1 \\
                                                               \psi_2
                                                             \end{array}
\right)=c_2\displaystyle\left(\begin{array}{c}
                                                               \varphi_c \\
                                                               m\varphi_c
                                                             \end{array}
\right),
 \end{eqnarray*} where $c_2\neq 0$ is a constant. This last fact
contradicts (\ref{lkg.50.5.1}) and therefore $a_{\ast}\neq \lambda_0$.

Now, we suppose that $\gamma=a_{\ast}\in(\lambda_0,0)$. From
(\ref{lkg.54}) and spectral properties concerning to the operator
$\mathcal{L}_{Re,\varphi_c}$, we obtain
\begin{eqnarray}\label{lkg.58}\displaystyle\left(\begin{array}{c}
                                                   \psi_1 \\
                                                   \psi_2
                                                 \end{array}
\right)=b(\mathcal{L}_{Re,\varphi_c}-a_{\ast}I)^{-1}\displaystyle\left(\begin{array}{c}
                                                   c\varphi_c \\
                                                   \varphi_c
                                                 \end{array}
\right),\end{eqnarray} that is, $b\neq 0$. Consider $a\in
(\lambda_0,0)$ and the auxiliary function $G$,
\begin{eqnarray*}G(a)=\displaystyle\left\langle(\mathcal{L}_{Re,\varphi_c}-aI)^{-1}\displaystyle\left(\begin{array}{c}
                                                  c\varphi_c \\
                                                   \varphi_c
                                                  \end{array}
\right),\displaystyle\left(\begin{array}{c}
                                                c\varphi_c \\
                                                   \varphi_c
                                                 \end{array}
\right)\right\rangle_{2,2}.\end{eqnarray*} From (\ref{lkg.50.5.1}) and
(\ref{lkg.58}), one has $G(a_{\ast})=0$. Since
$\mathcal{L}_{Re,\varphi_c}$ is a self-adjoint operator, we get for
all $a\in (\lambda_0,0)$,
\begin{eqnarray*}G'(a)=\displaystyle\left\|(\mathcal{L}_{Re,\varphi_c}-aI)^{-1}\displaystyle\left(\begin{array}{c}
                                                                                                          c\varphi_c \\
                                                                                                          \varphi_c
                                                                                                        \end{array}
\right)\right\|^2_{2,2}>0.\end{eqnarray*}

Moreover, from (\ref{lkg.56}), we have
\begin{eqnarray*}G(0)=\displaystyle\left\langle\mathcal{L}_{Re,\varphi_c}^{-1}\displaystyle\left(\begin{array}{c}
                                                  c\varphi_c \\
                                                   \varphi_c
                                                  \end{array}
\right),\displaystyle\left(\begin{array}{c}
                                                 c\varphi_c \\
                                                   \varphi_c
                                                 \end{array}
\right)\right\rangle_{2,2}< 0,\end{eqnarray*} that is, $G(a)\neq 0$
for all $a\in(\lambda_0,0)$. Hence, $a_{\ast}\notin(\lambda_0,0)$
and therefore, $\gamma=a_{\ast}=0$. From (\ref{tildegamma}), the proof of ($i$) is now
completed.

($ii$). From item ($i$), we infer that $\kappa\geq 0$. Let us
suppose that $\kappa=0$. We can repeat the same argument as in
Proposition \ref{p.lkg.4} to deduce
\begin{eqnarray}\label{lkg.59}\ \ \ \ \ \ \displaystyle\left\langle\mathcal{L}_{Re,\varphi_c}\displaystyle\left(\begin{array}{c}
                                             \psi_1 \\
                                             \psi_2
                                           \end{array}
\right),\displaystyle\left(\begin{array}{c}
                                             \psi_1 \\
                                             \psi_2
                                           \end{array}
\right)\right\rangle_{2,2}=0,\ \textrm{for\ some}\ \left(
                                                     \begin{array}{c}
                                                       \psi_1 \\
                                                       \psi_2 \\
                                                     \end{array}
                                                   \right)\in X,
\end{eqnarray}
with
\begin{eqnarray}\label{lkg.60}
\displaystyle\left\|\displaystyle\left(\begin{array}{c}
                                                                                                       \psi_1 \\
                                                                                                       \psi_2
                                                                                                     \end{array}
\right)\right\|_{2,2}=1,\ \displaystyle\left\langle
\displaystyle\left(\begin{array}{c}
                                                                                                       \psi_1 \\
                                                                                                       \psi_2
                                                                                                     \end{array}
\right),\displaystyle\left(\begin{array}{c}
                                                                                                       c\varphi_c \\
                                                                                                       \varphi_c
                                                                                                     \end{array}
\right)\right\rangle_{2,2}=0,\end{eqnarray} and
\begin{eqnarray}\label{lkg.61} \displaystyle\left\langle
\displaystyle\left(\begin{array}{c}
                                                                                                       \psi_1 \\
                                                                                                       \psi_2
                                                                                                     \end{array}
\right),\displaystyle\left(\begin{array}{c}
                                                                                                       \log(\varphi_c^p)\varphi_c'+p\varphi_c' \\
                                                                                                       c\varphi_c'
                                                                                                     \end{array}
\right)\right\rangle_{2,2}=0.\end{eqnarray} \indent Next, we can use
the Lagrange multiplier theory to guarantee the existence of
$(b_1,b_2,b_3)\in\mathbb{R}^3$ such that
\begin{eqnarray}\label{lkg.62}\mathcal{L}_{Re,\varphi_c}\displaystyle\left(\begin{array}{c}
                                             \psi_1 \\
                                             \psi_2
                                           \end{array}
\right)=b_1\displaystyle\left(\begin{array}{c}
                                             \psi_1 \\
                                             \psi_2
                                           \end{array}
\right)+b_2\displaystyle\left(\begin{array}{c}
                                                                                                       c\varphi_c \\
                                                                                                       \varphi_c
                                                                                                     \end{array}
\right)+b_3\displaystyle\left(\begin{array}{c}
                                                                                                       \log(\varphi_c^p)\varphi_c'+p\varphi_c' \\
                                                                                                       c\varphi_c'
                                                                                                     \end{array}
\right).\end{eqnarray} \indent Taking the inner product of
$\mathcal{L}_{Re,\varphi_c}\left(
                                                              \begin{array}{c}
                                                                \psi_1 \\
                                                                \psi_2 \\
                                                              \end{array}
                                                            \right)
$ in (\ref{lkg.62}) with $\displaystyle\left(\begin{array}{c}
                             \psi_1 \\
                             \psi_2
                           \end{array}
\right)\in X$, we have from (\ref{lkg.59}), (\ref{lkg.60}) and
(\ref{lkg.61}) that $b_1=0$. On the one hand, since
$\mathcal{L}_{Re,\varphi_c}$ is self-adjoint, we get
\begin{eqnarray*}0=\displaystyle\left\langle\mathcal{L}_{Re,\varphi_c}
\displaystyle\left(\begin{array}{c}
                                                                     \psi_1 \\
                                                                     \psi_2
                                                                   \end{array}
\right),\displaystyle\left(\begin{array}{c}
                                                                     \varphi_c' \\
                                                                     c\varphi_c'
                                                                   \end{array}
\right)\right\rangle_{2,2}=b_3\displaystyle\left\langle
\displaystyle\left(\begin{array}{c}
                                                                                                      \log(\varphi_c^p)\varphi_c'+p\varphi_c' \\
                                                                                                       c\varphi_c'
                                                                                                     \end{array}
\right),\displaystyle\left(\begin{array}{c}
                                                                                                       \varphi_c' \\
                                                                                                       c\varphi_c'
                                                                                                     \end{array}
\right)\right\rangle_{2,2}.
\end{eqnarray*} We also can see that
\begin{eqnarray*}&&\displaystyle\left\langle \displaystyle\left(\begin{array}{c}
                                                                                                      \log(\varphi_c^p)\varphi_c'+p\varphi_c' \\
                                                                                                       c\varphi_c'
                                                                                                     \end{array}
\right),\displaystyle\left(\begin{array}{c}
                                                                                                       \varphi_c' \\
                                                                                                       c\varphi_c'
                                                                                                     \end{array}
\right)\right\rangle_{2,2}=\displaystyle\int_0^L\varphi_c'\displaystyle\left(\log(\varphi_c^p)\varphi_c'+p\varphi_c'+c^2\varphi_c'\right)\
dx\nonumber\\
\\
&&=\displaystyle\int_0^L\varphi_c'\displaystyle\left(\varphi_c'-\varphi_c'''\right)\
dx=\displaystyle\int_0^L(\varphi_c')^2+(\varphi_c'')^2\
dx>0.\nonumber\end{eqnarray*} Thus $b_3=0$ and, therefore,
\begin{eqnarray}\label{lkg.63}\mathcal{L}_{Re,\varphi_c}\displaystyle\left(\begin{array}{c}
                                             \psi_1 \\
                                             \psi_2
                                           \end{array}
\right)=b_2\displaystyle\left(\begin{array}{c}
                                                                                                       c\varphi_c \\
                                                                                                       \varphi_c
                                                                                                     \end{array}
\right).\end{eqnarray} \indent From (\ref{lkg.54.1}), we deduce
\begin{eqnarray}\label{lkg.63.1}\mathcal{L}_{Re,\varphi_c}\displaystyle\left(\begin{array}{c}
                                                               \displaystyle\frac{d}{dc}(\varphi_c) \\
                                                               \varphi_c+ c\displaystyle\frac{d}{dc}(\varphi_c)
                                                             \end{array}
\right)=\displaystyle\left(\begin{array}{c}
                                                               c\varphi_c \\
                                                               \varphi_c
                                                             \end{array}
\right).\end{eqnarray} If we combine (\ref{lkg.63}) and
(\ref{lkg.63.1}), there exists $b_4\in\mathbb{R}$ such that
\begin{eqnarray*}\displaystyle\left(\begin{array}{c}
                                             \psi_1 \\
                                             \psi_2
                                           \end{array}
\right)-b_2\displaystyle\left(\begin{array}{c}
                                                               \displaystyle\frac{d}{dc}(\varphi_c) \\
                                                               \varphi_c+ c\displaystyle\frac{d}{dc}(\varphi_c)
                                                             \end{array}
\right)=b_4\displaystyle\left(\begin{array}{c}
                                                               \varphi_c' \\
                                                               c\varphi_c'
                                                             \end{array}
\right).\end{eqnarray*} This last identity implies that
\begin{eqnarray}\label{lkg.64}\begin{array}{l}\displaystyle\left\langle\left(\begin{array}{c}
                                                             \psi_1 \\
                                                             \psi_2 \\
                                                           \end{array}
                                                         \right)
,\left(
   \begin{array}{c}
     c\varphi_c \\
     \varphi_c \\
   \end{array}
 \right)
\right\rangle_{2,2}-b_2\displaystyle\left\langle\left(\begin{array}{c}
                                                             M \\
                                                             N \\
                                                           \end{array}
                                                         \right)
,\left(
   \begin{array}{c}
     c\varphi_c \\
     \varphi_c \\
   \end{array}
 \right)
\right\rangle_{2,2}\\
\\
=b_4\displaystyle\left\langle\left(
                                                           \begin{array}{c}
                                                             \varphi_c' \\
                                                             c\varphi_c' \\
                                                           \end{array}
                                                         \right)
,\left(
   \begin{array}{c}
     c\varphi_c \\
     \varphi_c \\
   \end{array}
 \right)
\right\rangle_{2,2}=0.\end{array}
\end{eqnarray}

Identities (\ref{lkg.56}), (\ref{lkg.60}) and (\ref{lkg.64}) allow
us to conclude $b_2=0$. So,
\begin{eqnarray*}\displaystyle\left(\begin{array}{c}
                      \psi_1 \\
                      \psi_2
                    \end{array}
\right)=b_4\displaystyle\left(\begin{array}{c}
                      \varphi_c' \\
                      c\varphi_c'
                    \end{array}
\right),\end{eqnarray*} where $b_4\neq 0$. Finally,
\begin{eqnarray*}\displaystyle\left\langle\displaystyle\left(\begin{array}{c}
                      \psi_1 \\
                      \psi_2
                    \end{array}
\right),\displaystyle\left(\begin{array}{c}
                      \log(\varphi_c^p)\varphi_c'+p\varphi_c' \\
                      c\varphi_c'
                    \end{array}
\right)\right\rangle_{2,2}\neq 0.\end{eqnarray*} So, we get a contradiction with
(\ref{lkg.61}).
\end{proof}

\indent Next, we are going to use Propositions \ref{p.lkg.4} and \ref{p.lkg.5} in order to prove Theorem $\ref{t.lkg.2.1.1}$.\\

\noindent \textit{Proof of Theorem $\ref{t.lkg.2.1.1}$.} First, let $u$ be a weak solution related to the problem
(\ref{gl.1}) with initial data
$(u(\cdot,0),u_t(\cdot,0))=(u_0,u_1)\in X.$ Define
\begin{eqnarray*}\mathcal{O}_{\varphi_c}:=\displaystyle\left\{e^{i\theta}(\varphi_c(\cdot+y),ic\varphi_c(\cdot+y));\ (y,\theta)\in\mathbb{R}\times
[0,2\pi]\right\}\end{eqnarray*} the orbit generated by $\varphi_c$
and $v:=u_t$. For $y\in [0,L]$, $\theta\in [0,2\pi]$ and $t\geq 0$,
consider the function $\Omega_t$ as
\begin{eqnarray*}\Omega_t(y,\theta):=\|u_x(\cdot+y,t)e^{i\theta}-\varphi_c'\|_{L^2_{per}}^2+(1-c^2)\|u(\cdot+y,t)e^{i\theta}-\varphi_c\|_{L^2_{per}}^2+\|v(\cdot+y,t)e^{i\theta}-ic\varphi_c\|_{L^2_{per}}^2.\end{eqnarray*}
Since $c^2\in \left(\frac{p}{4},1\right)$ and $p=1,2,3$ for
$t\geq 0$ fixed, one see that the square root of $\Omega_t(y,\theta)$ \textit{ defines an
equivalent norm} in $X$. We see that for all $t\geq 0$, function $\Omega_t$ is
continuous on the compact set $[0,L]\times [0,2\pi]$. There exists $(y,\theta)=(y(t),\theta(t))$ such that
\begin{eqnarray}\label{lkg.70.1}\ \ \ \ \Omega_t(y(t),\theta(t))=\displaystyle\inf_{(y,\theta)\in[0,L]\times
[0,2\pi]}\Omega_t(y,\theta)=\displaystyle\left[\rho_c(\vec{u}(\cdot,t),\mathcal{O}_{\varphi_c})\right]^2,\end{eqnarray}
where $\rho_c(\vec{u}(\cdot,t),\mathcal{O}_{\varphi_c})$ is the
``distance'' between function $\vec{u}$ and the orbit
$\mathcal{O}_{\varphi_c}$ generated by $\varphi_c$. \indent
Moreover, the map
$$t\mapsto \displaystyle\inf_{(y,\theta)\in
[0,L]\times [0,2\pi]}\Omega_{t}(y,\theta)$$ is continuous (see
\cite[Chapter 4, Lemma 2]{bona}).

Next, we consider the perturbation of the periodic wave
$(\varphi_c,ic\varphi_c)$. Suppose that
\begin{eqnarray}\label{lkg.71} u(x+y,t)e^{i\theta}:=\varphi_c(x)+w(x,t)\ \ \textrm{with} \
w:=A+iB\end{eqnarray} and
\begin{eqnarray}\label{lkg.72} v(x+y,t)e^{i\theta}:=ic\varphi_c(x)+z(x,t)\ \ \textrm{with} \ z:=C+iD,\end{eqnarray}
where $t\geq 0$, $x\in \mathbb{R}$ and $y=y(t)$ and
$\theta=\theta(t)$ are determined by (\ref{lkg.70.1}).

Denote the vector $$\vec{w}=\vec{w}_c=(w,z)=(\textrm{Re}\ w,
\textrm{Im}\ z,\textrm{Im}\ w,\textrm{Re}\ z)=(A,D,B,C),$$ where
$(A,D,B,C)\in\mathbb{R}^4$. So, using the property of minimum
$(y(t),\theta(t))$, we obtain from (\ref{lkg.71}) and (\ref{lkg.72})
that $A$, $B$, $C$ and $D$ must satisfy the compatibility relations
\begin{eqnarray}\label{lkg.73}\displaystyle\left\langle\displaystyle\left(\begin{array}{c}
                                                                            A(\cdot,t) \\
                                                                            D(\cdot,t)
                                                                          \end{array}
\right),\displaystyle\left(\begin{array}{c}
                                                                            \log(\varphi_c^p)\varphi_c'+p\varphi_c' \\
                                                                            c\varphi_c'
                                                                          \end{array}
\right)\right\rangle_{2,2}=0\end{eqnarray} and
\begin{eqnarray}\label{lkg.74}\displaystyle\left\langle\displaystyle\left(\begin{array}{c}
                                                                            B(\cdot,t) \\
                                                                            C(\cdot,t)
                                                                          \end{array}
\right),\displaystyle\left(\begin{array}{c}
                                                                            \log(\varphi_c^p)\varphi_c \\
                                                                            -c\varphi_c
                                                                          \end{array}
\right)\right\rangle_{2,2}=0,\end{eqnarray} for all $t\geq 0$. Next,
we use the fact that $\mathcal{E}$ and $\mathcal{F}$ defined in
(\ref{gl.108.1}) and (\ref{gl.109.1}) are invariant by translations
and rotations to get
\begin{eqnarray*}\Delta\mathcal{G}=\mathcal{G}(u_0,u_1)-\mathcal{G}(\varphi_c,ic\varphi_c)=\mathcal{G}(w(\cdot,t)+\varphi_c,z(\cdot,t)+ic\varphi_c)-\mathcal{G}(\varphi_c,ic\varphi_c),
\end{eqnarray*} for all $t\geq 0$. Since
$\mathcal{G}'(\varphi_c,ic\varphi_c)=\mathcal{G}'(\varphi_c,c\varphi_c,0,0)=\vec{0}$,
we can write
\begin{eqnarray*}\Delta\mathcal{G}&=&\displaystyle\sum_{n=2}^{4}\displaystyle\left[\displaystyle\frac{\mathcal{G}^{(n)}(\varphi_c,ic\varphi_c)}{n!}\right]\cdot[\vec{w}(\cdot,t)]^n+{O}(w(\cdot,t))\end{eqnarray*}
and for all $t\geq 0$, we deduce
\begin{eqnarray*}\label{lkg.66}\Delta\mathcal{G}&=&\displaystyle\frac{1}{2}\displaystyle\left\langle\mathcal{L}_{Re,\varphi_c}\left(
                                                                                                               \begin{array}{c}
                                                                                                                 A(\cdot,t) \\
                                                                                                                 D(\cdot,t) \\
                                                                                                               \end{array}
                                                                                                             \right)
,\left(\begin{array}{c} A(\cdot,t) \\
                                                                                                                 D(\cdot,t) \\
                                                                                                               \end{array}
                                                                                                             \right)\right\rangle_{2,2}\nonumber\\
\nonumber\\
&+&\displaystyle\frac{1}{2}\displaystyle\left\langle\mathcal{L}_{Im,\varphi_c}\left(
                                                                                                               \begin{array}{c}
                                                                                                                 B(\cdot,t) \\
                                                                                                                 C(\cdot,t) \\
                                                                                                               \end{array}
                                                                                                             \right)
,\left(
                                                                                                               \begin{array}{c}
                                                                                                                 B(\cdot,t) \\
                                                                                                                 C(\cdot,t) \\
                                                                                                               \end{array}
                                                                                                             \right)\right\rangle_{2,2}\nonumber\\
                                                                                                             \\
&+&\displaystyle\frac{p}{6}\displaystyle\int_0^L\displaystyle\left[\displaystyle\frac{A(\cdot,t)B(\cdot,t)^2-A(\cdot,t)^3}{\varphi_c}\right]\
dx\nonumber\\
\nonumber\\
&+&\displaystyle\frac{p}{24}\displaystyle\int_0^L\displaystyle\left[\displaystyle\frac{A(\cdot,t)^4-2A(\cdot,t)^2B(\cdot,t)^2-3B(\cdot,t)^4}{\varphi_c^2}\right]\
dx+{O}({w}(\cdot,t)),\nonumber
\end{eqnarray*} where $|{O}(w(\cdot),t)|\leq {O}(\|\vec{w}(\cdot,t)\|_{X\times
X}^5)$. There exist positive constants $\beta_3$ and $\beta_4$
such that
\begin{eqnarray}\label{lkg.73.2}\Delta\mathcal{G}(t)=\Delta\mathcal{G}&\geq&\displaystyle\frac{1}{2}\displaystyle\left\langle\mathcal{L}_{Re,\varphi_c}\left(
                                                                                                               \begin{array}{c}
                                                                                                                 A(\cdot,t) \\
                                                                                                                 D(\cdot,t) \\
                                                                                                               \end{array}
                                                                                                             \right)
,\left(
                                                                                                               \begin{array}{c}
                                                                                                                 A(\cdot,t) \\
                                                                                                                 D(\cdot,t) \\
                                                                                                               \end{array}
                                                                                                             \right)\right\rangle_{2,2}\nonumber\\
                                                                                                             \nonumber\\
                                                                                                             &+&\displaystyle\frac{1}{2}\displaystyle\left\langle\mathcal{L}_{Im,\varphi_c}\left(
                                                                                                               \begin{array}{c}
                                                                                                                 B(\cdot,t) \\
                                                                                                                 C(\cdot,t) \\
                                                                                                               \end{array}
                                                                                                             \right)
,\left(
                                                                                                               \begin{array}{c}
                                                                                                                 B(\cdot,t) \\
                                                                                                                 C(\cdot,t) \\
                                                                                                               \end{array}
                                                                                                             \right)\right\rangle_{2,2}\\
\nonumber\\
&-&\beta_3\|\vec{w}(\cdot,t)\|^3_{X\times
X}-\beta_4\|\vec{w}(\cdot,t)\|^4_{X\times X}
-{O}(\|\vec{w}(\cdot,t)\|_{X\times X}^5),\nonumber\end{eqnarray} for
all $t\geq 0$.

Initially, let us suppose that
$\mathcal{F}(\varphi_c,ic\varphi_c)=\mathcal{F}(u_0,u_1)=\mathcal{F}(u(\cdot,t),v(\cdot,t)).$
So,
\begin{eqnarray}\label{lkg.73.1.1}\displaystyle\left\langle\left(
                                            \begin{array}{c}
                                              A(\cdot,t) \\
                                              D(\cdot,t) \\
                                            \end{array}
                                          \right)
,\left(\begin{array}{c}
   c\varphi_c \\
   \varphi_c
 \end{array}\right)
\right\rangle_{2,2}=\displaystyle\int_0^L\displaystyle\left[B(\cdot,t)C(\cdot,t)-A(\cdot,t)D(\cdot,t)\right]\
dx.\end{eqnarray}

Without loss of generality, we consider
\begin{eqnarray}\label{lkg.73.1.2}\displaystyle\left\|\displaystyle\left(
                                       \begin{array}{c}
                                         c\varphi_c \\
                                         \varphi_c \\
                                       \end{array}
                                     \right)
\right\|_{2,2}=1.\end{eqnarray}

Let us define the auxiliary functions,
\begin{eqnarray*}P_{\|}(\cdot,t)&=&\displaystyle\left\langle \left(
                                                    \begin{array}{c}
                                                      A(\cdot,t) \\
                                                      D(\cdot,t) \\
                                                    \end{array}
                                                  \right)
,\left(\begin{array}{c}
     c\varphi_c \\
     \varphi_c \\
   \end{array}
 \right)
\right\rangle_{2,2}\left(\begin{array}{c}
                       c\varphi_c \\
                       \varphi_c \\
                     \end{array}
                   \right)\ \ \textrm{and} \ \ P_{\perp}(\cdot,t)=\left(
                                                          \begin{array}{c}
                                                            A(\cdot,t) \\
                                                            D(\cdot,t) \\
                                                          \end{array}
                                                        \right)-P_{\|}(\cdot,t).\end{eqnarray*}

In view of the identity (\ref{lkg.73.1.2}), one has $P_{\perp}(\cdot,t)\perp \left(\begin{array}{c}
                                                                   c\varphi_c \\
                                                                   \varphi_c \\
                                                                 \end{array}
                                                               \right).$
Moreover, we can use the
compatibility condition (\ref{lkg.73}) to deduce
$$P_{\perp}(\cdot,t)\perp \left(\begin{array}{c}
                                                                   \log(\varphi_c^p)\varphi_c'+p\varphi_c' \\
                                                                   c\varphi_c' \\
                                                                 \end{array}
                                                               \right).$$
 So, Proposition
\ref{p.lkg.5} can be used to obtain the existence of a constant $\kappa>0$ such that
\begin{eqnarray}\label{lkg.73.1}\displaystyle\left\langle\mathcal{L}_{Re,\varphi_c}P_{\perp}(\cdot,t),P_{\perp}(\cdot,t)\right\rangle_{2,2}\geq \kappa\|P_{\perp}(\cdot,t)\|_{2,2}^2.\end{eqnarray}
\indent Identity (\ref{lkg.73.1.1}) combined with basic inequalities allows
us to conclude
\begin{eqnarray}\label{lkg.76}\|P_{\perp}(\cdot,t)\|_{2,2}^2&=&\displaystyle\left\|\left(
                                                                    \begin{array}{c}
                                                                      A(\cdot,t) \\
                                                                      D(\cdot,t) \\
                                                                    \end{array}
                                                                  \right)
\right\|^2_{2,2}-\displaystyle\left[\displaystyle\left\langle \left(
                                                                \begin{array}{c}
                                                                  A(\cdot,t) \\
                                                                  D(\cdot,t) \\
                                                                \end{array}
                                                              \right)
,\left(
   \begin{array}{c}
     c\varphi_c \\
     \varphi_c \\
   \end{array}
 \right)
\right\rangle_{2,2}\right]^2\nonumber\\
\\
&\geq&\displaystyle\left\|\left(
                                                                    \begin{array}{c}
                                                                      A(\cdot,t) \\
                                                                      D(\cdot,t) \\
                                                                    \end{array}
                                                                  \right)
\right\|^2_{2,2}-\displaystyle\frac{1}{4}\displaystyle\left[\displaystyle\left\|\left(
                                                                                                                          \begin{array}{c}
                                                                                                                            A(\cdot,t) \\
                                                                                                                            D(\cdot,t) \\
                                                                                                                          \end{array}
                                                                                                                        \right)
                             \right\|_{2,2}^2+\displaystyle\left\|\left(\begin{array}{c}
                                                                                                                            B(\cdot,t) \\
                                                                                                                            C(\cdot,t) \\
                                                                                                                          \end{array}
                                                                                                                        \right)\right\|_{2,2}^2\right]^2.\nonumber\end{eqnarray}

On the other hand,
\begin{eqnarray}\label{lkg.75}\begin{array}{l}\displaystyle\left\langle\mathcal{L}_{Re,\varphi_c}\left(
                                                                     \begin{array}{c}
                                                                       A(\cdot,t) \\
                                                                       D(\cdot,t) \\
                                                                     \end{array}
                                                                   \right)
,\left(\begin{array}{c}
                                                                       A(\cdot,t) \\
                                                                       D(\cdot,t) \\
                                                                     \end{array}
                                                                   \right)\right\rangle_{2,2}=\displaystyle\left\langle\mathcal{L}_{Re,\varphi_c}P_{\perp}(\cdot,t),P_{\perp}(\cdot,t)\right\rangle_{2,2}\\
\\
+2\displaystyle\left\langle\mathcal{L}_{Re,\varphi_c}P_{\perp}(\cdot,t),P_{\|}(\cdot,t)\right\rangle_{2,2}+\displaystyle\left\langle\mathcal{L}_{Re,\varphi_c}P_{\|}(\cdot,t),P_{\|}(\cdot,t)\right\rangle_{2,2},\end{array}
\end{eqnarray} for all $t\geq 0$. Moreover,
\begin{eqnarray}\label{lkg.77}\displaystyle\left|2\displaystyle\left\langle \mathcal{L}_{Re,\varphi_c}P_{\perp}(\cdot,t),P_{\|}(\cdot,t)\right\rangle_{2,2}\right|\leq\beta_5\|\vec{w}(\cdot,t)\|^3_{X\times X}
\end{eqnarray} and
\begin{eqnarray}\label{lkg.78}\displaystyle\left|\displaystyle\left\langle \mathcal{L}_{Re,\varphi_c}P_{\|}(\cdot,t),P_{\|}(\cdot,t)\right\rangle_{2,2}\right|\leq\beta_6\|\vec{w}(\cdot,t)\|^4_{X\times X},
\end{eqnarray} where $\beta_5$ and $\beta_6$ are positive constants.
Gathering the results in (\ref{lkg.73.1}), (\ref{lkg.76}), (\ref{lkg.75}),
(\ref{lkg.77}) and (\ref{lkg.78}) we see that
\begin{eqnarray}\label{lkg.78.1}\displaystyle\left\langle
\mathcal{L}_{Re,\varphi_c}\left(
                                             \begin{array}{c}
                                               A(\cdot,t) \\
                                               D(\cdot,t) \\
                                             \end{array}
                                           \right)
,\left(\begin{array}{c}
                                               A(\cdot,t) \\
                                               D(\cdot,t) \\
                                             \end{array}
                                           \right)\right\rangle_{2,2}&\geq&
                                           \kappa\displaystyle\left\|\left(
                                             \begin{array}{c}
                                               A(\cdot,t) \\
                                               D(\cdot,t) \\
                                             \end{array}
                                           \right)\right\|^2_{2,2}\nonumber\\
\\
&-&\beta_5\|\vec{w}(\cdot,t)\|^3_{X\times
X}-(\kappa+\beta_6)\|\vec{w}(\cdot,t)\|^4_{X\times X}.\nonumber
\end{eqnarray}

Using the definition of the operator $\mathcal{L}_{Re,\varphi_c}$, we
get the existence of a constant $\beta_7>0$ such that
\begin{eqnarray}\label{lkg.78.2}\displaystyle\left\langle
\mathcal{L}_{Re,\varphi_c}\left(\begin{array}{c}
                                               A(\cdot,t) \\
                                               D(\cdot,t) \\
                                             \end{array}
                                           \right)
,\left(\begin{array}{c}
                                               A(\cdot,t) \\
                                               D(\cdot,t) \\
                                             \end{array}
                                           \right)\right\rangle_{2,2}&\geq&
\displaystyle\int_0^L|A_x(\cdot,t)|^2\
dx-\beta_7\displaystyle\left\|\left(
                                \begin{array}{c}
                                  A(\cdot,t) \\
                                  D(\cdot,t) \\
                                \end{array}
                              \right)
\right\|_{2,2}^2.
\end{eqnarray}
\indent Inequalities (\ref{lkg.78.1}) and (\ref{lkg.78.2}) enable us to
guarantee the existence of positive constants $\beta_8$, $\beta_9$
and $\beta_{10}$ such that
\begin{eqnarray}\label{lkg.81}\displaystyle\left\langle
\mathcal{L}_{Re,\varphi_c}\left(\begin{array}{c}
                                               A(\cdot,t) \\
                                               D(\cdot,t) \\
                                             \end{array}
                                           \right)
,\left(
                                             \begin{array}{c}
                                               A(\cdot,t) \\
                                               D(\cdot,t) \\
                                             \end{array}
                                           \right)\right\rangle_{2,2}&\geq&\beta_8\displaystyle\left\|\left(
                                                                                                         \begin{array}{c}
                                                                                                           A(\cdot,t) \\
                                                                                                           D(\cdot,t) \\
                                                                                                         \end{array}
                                                                                                       \right)\right\|^2_{X}\nonumber\\
\\
&-&\beta_9\|\vec{w}(\cdot,t)\|^3_{X\times
X}-\beta_{10}\|\vec{w}(\cdot,t)\|^4_{X\times X}.\nonumber
\end{eqnarray}

In addition, the compatibility conditions in (\ref{lkg.74}) and
Proposition \ref{p.lkg.4} give us the existence of a constant
$\beta>0$ such that
\begin{eqnarray*}\displaystyle\left\langle\mathcal{L}_{Im,\varphi_c}\left(
                                                                                   \begin{array}{c}
                                                                                     B(\cdot,t) \\
                                                                                     C(\cdot,t)
                                                                                   \end{array}
                                                                                 \right)
,\left(\begin{array}{c}B(\cdot,t) \\
                                                                                     C(\cdot,t)
                                                                                   \end{array}
                                                                                 \right)\right\rangle_{2,2}\geq\beta\displaystyle\left\|\left(
                                                                                   \begin{array}{c}
                                                                                     B(\cdot,t) \\
                                                                                     C(\cdot,t)
                                                                                   \end{array}
                                                             \right)\right\|_{2,2}^2.\end{eqnarray*}
Finally, from the definition of the operator
$\mathcal{L}_{Im,\varphi_c}$ and similar arguments as above we obtain
the existence of a constant $\beta_{11}>0$ such that
\begin{eqnarray}\label{lkg.79}\ \ \ \displaystyle\left\langle\mathcal{L}_{Im,\varphi_c}\left(
                                                                                   \begin{array}{c}
                                                                                     B(\cdot,t) \\
                                                                                     C(\cdot,t)
                                                                                   \end{array}
                                                                                 \right)
,\left(
                                                                                   \begin{array}{c}
                                                                                     B(\cdot,t) \\
                                                                                     C(\cdot,t)
                                                                                   \end{array}\right)\right\rangle_{2,2}\geq\beta_{11}\displaystyle\left\|\left(\begin{array}{c}B(\cdot,t) \\
C(\cdot,t) \\
\end{array}\right)
\right\|_{X}^2.\end{eqnarray}

Therefore, by substituting (\ref{lkg.81}) and (\ref{lkg.79}) in
(\ref{lkg.73.2}), we have that $\Delta\mathcal{G}(t)\geq
h_1(\|\vec{w}(\cdot,t)\|_{X\times X})$ for all $t\geq 0$, where
$h_1(x):=\eta_1x^2(1-\eta_2x-\eta_3x^2-{O}(x^3))$ is a smooth
function and $\eta_1,\ \eta_2$ and $\eta_3$ are positive constants.
We see that $h_1(0)=0$ and $h_1(x)>0$ for $x$ small enough. Consider
$\varepsilon>0$. Then, using the property that $\mathcal{E}$ is
continuous on the manifold
$$S=S_c:=\displaystyle\left\{(u_0,u_1)\in X;\
\mathcal{F}(u_0,u_1)=\mathcal{F}(\varphi_c,ic\varphi_c)\right\},$$
there exists $\delta=\delta(\varepsilon,c)>0$  such that if
$(u_0,u_1)\in S$ and
$\|(u_0,u_1)-(\varphi_c,ic\varphi_c)\|_{X}<\delta,$ then for all
$t\geq 0$, $h_1(\|\vec{w}(\cdot,t)\|_{X\times X})\leq
\Delta\mathcal{G}(t)=\Delta\mathcal{G}(0)<h_1(\varepsilon).$

Since $h_1$ is an invertible function over $(0,\varepsilon]$, we
have that
\begin{eqnarray}\label{lkg.82}\|\vec{w}(\cdot,t)\|_{X\times
X}<\varepsilon\ \textrm{ for\ all\ } t\geq 0.\end{eqnarray}
Statement (\ref{lkg.82}) is enough to prove the stability on the
manifold $S$ since
\begin{eqnarray*}\|(u(\cdot+y(t),t)e^{i\theta(t)}-\varphi_c,v(\cdot+y(t),t)e^{i\theta(t)}-ic\varphi_c)\|_{X}<\varepsilon,\end{eqnarray*}
for all $t\geq 0$.

The next step is to prove the general case. Suppose a particular
situation where $c\in\left(\frac{\sqrt{p}}{2},1\right)$ is fixed. By
repeating the same process above, there exist a constant
$\delta_1=\delta_1(c)>0$ and a real function $z_0$ such that
$z_0(0)=0$, where $z_0$ is a strictly increasing for (small)
positive values. Moreover for all $t\geq 0$,
\begin{eqnarray}\label{lkg.83}\mathcal{G}_s(u_0,u_1)-\mathcal{G}_s(\varphi_s,is\varphi_s)\geq
z_0(\|\vec{w}_s(\cdot,t)\|_{X\times X}),\end{eqnarray} provided that
$|s-c|<\delta_1$ and
$\mathcal{F}(u_0,u_1)=\mathcal{F}(\varphi_s,is\varphi_s)$.

For $\epsilon>0$ small enough, the continuity of $\mathcal{E}$
enables us to guarantee the existence of the parameter $\delta_2$,
$0<\delta_2=\delta_2(c,\epsilon)<\displaystyle\frac{\epsilon}{2}$,
such that if $(u_0,u_1),\ (v_0,v_1)\in
B_{\delta_2}(\varphi_c,ic\varphi_c)\subset X,$ then
\begin{eqnarray}\label{lkg.84}|\mathcal{E}(u_0,u_1)-\mathcal{E}(v_0,v_1)|<z_0\displaystyle\left(\displaystyle\frac{\epsilon}{2}\right).\end{eqnarray}

Since the map $s\in \mathbb{R}\mapsto\varphi_s\in
H^2_{per,e}([0,L])$ is smooth, there exists  $\delta_3>0$ such that if
$|s-c|<\delta_3<\delta_1$, then
\begin{eqnarray}\label{lkg.85}\ \ \ \ \frac{\sqrt{p}}{2}<s<1\ \ \textrm{and} \ \ \|(\varphi_s,is\varphi_s)-(\varphi_c,ic\varphi_c)\|_X<\delta_2<\displaystyle\frac{\epsilon}{2}.\end{eqnarray}

Furthermore, since
$-\frac{d}{ds}\left[\mathcal{F}(\varphi_s,is\varphi_s)\right]=d''(s)>0\
\textrm{for\ all}\ s\in \left(\frac{\sqrt{p}}{2},1\right),$ there exists
$\delta_4>0$ such that if
\begin{eqnarray}\label{lkg.87}\|(u_0,u_1)-(\varphi_c,ic\varphi_c)\|_{X}<\delta_4<\displaystyle\frac{\delta_2}{2},\end{eqnarray}
one has $s^{\ast}=s^{\ast}(u_0,u_1)\in\mathbb{R}$, where
$|s^{\ast}-c|<\delta_3$ and $s^{\ast}$ satisfy
$\mathcal{F}(u_0,u_1)=\mathcal{F}(\varphi_{s^{\ast}},is^{\ast}\varphi_{s^{\ast}})$.

Consider $\delta_4>0$ as above and suppose that $(u_0,u_1)$ verifies
(\ref{lkg.87}). Since $|s^{\ast}-c|<\delta_3$, we get condition
(\ref{lkg.85}) applied to $s=s^{\ast}$, that is,
\begin{eqnarray}\label{lkg.87.1}\ \ \ \ \frac{\sqrt{p}}{2}<s^{\ast}<1\ \ \textrm{and} \ \ \|(\varphi_{s^{\ast}},is^{\ast}\varphi_{s^{\ast}})-(\varphi_c,ic\varphi_c)\|_X<\delta_2<\displaystyle\frac{\epsilon}{2}.\end{eqnarray} Thus, from (\ref{lkg.84}), (\ref{lkg.87}) and (\ref{lkg.87.1}), we obtain
\begin{eqnarray}\label{lkg.87.2}|\mathcal{E}(u_0,u_1)-\mathcal{E}(\varphi_{s^{\ast}},is^{\ast}\varphi_{s^{\ast}})|<z_0\displaystyle\left(\displaystyle\frac{\epsilon}{2}\right).\end{eqnarray}
\indent Next, since $|s^{\ast}-c|<\delta_3<\delta_1$ and
$\mathcal{F}(u_0,u_1)=\mathcal{F}(\varphi_{s^{\ast}},is^{\ast}\varphi_{s^{\ast}})$, we deduce from (\ref{lkg.83}) and (\ref{lkg.87.2}) that
\begin{eqnarray*}z_0(\|\vec{w}_{s^{\ast}}(\cdot,t)\|_{X\times
X})\leq\mathcal{G}_{s^{\ast}}(u_0,u_1)-\mathcal{G}_{s^{\ast}}(\varphi_{s^{\ast}},is^{\ast}\varphi_{s^{\ast}})<z_0\displaystyle\left(\displaystyle\frac{\epsilon}{2}\right).\end{eqnarray*}
Hence, the fact that $z_0$ is a locally invertible function gives us
\begin{eqnarray}\label{lkg.91}\|\vec{w}_{s^{\ast}}(\cdot,t)\|_{X\times
X}<\displaystyle\frac{\epsilon}{2}\ \textrm{for\ all} \ t\geq
0.\end{eqnarray} \indent Finally, if one combines (\ref{lkg.87.1}) and (\ref{lkg.91}), \begin{eqnarray*}\|\vec{w}_c(\cdot,t)\|_{X\times X}&\approx&\displaystyle\inf_{y\in [0,L],\ \theta\in [0,2\pi]}\|(u(\cdot+y,t)e^{i\theta},v(\cdot+y,t)e^{i\theta})-(\varphi_c,ic\varphi_c)\|_{X}\nonumber\\
\nonumber\\
&\leq&\|\vec{w}_{s^{\ast}}(\cdot,t)\|_{X\times
X}+\|(\varphi_{s^{\ast}},is^{\ast}\varphi_{s^{\ast}})-(\varphi_c,ic\varphi_c)\|_{X}<\displaystyle\frac{\epsilon}{2}+\displaystyle\frac{\epsilon}{2}=\epsilon\
\textrm{for\ all\ } t\geq 0,\nonumber\end{eqnarray*} which concludes
the proof for the case $c\in\left(\frac{\sqrt{p}}{2},1\right)$. A
similar argument can be used to prove the stability for the case
$c\in\left(-1,-\frac{\sqrt{p}}{2}\right)$.
\begin{flushright}
$\square$
\end{flushright}

\subsection{Orbital Stability of
Standing Waves in the space \textbf{${H^1_{per,e}([0,L])\times
L^2_{per,e}([0,L])}$.}}

We present a brief comment about results of stability/instability of
standing waves for the Logarithmic Klein-Gordon equation (\ref{KG1})
in the Hilbert space $Y_e:=H^1_{per,e}([0,L])\times
L^2_{per,e}([0,L])$ constituted by even periodic functions.

The abstract theory in
\cite{grillakis1}  gives us a general method
to study the orbital stability/instability to standing waves
accordingly with Definition \ref{d.lkg.2.2} for the
abstract Hamiltonian systems of form
\begin{equation}\label{hamiltoGSS}
U_t=J\mathcal{E}'(U(t)),
\end{equation}
where $J$ is a skew-symmetric linear operator and $\mathcal E$ is a
convenient conserved quantity. In our case, if
$U=(u,u_t):=(\mbox{Re}\ u,\mbox{Im}\ u_t,\mbox{Im}\ u,\mbox{Re}\
u)$, $J$ is taken to be
\begin{equation}
\displaystyle J=\left(
                    \begin{array}{cccc}
                      0 & 0 & 0 & 1 \\
                      0 & 0 & -1 & 0 \\
                      0 & 1 & 0 & 0 \\
                      -1 & 0 & 0 & 0 \\
                    \end{array}
                  \right)
\label{matrixJ}
\end{equation}
and $\mathcal{E}$ is the conserved quantity given by
$(\ref{gl.108.1})$, we can follow the arguments determined on
Sections 3 and 4 to obtain:
\begin{itemize}\item{the existence of a weak solution to the problem
(\ref{gl.1}) for all $(u_0,v_0)\in H^1_{per,e}([0,L])\times
L^2_{per,e}([0,L])$, as determined in Section 3.}
\item The existence of a smooth curve of positive solutions $$c\in I\mapsto \varphi_c\in
H^2_{per,e}([0,L])$$ which solves (\ref{travKG}) all of them with the same period
$L>\frac{2\pi}{\sqrt{p}}$.
\item
$\textrm{in}(\displaystyle\left.\mathcal{L}_{\varphi_c}\right|_{X_e})=(1,1)$\
for all $c\in I$.
\item $d''(c)<0$ if $|c|<\frac{\sqrt{p}}{2}$ and $d''(c)>0$ if
$|c|>\frac{\sqrt{p}}{2}$.
\end{itemize} So, by a direct application of the result in \cite{grillakis1} we are in position to enunciate the following result.

\begin{theorem} Consider $c\in I$ satisfying $|c|>\frac{\sqrt{p}}{2}$, $p=1,2,3$.
The periodic solution $\varphi_c(x)$ is orbitally stable in $Y_e$ by
the periodic flow of the equation (\ref{KG1}) accordingly Definition
\ref{d.lkg.2.2}. If $|c|<\frac{\sqrt{p}}{2}$ then
$u(x,t)=e^{ict}\varphi_c(x)$ is orbitally unstable by the periodic
flow of the equation (\ref{KG1}).
\end{theorem}
{\begin{flushright} $\square$ \end{flushright}}

\section*{Acknowledgements}

F. N. is supported by Funda\c{c}\~ao Arauc\'aria and CNPq 304240/2018-4.



\begin{thebibliography}{99}

\bibitem{AN1} J. Angulo and F. Natali, \textit{Instability  of periodic traveling waves for dispersive models}. Differential and Integral Equations, 29 (2016), 837--874.

\bibitem{AN2} J. Angulo and F. Natali, \textit{Stability and instability of periodic
travelling waves solutions for the critical Korteweg-de Vries and
non-linear Schr\"odinger equations}, Physica D, 238 (2009),
603-621.

\bibitem{AN} J. Angulo and F. Natali, \textit{Positivity properties of the Fourier transform and the stability of
periodic travelling-wave solutions}, SIAM J. Math. Anal., 40 (2008),
1123-1151.


\bibitem{gorka1} K. Bartkowski and P. G\'orka, \textit{One-dimensional KleinGordon equation with logarithmic nonlinearities},  J. Phys. A: Math. Theor., 41 (2008), 355201, 11 pp.


\bibitem{birula2} I. Bia{\l}ynicki-Birula and J. Mycielski. \textit{Nonlinear wave
mechanics}, Ann. Phys., 100 (1976), 62-93.

\bibitem{birula1} I. Bia{\l}ynicki-Birula and J. Mycielski, \textit{Wave equations with logarithmic nonlinearities},
Bull. Acad. Polon. Sci., 23 (1975), 461-466.

\bibitem{blanchard} P. H. Blanchard, J. Stubbe and L. V\'azquez, \textit{On the
stability of solitary waves for classical scalar fields}, Annal.
Inst. Henry Poincar\'e Sec. A., 47 (1987), 309-336.

\bibitem{bona} J. L. Bona, \textit{On the Stability Theory of Solitary
Waves.} Proc. R. Soc. Lond. Ser. A, 344 (1975), 363-374.

\bibitem{lions1} F. Boyer and P. Fabrie, \textit{Mathematical Tools for the Study of the Incompressible Navier-Stokes Equations and Related Models}, Applied Mathematical Sciences 183, Springer, 2013.

\bibitem{carles} R. Carles and D. Pelinovsky, \textit{On the orbital stability of Gaussian solitary waves in the log-KdV equation}, Nonlinearity, 27 (2014), 3185-3202.

\bibitem{Cazenave3} T. Cazenave, \textit{Stable solutions of the logarithmic Schr\"odinger equation},
Non. Anal. Theory Meth. Appl., 7 (1983), 1127-1140.

\bibitem{cazenave2} T. Cazenave and P. L. Lions, \textit{Orbital stability of
standing waves for some nonlinear Schr\"odinger equations}, Comm.
Math. Phys., 85 (1982), 549-561.

\bibitem{cazenave} T. Cazenave and A. Haraux. \textit{\'{E}quations d'\'{e}volution avec non
Lin\'{e}arit\'{e} Logarithmique.} Annales de la Facult\'{e} des
Sciences de Toulouse, 2 (1980), 21-51.

\bibitem{CNP} F. Crist\'ofani, F. Natali and A. Pastor, \textit{Orbital stability of periodic traveling-wave solutions for the Log-KdV equation}, J. Differential Equations, 263 (2017), pp. 2630--2660.

\bibitem{gahler} R. Gahler, A. G. Klein and A. Zeilinger, \textit{Neutron optical tests of nonlinear wave
mechanics}, Phys. Rev A., 23 (1981), 1611-1617.

\bibitem{gorka2} P. G\'orka, \textit{Logarithmic Klein-Gordon Equation.}
Acta Physica Polonica B, 40 (2009), 59-66.

\bibitem{grillakis1} M. Grillakis, J. Shatah and W. Strauss, \textit{Stability Theory of Solitary Waves in the Presence of Symmetry
I.} J. Funct. Anal., 74 (1987), 160-197.


\bibitem{iorio1} R. J. Iorio Jr. and V. M. V. Iorio, \textit{Fourier Analysis and
Partial Differential Equations}, Cambridge Stud. in Advan. Math.,
2001.

\bibitem{lilin} P. Li and C. Liu, \textit{A class of fourth-order parabolic equation with logarithmic
nonlinearity}. J. of Ineq. and Appl., (2018) 2018:328.
https://doi.org/10.1186/s13660-018-1920-7.

\bibitem{lions} J. L. Lions, \textit{Quelques M\'ethodes de R\'{e}solution des Probl\`{e}mes aux Limites non
Lin\'{e}aires.} Dunod Gauthier-Villars, Paris, 1969.

\bibitem{magnus} W. Magnus and S. Winkler, \textit{Hill's Equation.}
Interscience, Tracts in Pure and Applied Mathematics, vol. 20,
Wiley, New York, 1966.

\bibitem{natali1} F. Natali and A. Neves, \textit{Orbital Stability of Solitary Waves}. IMA Journal of Applied Mathematics, 79 (2014), 1161-1179

\bibitem{neves} A. Neves, \textit{Floquet's Theorem and Stability of Periodic Solitary
Waves.} J. Dyn. Diff. Equat., 21 (2009), 555-565.

\bibitem{neves1} A. Neves. \textit{Isoinertial Family of Operators and Convergence of KdV Cnoidal Waves to
Solitons}. J. Differ. Equat., 244 (2008), 875-886.

\bibitem{rosen} G. Rosen, \textit{Dilatation covariance and exact
solutions in local relativistic field theories}, Phys. Rev.,
183 (1969), 1186-1188.

\bibitem{shimony} A. Shimony, \textit{Proposed neutron interferomenter test of some nonlinear variants of wave mechanics},
Phys. Rev. A, 20 (1979), 394-396.

\bibitem{shull} C. G. Shull, D. K. Atwood, J. Arthur and M. A. Horne,
\textit{Search for a nonlinear variant of the Schr\"odinger equation
by neutron interferometry}, Phys. Rev. Lett., 44 (1980),
765-768.

\bibitem{weinstein1} M. I. Weinstein, \textit{Modulation Stability of Ground States of Nonlinear Schr\"{o}dinger
Equations}. SIAM J. Math, 16 (1985), 472-490.

\end{thebibliography}
\end{document}